\documentclass[reqno]{amsart} 
\usepackage{amsfonts, amsmath, amsthm, amssymb, latexsym, graphicx}

\def\cal{\mathcal }
\theoremstyle{plain}
\newtheorem{theorem}{Theorem}[section]
\newtheorem{corollary}[theorem]{Corollary}
\newtheorem{lemma}[theorem]{Lemma}
\newtheorem{proposition}[theorem]{Proposition}
\theoremstyle{definition}

\newtheorem{remark}[theorem]{Remark}

\newtheorem{examples}[theorem]{Examples}
\newtheorem{example}[theorem]{Example}
\newtheorem{ODE}[theorem]{The ODE}

\numberwithin{equation}{section}

\title[Limits of Bessel and Dunkl processes  and free
  convolutions]{Limit theorems for Bessel and Dunkl processes of large dimensions and free convolutions}

\author{Michael Voit, Jeannette H.C. Woerner} 
\address{Fakult\"at Mathematik, Technische Universit\"at Dortmund,
          Vogelpothsweg 87,
          D-44221 Dortmund, Germany}
\email{michael.voit@math.tu-dortmund.de, jeannette.woerner@math.tu-dortmund.de}


\begin{document}
\date{\today}

\begin{abstract}
  We study Bessel and Dunkl processes  $(X_{t,k})_{t\ge0}$  on $\mathbb R^N$ with possibly multivariate
  coupling constants $k\ge0$. These processes
  describe interacting particle systems of 
  Calogero-Moser-Sutherland type with $N$ particles.
  For the root systems  $A_{N-1}$ and $B_N$ these Bessel processes
  are
related with $\beta$-Hermite and $\beta$-Laguerre ensembles. Moreover, for the frozen case $k=\infty$, these processes
degenerate to deterministic
or pure jump processes.

We use the generators for Bessel and Dunkl  processes  of types A and B and  derive analogues of Wigner's semicircle  and  Marchenko-Pastur
limit laws for  $N\to\infty$ for the empirical distributions of the particles with arbitrary initial empirical distributions
by using free convolutions. In particular, for Dunkl processes of type B new non-symmetric
semicircle-type limit distributions on 
 $\mathbb R$ appear. Our results imply that the form of the limiting measures is already completely determined by the frozen processes. Moreover, 
in the frozen cases, our approach leads to  a new simple
 proof of the  semicircle  and  Marchenko-Pastur limit
laws for the empirical measures of the zeroes of  Hermite  and Laguerre polynomials respectively.
\end{abstract}

\maketitle

\section{Introduction}

Calogero-Moser-Suther\-land particle systems on  $\mathbb R$ or $[0,\infty[$ 
 with $N$ particles 
can be described as  multivariate Bessel processes on 
 closed Weyl chambers in  $\mathbb R^N$. These Bessel processes 
 are time-homogeneous diffusions with well-known transition probabilities
and
generators of the transition semigroups; moreover they
are solution of the associated stochastic differential equations (SDEs); see  \cite{CGY,GY,R1,R2, 
RV1,RV2,DV, An} for the background. 
These multivariate Bessel processes $(X_{t,k})_{t\ge0}$ depend on their starting configurations for $t=0$,
root systems, and  a possibly multidimensional
multiplicity parameter $k$ which describes the strength of interaction
of the particles to each other and to the boundary.

Furthermore, based on the theory of Dunkl operators, these Bessel processes on Weyl chambers in
$\mathbb R^N$ can be extended in a canonical way to Feller processes  $(X_{t,k})_{t\ge0}$ on $\mathbb R^N$ 
by adding random reflections which are associated with the underlying root systems and 
multiplicity parameters $k$; see \cite{CGY, GY, R1, RV1, RV2} for the background.
These diffusion-reflection processes are called Dunkl processes;
for the background in analysis and mathematical physics see \cite{R2, An, DV}
and references there.
For these  Bessel and Dunkl processes
$(X_{t,k}=(X_{t,k}^1,\ldots,X_{t,k}^N))_{t\ge0}$
we derive  limit theorems for the empirical distributions
\begin{equation}\label{intro-empirical} \frac{1}{N}(\delta_{X_{t,k}^1/\sqrt N}+\ldots +\delta_{X_{t,k}^N/\sqrt N})
  \end{equation}
of the $N$ particles as $N\to\infty$ for $t>0$ under the condition that these 
empirical distributions converge for $t=0$ and $N\to\infty$ weakly to some given probability measure
 $\mu\in M^1( \mathbb R)$ which satisfies some moment condition.
We  prove that then  the measures in (\ref{intro-empirical})
converge a.s.~weakly to
probability measures $\mu_t\in M^1( \mathbb R)$ which can be described  in terms of  $\mu$
and free additive convolutions $\boxplus$. The appearance of free probability is not surprising, as for some root systems and
multiplicities $k$, our Bessel processes describe  the evolutions of
spectra of classical random matrix models like the $\beta$-Hermite and  $\beta$-Laguerre ensembles
of Dumitriu and Edelman \cite{DE1, DE2}.
Thus our results are closely related to Wigner's semicircle laws and Marchenko-Pastur limit laws
in different random matrix settings; see e.g.~\cite{AGZ, D, HT, Me, NS, OP, RS}. 
We  mention that in particular the dynamic approach in Section 4.3 of \cite{AGZ} is closely related to our paper.
However,  our approach via moments is simpler than that in \cite{AGZ}
in view of the technical tools  on  stochastic processes.
Moreover,  in \cite{AGZ} only  processes of type A are considered.

It is clear that for our limit theorems
we  need some control on the  parameters $k$ and the types of root systems
which must exist for all dimensions $N$.
This and the need of nontrivial interactions of the particles are the
 reason that we will restrict our attention to  the root systems of types $A_{N-1}$
and $B_N$ on  $\mathbb R^N$. Moreover, 
as the processes for the root systems $D_N$ differ from those for $B_N$ only in the behavior of one extremal
particle (with a suitable relation between the multiplicities; see e.g. \cite{AV1, V}),
our results on the empirical distributions of $N$ particles for $N\to\infty$
for the root systems $D_N$ may be easily regarded as a special case of some $B_N$-case.

We next briefly summarize some details of the main results of this paper.

For the root systems $A_{N-1}$, we fix a multiplicity $k\in]0,\infty[$. The associated Bessel processes
    $(X_{t,k})_{t\ge0}$ then live on the closed Weyl chambers
$$C_N^A:=\{x\in \mathbb R^N: \quad x_1\ge x_2\ge\ldots\ge x_N\},$$
and  the generators of the transition semigroups are
\begin{equation}\label{def-L-A-intro} \cal L_k f:= \frac{1}{2} \Delta f +
 k \sum_{i=1}^N\Bigl( \sum_{j\ne i} \frac{1}{x_i-x_j}\Bigr) \frac{\partial}{\partial x_i}f ,
 \end{equation}
where we assume reflecting boundaries, i.e.,
the domain of $\cal L_k$ is
$$D(\cal L_k):=\{f|_{C_N^A}: \>\> f\in C^{(2)}(\mathbb R^N),
\>\>\> f\>\>\text{ invariant under all coordinate permutations}\}.$$
It will be convenient, also to consider the renormalized processes
$(\tilde X_{t,k}:=\frac{ X_{t,k}}{\sqrt k})_{t\ge0}$, which  satisfy the SDEs
\begin{equation}\label{SDE-A-normalized-intro}
d\tilde X_{t,k}^i =\frac{1}{\sqrt k}dB_t^i + \sum_{j\ne i} 
 \frac{1}{\tilde X_{t,k}^i-\tilde X_{t,k}^j}dt\quad\quad(i=1,\ldots,N)
\end{equation}
with   $N$-dimensional Brownian motions $(B_t^1,\ldots,B_t^N)_{t\ge0}$. We mention that
these SDEs admit unique strong solutions  by \cite{GM} even if
these SDEs do not satisfy the standard assumptions for general  SDEs as e.g.~in the monograph \cite{P}
due to
the singularities on the boundary.
For $k=\infty$ these SDEs  degenerate to the  ODEs
\begin{equation}\label{ODE-A-normalized-intro}
\frac{d}{dt}\tilde X_{t,\infty}^i = \sum_{j\ne i} 
 \frac{1}{\tilde X_{t,\infty}^i-\tilde X_{t,\infty}^j}\quad\quad(i=1,\ldots,N).
\end{equation}
For the root systems $B_N$, we have  $k=(k_1,k_2)\in]0,\infty[^2$, the  Bessel  processes live on 
$$C_N^B:=\{x\in \mathbb R^N: \quad x_1\ge x_2\ge\ldots\ge x_N\ge0\},$$
and the generators  are
\begin{equation}\label{def-L-B}\cal L_k f:= \frac{1}{2} \Delta f +
 k_2 \sum_{i=1}^N \sum_{j\ne i} \Bigl( \frac{1}{x_i-x_j}+\frac{1}{x_i+x_j}  \Bigr)
 \frac{\partial}{\partial x_i}f 
\quad + k_1\sum_{i=1}^N\frac{1}{x_i}\frac{\partial}{\partial x_i}f, \end{equation}
where we again assume reflecting boundaries.
We now write the
 multiplicities
 as $k=(k_1,k_2)=(\nu\cdot \beta, \beta)$ with $\nu\ge0,\beta>0$. Moreover, we
 study the  renormalized Bessel processes $(\tilde X_{t,k}:=X_{t,k}/\sqrt \beta)_{t\ge0}$ which then
 satisfy the SDEs
\begin{equation}\label{SDE-B-normalized-intro}
  d\tilde X_{t,k}^i = \frac{1}{\sqrt{\beta}}dB_t^i+
  \Bigl(\sum_{j\ne i} \frac{\tilde X_{t,k}^i}{(\tilde X_{t,k}^i)^2-(\tilde X_{t,k}^j)^2}+ \frac{\nu}{\tilde X_{t,k}^i}\Bigr) dt
  \quad\quad(i=1,\ldots,N)
\end{equation}
with  $(B_t^1,\ldots,B_t^N)_{t\ge0}$ as above.
For $\beta=\infty$ these SDEs  degenerate to the ODEs
\begin{equation}\label{basic-ode-b-intro}
 \frac{d}{dt}\tilde X_{t,\infty}^i = 
  \sum_{j\ne i} \frac{\tilde X_{t,\infty}^i}{(\tilde X_{t,\infty}^i)^2-(\tilde X_{t,\infty}^j)^2}+ \frac{\nu}{\tilde X_{t,\infty}^i}
  \quad\quad(i=1,\ldots,N).
\end{equation}

We point out that the limit transitions $k,\beta\to\infty$  above for the root systems of type A and B lead to interesting
limit theorems which admit interpretations for $\beta$-random matrix ensembles; see \cite{DE2, AHV, AKM1, AKM2, AV1, AV2, GK, GoM, V, VW}.

We next recapitulate from  \cite{R1,R2,RV1,RV2} that for the root systems $A_{N-1}$ and $B_N$,
the transition probabilities of the Bessel processes  $(X_{t,k})_{t\ge0}$ have the form
\begin{equation}\label{density-general}
K_t(x,A)=c_k \int_A \frac{1}{t^{\gamma_k+N/2}} e^{-(\|x\|^2+\|y\|^2)/(2t)} J_k(\frac{x}{\sqrt{t}}, \frac{y}{\sqrt{t}}) 
\cdot w_k(y)\> dy
\end{equation}
for $t>0$,  $x\in C_N$, and  $A\subset C_N$ a Borel set (with $C_N=C_N^A, C_N^B$ respectively), with the  weight functions 
\begin{equation}\label{def-wk}
w_k^A(x):= \prod_{i<j}(x_i-x_j)^{2k}, \quad\quad 
w_k^B(x):= \prod_{i<j}(x_i^2-x_j^2)^{2k_2}\cdot \prod_{i=1}^N x_i^{2k_1},\end{equation}
and with the constants $\gamma_k^A(k):=kN(N-1)/2$ and   $\gamma_k^B(k_1,k_2):=k_2N(N-1)+k_1N$
respectively. In both cases,
$w_k$ is homogeneous of degree $2\gamma_k$, 
$c_k>0$ is a known normalization.
$J_k$ is a multivariate Bessel function of type $A_{N-1}$ or $B_N$ with multiplicities $k$ or $(k_1,k_2)$ 
which is analytic on $\mathbb C^N \times \mathbb C^N $ with
$ J_k(x,y)>0$ for $x,y\in \mathbb R^N $.
Moreover, $J_k(x,y)=J_k(y,x)$ and $J_k(0,y)=1$
for  $x,y\in \mathbb C^N $; see  e.g.  \cite{R1,R2}.
In particular, if $X_{0,k}=0$, then $X_{t,k}$ has the Lebesgue density
\begin{equation}\label{density-A-0}
 \frac{c_k}{t^{\gamma+N/2}} e^{-\|y\|^2/(2t)} \cdot w_k(y)
\end{equation}
on $C_N$ for $t>0$.
Hence, for the root systems $A_{N-1}$ and $B_N$, the  processes $(X_{t,k})_{t\ge0}$  are related to
$\beta$-Hermite and $\beta$-Laguerre  ensembles in \cite{DE1}.

We now turn to the main results of this paper.

The following  result in the A-case for $k=\infty$ is a special case of the main result of Section 2.
It uses the classical  free additive convolution $\boxplus$ and the semicircle distributions $\mu_{sc, R}$ with supports $[-R,R]$ for $R\ge0$
as discussed e.g.~in \cite{AGZ, NS}.

\begin{theorem}\label{general-free-convolution-a-intro}
  Let $\mu\in M^1(\mathbb R)$ be a probability measure with compact support, and let
 $(x_{N,n})_{N\ge1, 1\le n\le N}\subset\mathbb R$ a sequence
with $x_{N, n-1}\ge x_{N, n}$ for  $2\le n\le N$ such that the normalized empirical measures
\begin{equation}\label{starting-empirical-a-intro}
\mu_{N,0}:= \frac{1}{N}(\delta_{x_{N, 1}/\sqrt N}+\ldots \delta_{x_{N, N}/\sqrt N})
\end{equation}
tend weakly to  $\mu$ for $N\to\infty$.
If we take the solutions
$(\phi_{N,1}(t),\ldots,\phi_{N,N}(t))$  of (\ref{ODE-A-normalized-intro}) with $\phi_{N,n}(0)=x_{N,n}$  and 
the associated  normalized
 empirical measures
$$\mu_{N,t}:= \frac{1}{N}(\delta_{\phi_{N,1}(t)/\sqrt N}+\ldots +\delta_{\phi_{N,N}(t)/\sqrt N}) \quad\quad(t\ge0)$$
for $N\in\mathbb N$, then  for each 
$t\in[0,\infty[$, 
the  $\mu_{N,t}$ tend weakly to  $\mu_{sc, 2\sqrt{t}}\boxplus \mu$.
\end{theorem}

The proof of this result will be based on recursive formulas for the moments of the measures $\mu_{N,t}$ which follow from 
(\ref{ODE-A-normalized-intro}). By using the Stieltjes   and R-transforms of
the measures $\mu_{N,t}$ and their limits, we shall see that the limits of the $\mu_{N,t}$ are equal to $\mu_{sc, 2\sqrt{t}}\boxplus \mu$.
We mention that this ODE-approach includes a classical limit result on the empirical distributions
of the zeroes of the classical Hermite polynomials $H_N$ for $N\to\infty$; see Corollary \ref{classical-limit-hermite-ns} below.
For other proofs of this result see e.g.~\cite{D, G, KM}.

In Section 3 we  use standard techniques from probability like the  Burkholder-Davis-Gundy inequality and
the lemma of Borel-Cantelli
to extend the
 results for $k=\infty$ to a.s.~results for Bessel processes of type A for finite multiplicities $k$:

\begin{theorem}\label{limit-theorem-a-final-intro}
Let $\mu\in M^1(\mathbb R)$ be a probability measure with compact support, and let
 $(x_{N,n})_{N\ge1, 1\le n\le N}\subset\mathbb R$ 
with $x_{N, n-1}\ge x_{N, n}$ for  $2\le n\le N$ such that the  measures
 in (\ref{starting-empirical-a-intro})
tend weakly to  $\mu$ for $N\to\infty$.

 For $k\geq 1/2$ and $N\in\mathbb N$, consider the renormalized 
Bessel processes $(\tilde X_{t,k})_{t\geq 0}$  with  start in
$(x_{N,1}, \ldots,x_{N,n})\in C_N^A$.
Then, for  $t\ge0$, the empirical measures
\begin{equation}\label{SDE-empirical-a-intro}
\tilde\mu_{N,t}:= \frac{1}{N}(\delta_{\tilde X_{t,k}^1/\sqrt N}+\ldots +\delta_{\tilde X_{t,k}^N/\sqrt N}).
\end{equation}
 tend weakly to 
 $\mu_{sc, 2\sqrt{t}}\boxplus \mu$ a.s.~for $N\to\infty$.
\end{theorem}

We now turn to the main results for the case B in the Sections  4 and 5.
The first result is  analogous to Theorem \ref{general-free-convolution-a-intro}:

\begin{theorem}\label{free-convolution-b-intro}
Let $\mu\in M^1([0,\infty[)$ be a probability measure with compact support.
Let
 $(x_{N,n})_{N\ge1, 1\le n\le N}\subset[0,\infty[$ 
with $(x_{N, 1},\ldots, x_{N, N})\in C_N^B$ such that the  measures
\begin{equation}\label{empirical-measure-start-b-intro}
\mu_{N,0}:= \frac{1}{N}(\delta_{x_{N, 1}/\sqrt{2N}}+\ldots \delta_{x_{N, N}/\sqrt{2N}})
\end{equation}
tend weakly to  $\mu$ for $N\to\infty$. Consider  the  solutions
$\phi_N$  of (\ref{basic-ode-b-intro}) with start in
$(x_{N,1},\cdots, x_{N,N})$. If
$$\lim_{N\to\infty} \nu(N)/N=:\nu_0\geq 0,$$
then for each 
$t\in[0,\infty[$, the 
 empirical measures
$$\mu_{N,t}:= \frac{1}{N}(\delta_{\frac{\phi_{N,1}(t)}{\sqrt{2 N}}}+\ldots +\delta_{\frac{\phi_{N,N}(t)}{\sqrt{2N}}}) \quad\quad(t\ge0),$$
 tend weakly to  $\sqrt{\mu_{MP, \nu_0, t}\boxplus (\mu_{sc, 2\sqrt{t}}\boxplus \mu_{even})^2}$
 where the symbols $\sqrt{.}$ and $.^2$ mean push forwards of probability measures under these mappings,  $\mu_{even}$ is the even part of $\mu$, and
the measures $\mu_{MP, \nu_0, t}$ are Marchenko-Pastur distributions with parameters $ \nu_0, t$.
\end{theorem}

Again,  this  ODE-approach leads to a classical limit result on the empirical distributions
of the zeroes of the classical Laguerre polynomials $L_N^{(\alpha)}$ for $N\to\infty$ in Section 4;
see also \cite{G, KM2} and references there for other proofs of these facts.
Moreover, this result admits the following extension:

\begin{theorem}\label{free-convolution-b-random-intro}
Let $\mu\in M^1([0,\infty[)$ be a probability measure with compact support.
Let
 $(x_{N,n})_{N\ge1, 1\le n\le N}\subset[0,\infty[$ 
with $(x_{N, 1},\ldots, x_{N, N})\in C_N^B$ such that the  measures in (\ref{empirical-measure-start-b-intro})
tend weakly to  $\mu$. Consider the normalized Bessel processes
$(\tilde{X}_{t,k})_{t\ge0}$ of type B with start in $(x_{N,1},\ldots,x_{N,N})\in C_N^B $.
Then, for each $t\ge0$,  and 
$\lim_{N\to\infty} \nu(N)/N=:\nu_0\geq 0$,
 the measures
 $$\mu_{N,t}:= \frac{1}{N}(\delta_{\frac{\tilde{X}^1_{t,k}}{\sqrt{2 N}}}+\ldots +\delta_{\frac{\tilde{X}^N_{t,k}}{\sqrt{2N}}}) $$
tend a.s.~weakly to  $\sqrt{\mu_{MP, \nu_0, t}\boxplus (\mu_{sc, 2\sqrt{t}}\boxplus \mu_{even})^2}$.
\end{theorem}

The description of the limits in Theorems \ref{free-convolution-b-intro} and \ref{free-convolution-b-random-intro}
seems to be new; a partial result on the PDEs of the Stieltjes transforms of
the limits can be found in  \cite{CG}.

Finally, in Sections 6-8 we turn to  Dunkl processes.
For the root systems $A_{N-1}$, the Dunkl processes differ from the corresponding Bessel processes only by additional
permutations of particles. As
these permutation have no influence to the limit theorems 
\ref{general-free-convolution-a-intro} and \ref{limit-theorem-a-final-intro}, the transition from Bessel to Dunkl processes
leads to the same results; we thus
do not  study this case.

However, for root systems of type $B$, the transition from Bessel to Dunkl processes leads to  additional random sign-changes of
all particles even in the freezing case $\beta=\infty$ and thus to new effects.
To explain the main results, we first recapitulate some notations. We fix some multiplicity
 $k=(k_1,k_2)\in [0,\infty[^2$ for the root system $B_N$  and
 write these constants as $(k_1,k_2)=(\beta,\nu\beta)$ with $\beta>0$ and $\nu\ge0$ as above.
 By \cite{RV1, RV2, CGY}, the associated renormalized Dunkl processes $(\tilde X_{t,\nu,\beta})_{t\ge0}$ on $\mathbb R^N$ are
 then defined as  Feller processes on 
$\mathbb R^N$ with the generators 
\begin{equation}\label{Dunkl-Laplacian-B-renormalized}
  \tilde{\cal L}_{k_0,\beta}u(x):= \frac{1}{2\beta}\Delta u(x)+ L_\nu u(x)
\end{equation}
for $u\in C_c^2(\mathbb R^N)$ where 
\begin{align}\label{generator-frozen-b-intro}
L_\nu u(x)&:= \sum_{i=1}^N \Bigl(\sum_{j:\> j\ne i} \frac{2x_i}{x_i^2-x_j^2}+\frac{\nu}{x_i}\Bigr) u_{x_i}(x)
+\frac{\nu}{2}\sum_{i=1}^N \frac{u(\sigma_ix)-u(x)}{x_i^2}\notag\\
&+\frac{1}{2} \sum_{i,j: \> j\ne i} \Bigl(\frac{u(\sigma_{i,j}x)-u(x)}{(x_i-x_j)^2}+\frac{u(\sigma_{i,j}^-x)-u(x)}{(x_i+x_j)^2}\Bigr)
\end{align}
is, by definition, the generator of the frozen process with $\beta=\infty$. $\sigma_i, \sigma_{i,j}, \sigma_{i,j}^-$ $(i\ne j$) denote reflections on  $\mathbb R^N$
  where $\sigma_i$ changes the sign of the $i$-th coordinate, $\sigma_{i,j}$ exchanges the coordinates $i,j$, and
  $\sigma_{i,j}^-$ exchanges the coordinates $i,j$ and changes the signs of these coordinates in addition.

If the starting measure $\mu\in M^1(\mathbb R)$ of a sequence of such renormalized Dunkl processes is symmetric,
then we may choose the starting sequences
$(x_{N,n})_{N\ge1, 1\le n\le N}\subset \mathbb R$ as e.g.~in Theorem \ref{free-convolution-b-intro}  in a symmetric way, and
 symmetry arguments lead to symmetric extensions of the Marchenko-Pastur limit theorems 
\ref{free-convolution-b-intro} and \ref{free-convolution-b-random-intro}.

However, for non-symmetric starting configurations, completely new limit distributions appear.
We study the analytic part of this problem in Section 7 for the frozen case
where we describe the even parts of the limit measures via  Theorem \ref{free-convolution-b-intro} and \ref{free-convolution-b-random-intro}.
while the odd parts are described via their Stieltjes transforms. For this we shall first derive  linear PDEs for
the Stieltjes transforms of the odd parts, and then we shall deduce these Stieltjes transforms in an explicit way.
Unfortunately, we are not able to describe the associated probability measures  via free convolutions in  general.
However, 
for the case  $\lim_{N\to\infty} \nu(N)/N=\nu_0=0$ and  a quarter circle distribution on $[0,2]$
as starting measure, we are able to compute the associated measures for all times $t\ge0$ in an explicit way;
see Example \ref{quarter-example}.
After this analytic part in Section 7 on the frozen case, we extend 
Theorem \ref{free-convolution-b-random-intro} to renormalized Dunkl processes $(\tilde X_{t,\nu,\beta})_{t\ge0}$ in Section 8.

\section{A sequence of ODEs and the semicircle law} 

In this section we study  a sequence of ODEs with $N\ge2$ equations which are closely related to
the zeroes of the  Hermite polynomials $H_N$. We show that 
the empirical distributions of the $N$-dimensional solutions
 of these ODEs for $N\to\infty$ are related  to the semicircle law. We identify the
 limits as free additive convolution of the semicircle law with the law associated to the starting value.
As a special case this  leads to the well-known semicircle law for 
the empirical distributions of the zeroes of $H_N$.
Let us start with the  ODEs:

\begin{ODE}\label{ODE-Hermite}
Let $N\ge 2$. On the interior of the closed Weyl chamber
$$C_N^A:=\{x\in \mathbb R^N: \quad x_1\ge x_2\ge\ldots\ge x_N\}\subset \mathbb R^N$$
of type A, consider the  $\mathbb R^N$-valued function
$$H(x):=
\Bigl( \sum_{j\ne1} \frac{1}{x_1-x_j},\ldots,\sum_{j\ne N} \frac{1}{x_N-x_j} \Bigr).$$
It is shown in \cite{VW2}
 that for each initial condition $x_0\in C_N^A$, the ODE
\begin{equation}\label{basic-ode-a}
\frac{dx}{dt}(t) =H(x(t)), \quad\quad x(0)=x_0
\end{equation}
has a unique solution for all $t\ge0$ in the sense that $[0,\infty[\to  C_N^A$, $t\mapsto x(t)$
is continuous such that  $x(t)$ is in the interior of  $C_N^A$ and solves the ODE in  (\ref{basic-ode-a}) for $t>0$.
For the ODE  (\ref{basic-ode-a})   we also refer to \cite{AV1, VW}. 
We  denote  solutions of the ODE (\ref{basic-ode-a})  by
 $\phi_N:=(\phi_{N,1},\ldots,\phi_{N,N})$  where we suppress the dependence on $x_0$.
\end{ODE}

For  $x_0=0\in C_N^A$, the  solution of (\ref{basic-ode-a})
  can be expressed via the 
zeroes of the  Hermite polynomial $H_N$   where, as usual,  the $(H_N)_{N\ge 0}$ are orthogonal w.r.t.
 the density  $e^{-x^2}$ on $\mathbb R$  as e.g.~in \cite{S}. For this we need the following fact  due to Stieltjes;
see Section 6.7 of \cite{S} or \cite{AKM1}:

\begin{lemma}\label{char-zero-A}
Let $z\in C_N^A$. Then $z:=(z_1,\ldots,z_N)$
 consists of the ordered zeroes of  $H_N$ if and only if
$$ z_i= \sum_{j: j\ne i} \frac{1}{z_i-z_j} \quad\text{for} \quad
i=1,\ldots,N.$$
\end{lemma}

Lemma \ref{char-zero-A}  immediately implies the following result; see \cite{AV1}:

\begin{corollary}\label{special-solution}
Let $z\in C_N^A$ as above and $c\ge 0$. Then $\phi_N(t)= \sqrt{2t+c^2}\cdot {z} $
is a solution of  (\ref{basic-ode-a}).
\end{corollary}

We now turn to the empirical measures of  solutions of  (\ref{basic-ode-a}). 
We choose starting sequences $(x_{N,k})_{1\le k\le N}\subset \mathbb R$ with
 $(x_{N,1},\ldots,x_{N,N})\in C_N^A$ such that for each $N$ the ODE (\ref{basic-ode-a}) has a solution with start in
 $(x_{N,1},\ldots,x_{N,N})$.  For $t\ge0$ consider the associated solutions  $\phi_N(t)$  and  normalized
 empirical measures
\begin{equation}\label{emp-measure-basic-a}
\mu_{N,t}:= \frac{1}{N}(\delta_{\phi_{N,1}(t)/\sqrt N}+\ldots +\delta_{\phi_{N,N}(t)/\sqrt N})\in M^1(\mathbb R).
\end{equation}

The aim of this section is to characterize the limiting empirical measures $\mu_t$ of $\mu_{N,t}$ 
for $N\to\infty$ and  $t\ge0$ under the condition that the $\mu_{N,0}$ converge to some probability measure $\mu$.
For this we first derive a recurrence equation for the  moments of the $\mu_{N,t}$. This will lead
to   PDEs for the  Stieltjes transforms of the $\mu_{N,t}$ and  $\mu_t$. With the aid of the  R-transform
from free probability (see Section 5.3 of \cite{AGZ}) we then identify the $\mu_t$ as free  additive convolutions
of $\mu$ with suitably scaled semicircle laws.
 For more details on free probability we refer to \cite{NS}.

Denote the $l$-th moment $(l\in\mathbb N_0$) of the probability measure $\mu_{N,t}$ by
\begin{equation}\label{def-moments-case-a}
S_{N,l}(t):= \int_{\mathbb R} y^l \> d\mu_{N,t}(y) = \frac{1}{N^{l/2+1}}(\phi_{N,1}(t)^l+\ldots+\phi_{N,N}(t)^l).\end{equation}
Then $S_{N,0}(t)=1$. Moreover, by  (\ref{basic-ode-a}),
\begin{equation}\label{dS1-hermite}
\frac{d}{dt}S_{N,1}(t)=\frac{1}{N^{3/2}}\sum_{i,j=1; i\ne j}^N \frac{1}{\phi_{N,i}(t)-\phi_{N,j}(t)} =0,\end{equation}
i.e., $S_{N,1}(t)=S_{N,1}(0)$ for all $t\ge0$. By the same reasons,
\begin{equation}\label{dS2-hermite}
\frac{d}{dt}S_{N,2}(t)=\frac{2}{N^{2}}\sum_{i,j=1; i\ne j}^N \frac{\phi_{N,i}(t)}{\phi_{N,i}(t)-\phi_{N,j}(t)}
=\frac{2}{N^{2}} \cdot \frac{N(N-1)}{2}=\frac{N-1}{N}
\end{equation}
and
\begin{align}\label{dS3-hermite}
\frac{d}{dt}S_{N,3}(t)=&\frac{3}{N^{5/2}}\sum_{i,j=1; i\ne j}^N \frac{\phi_{N,i}(t)^2}{\phi_{N,i}(t)-\phi_{N,j}(t)}
=\frac{3}{2N^{5/2}}\sum_{i,j=1; i\ne j}^N (\phi_{N,i}(t)+\phi_{N,j}(t))\notag\\
=& \frac{3(N-1)}{N^{5/2}}\sum_{i=1}^N \phi_{N,j}(t) = \frac{3(N-1)}{N} S_{N,1}(0).
\end{align}
These computations yield the following result for $l=0,1,2,3$:

\begin{lemma}\label{recurrence-sl-hermite-det}
Let $(x_{N,k})_{1\le k\le N}\subset\mathbb R$ be  starting sequences  such that for all  $l\in \mathbb N_0$,
$$c_l(0):=\lim_{N\to\infty} S_{N,l}(0)=\lim_{n\to\infty}  \frac{1}{N^{l/2+1}}(x_{N,1}^l+\ldots+x_{N,N}^l)<\infty$$
exists. Then
for  $l\in \mathbb N_0$,
$$c_l(t):=\lim_{N\to\infty} S_{N,l}(t)$$
exists locally uniformly in $t\in[0,\infty[$ and satisfies the recurrence relation
\begin{equation}\label{recurrence-herm}
c_l(t)=c_l(0)+ \frac{l}{2}\int_0^t\Bigl(\sum_{k=0}^{l-2} c_{l-2-k}(s)c_k(s)\Bigr)\> ds
\end{equation}
for $l\geq 2$ and start $c_0(t)=1,\,\, c_1(t)=c_1(0)$. 

For each $l\in  \mathbb N_0$, $c_l(t)$ is a polynomial in $t$
of degree at most $\lfloor l/2\rfloor$ with a nonnegative ``leading'' coefficient of order $\lfloor l/2\rfloor$.
\end{lemma}

\begin{proof} By our preceding computations,
\begin{equation}\label{c0-3-herm}
c_0(t)=1, \quad c_1(t)=c_1(0),  \quad c_2(t)=c_2(0)+t,  \quad c_3(t)=c_3(0)+3c_1(0)t.
\end{equation}
For $l\ge4$, we have
\begin{align}
\frac{d}{dt}S_{N,l}(t)=& \frac{l}{N^{l/2+1}}\sum_{i,j=1; i\ne j}^N 
\frac{\phi_{N,i}(t)^{l-1}}{\phi_{N,i}(t)-\phi_{N,j}(t)}\notag \\
=& \frac{l}{N^{l/2+1}}\sum_{1\le i<j\le N}\frac{\phi_{N,i}(t)^{l-1}-\phi_{N,j}(t)^{l-1}}{\phi_{N,i}(t)-\phi_{N,j}(t)}
\notag\\
=& \frac{l}{2N^{l/2+1}}\sum_{i,j=1; i\ne j}^N \Bigl( \phi_{N,i}(t)^{l-2}+\phi_{N,i}(t)^{l-3}\phi_{N,j}(t)+\ldots+
\phi_{N,j}(t)^{l-2}\Bigr).
\notag\end{align}
As 
$$\frac{1}{N^{l/2+1}}\sum_{i,j=1; i\ne j}^N \phi_{N,i}(t)^{l-2-k}\phi_{N,j}(t)^k= S_{N,l-2-k}(t)S_{N,k}(t)-\frac{S_{N,l-2}(t)}{N}$$
for $k=1,2,\ldots,l-3$, and as 
$$\frac{1}{N^{l/2+1}}\sum_{i,j=1; i\ne j}^N \phi_{N,i}(t)^{l-2}=\frac{N-1}{N}S_{N,l-2}(t),$$
we get
\begin{align}\label{derivative-recursion-n-a}
\frac{d}{dt}S_{N,l}(t)&= \frac{l}{2}\Bigl(\frac{2N+1-l}{N}S_{N,l-2}(t)+
\sum_{k=1}^{l-3} S_{N,l-2-k}(t)S_{N,k}(t)\Bigr) \notag\\
&= \frac{l}{2}\Bigl(\frac{1-l}{N}S_{N,l-2}(t)+
\sum_{k=0}^{l-2} S_{N,l-2-k}(t)S_{N,k}(t)\Bigr).
\end{align}
Hence, for  $l\ge4$ we obtain in an inductive way that the limit
\begin{align}
c_l(t):=& \lim_{N\to\infty}S_{N,l}(t) =c_l(0)+ \lim_{N\to\infty}\int_0^t \frac{d}{ds}S_{N,l}(s) \> ds \notag \\
=&c_l(0)+ \frac{l}{2}\int_0^t\Bigl(2c_{l-2}(s)+\sum_{k=1}^{l-3} c_{l-2-k}(s)c_k(s)\Bigr)\> ds\notag
\end{align}
exists locally uniformly in $t\in[0,\infty[$. Moreover, the $c_l(t)$ satisfy 
$$c_l(t)=c_l(0)+ \frac{l}{2}\int_0^t\Bigl(\sum_{k=0}^{l-2} c_{l-2-k}(s)c_k(s)\Bigr)\> ds.$$
(\ref{c0-3-herm}) and this recurrence imply by an easy induction that for each $l\in\mathbb N_0$,
$c_l(t)$ is a polynomial of degree at most $\lfloor l/2\rfloor$ with a nonnegative  coefficient for this order.
\end{proof}

For even $l$ we  next  determine the leading coefficients of the polynomials  $c_l(t)$ of order $ l/2$.
For this we
recapitulate  the Catalan numbers
\begin{equation}\label{catalan-def}
C_0:=1, \quad C_n:=\frac{1}{n+1}{2n \choose n}= {2n \choose n}-{2n \choose n+1} \quad(n\ge1)
\end{equation}
which admit the  well known  recurrence relation (see e.g.~Section 2.1.1 of \cite{AGZ}):
\begin{equation}\label{recurrence-Catalan}
C_0=C_1=1, \quad C_{n+1}= \sum_{k=0}^n C_k C_{n-k} \quad (n\ge1).
\end{equation}
We compare this with (\ref{c0-3-herm}) and (\ref{recurrence-herm}) where
 $c_l(t)$ has degree at most  $\lfloor l/2\rfloor$. A simple induction then yields:

\begin{lemma}\label{leading-Catalan}
The polynomial $c_{2l}(t)$ has the degree $l$ with the Catalan number $C_l$ as leading coefficient for 
$l\in\mathbb N_0$.
\end{lemma}

\begin{example}\label{semicircle}
Assume that the solutions of our ODEs satisfy $\phi_N(0)=0$ for all $N$, i.e., that $x_{N,k}=0$ for all $N,k$.
Then $c_0(0)=1$ and $c_l(0)=0$ for  $l\ge1$. Therefore  $c_0(t)=1$, 
$c_1(t)=0$, $c_2(t)=t$ and $c_3(t)=0$ for $t\ge0$. Hence, by
(\ref{recurrence-herm}),
$$c_{2l}(t)=C_lt^l\quad\text{and}\quad c_{2l+1}(t)=0 \quad (t\ge0, l\in\mathbb N_0).$$ 

We next recapitulate that for $R>0$, a random variable $X_R$ with the  semicircle law $\mu_{sc,R}$ with
density $f_R(x):= \frac{2}{\pi R^2}\sqrt{R^2-x^2}$ for $|x|\le R$ and
$f_R(x)=0$ otherwise has  the moments
$$E(X_R^{2n})=\left({R \over 2}\right)^{2n} C_n \quad\text{and}\quad E(X_R^{2n+1})=0 \quad\text{for}\quad n\ge0;$$
see  e.g.~Section 2.1.1 of \cite{AGZ}.
We thus conclude from the moment convergence theorem that 
for $t>0$ the  empirical measures
$\mu_{N,t}$ of the (renormalized)  solutions of our ODEs with start in the origin tend weakly to
$\mu_{sc,2\sqrt t}$ for $N\to\infty$.
\end{example}

If we combine Example \ref{semicircle} with Corollary \ref{special-solution} for $c=0$, $t=1/2$ there, 
we obtain the following classical result on the zeroes of the Hermite polynomials;
 see also  \cite{D, G, KM} for different proofs:

\begin{corollary}\label{classical-limit-hermite-ns} 
For $N\in\mathbb N$ let $z_1,\ldots,z_N$ be the zeroes of the Hermite polynomial $H_N$. Then the normalized 
empirical measures
$$\mu_{N}:= \frac{1}{N}(\delta_{z_1/\sqrt N}+\ldots+ \delta_{z_N/\sqrt N})$$
tend weakly to $\mu_{sc,\sqrt 2}$ for $N\to\infty$.
\end{corollary}

We next study the general case  with a start with an arbitrary probability measure $\mu\in M^1(\mathbb R)$ which
is determined uniquely by its moments
$c_l:=\int x^l\> d\mu(x)$ ($l\ge0$). 
This uniqueness holds in particular under the Carleman condition
\begin{equation}\label{carleman}
\sum_{l=1}^\infty c_{2l}^{-\frac{1}{2l}}=\infty;
\end{equation}
see p. 85 of \cite{A}. Moreover,   (\ref{carleman}) clearly follows from the condition
\begin{equation}\label{strong-carleman}
|c_l|\le (cl)^l \quad\quad \text{ for all}\quad l\ge0\quad \text{ and some}\quad c>0.
\end{equation}

Now let $\mu\in M^1(\mathbb R)$ be  determined uniquely by its moments $c_l$.
We choose a family
 $(x_{N,n})_{N\ge1, 1\le n\le N}\subset\mathbb R$  of numbers
with $x_{N, n-1}\ge x_{N, n}$ for  $2\le n\le N$ such that the  empirical measures
$$\mu_{N,0}:= \frac{1}{N}(\delta_{x_{N, 1}/\sqrt N}+\ldots \delta_{x_{N, N}/\sqrt N})$$
tend weakly to  $\mu$ for $N\to\infty$, i.e., by the moment convergence theorem, that
$$\lim_{N\to\infty} S_{N,l}(0):=\lim_{N\to\infty}  \frac{1}{N^{l/2+1}}(x_{N,1}^l+\ldots+x_{N,N}^l)=c_l \quad(l\ge0).$$
For  $N\ge 2$ we now consider  the solutions   $\phi_N(t)$ of 
 (\ref{basic-ode-a})  with start in $(x_{N,1},\ldots,x_{N,N})$ and the  normalized
 empirical measures
$$\mu_{N,t}:= \frac{1}{N}(\delta_{\phi_{N,1}(t)/\sqrt N}+\ldots +\delta_{\phi_{N,N}(t)/\sqrt N}) \quad\quad(t\ge0).$$

\begin{proposition}\label{general-free-convolution-a-firststep}
In the preceding setting, the limits
 $$c_l(t):=\lim_{N\to\infty} S_{N,l}(t)  \quad\quad (t\in[0,\infty[, l\ge0)$$
exist. Moreover, if the moment condition (\ref{strong-carleman}) holds for $\mu$, then
 for each 
$t\in[0,\infty[$, the sequence $(c_l(t))_{l\ge0}$ is the sequence of moments of some unique probability
measure $\mu_t\in M^1(\mathbb R)$ for which  (\ref{strong-carleman}) also holds. Moreover, 
the  $\mu_{N,t}$ tend weakly to  $\mu_t$ for $N\to\infty$.
\end{proposition}

\begin{proof}
 The arguments in the proof of Lemma \ref{recurrence-sl-hermite-det}
show that the limits $c_l(t)$ exist for all $l,t$, and that
 the  $c_l(t)$ satisfy the recurrence (\ref{recurrence-herm}).

Assume now that the $c_l=c_l(0)$ satisfy  (\ref{strong-carleman}), i.e.,
$|c_l|\le (cl)^l$ for all $l$ and some $c>0$. We fix $t>0$ and show that there exists $R=R(t)>1$
such that 
\begin{equation}\label{estimate-c-l-t}
|c_l(s)|\le (Rl)^l \quad\quad \text{ for all}\quad l\ge0, \quad s\in[0,t].
\end{equation}
It is clear from (\ref{c0-3-herm}) that  (\ref{estimate-c-l-t}) holds for $l=0,1,2,3$ and $R$ sufficiently large.
Moreover, for $l\ge4$, we use induction on $l$. In fact, the assumption of our induction, 
 the recurrence (\ref{recurrence-herm}), and  the condition  (\ref{strong-carleman}) imply that for $\tau\in[0,t]$,
\begin{align}\label{est-c-l}
|c_l(\tau)|&\le |c_l(0)|+ \frac{l}{2}\int_0^\tau\Bigl(\sum_{k=0}^{l-2}| c_{l-2-k}(s)|\cdot|c_k(s)|\Bigr)\> ds \notag\\
&\le (cl)^l+ \frac{l(l-1)t}{2}R^{l-2}l^{l-2}.
\end{align}
If we choose $R$ large enough depending on $c,t$, we see that the RHS of (\ref{est-c-l}) is bounded by $R^{l}l^{l}$,
 which then proves (\ref{estimate-c-l-t}). In summary,
 for each $t\ge0$, the sequence $(c_l(t))_l$ satisfies the Carleman condition, and
 this sequence is the limit of the moment sequences of the measures  $\mu_{N,t}\in M^1(\mathbb R)$ for $N\to\infty$.
Hence, by the moment convergence theorem, $(c_l(t))_l$ is the  moment sequences of a unique probability
measure $\mu_t\in M^1(\mathbb R)$, and  the  $\mu_{N,t}$ tend weakly to  $\mu_t$.
\end{proof}

We next identify the limit measures $\mu_t$ in Proposition
 \ref{general-free-convolution-a-firststep} as the free additive convolutions 
\begin{equation}\label{ident-free-a}
\mu_t= \mu_{sc, 2\sqrt{t}}\boxplus \mu  \quad\quad\text{for}\quad t\ge0
\end{equation}
with the free  additive convolution  discussed e.g.~in  \cite{NS, AGZ}.
To prove this we need some additional tools. 
We first recapitulate the Stieltjes transform
\begin{equation}\label{def-Stieltjes}
G_\mu(z):=\int_{\mathbb R} \frac{1}{z-x}\> d\mu(x)  \quad\quad\text{for}\quad z\in H:=\{z\in\mathbb C:\> \Im z>0\}
\end{equation}
of  a probability measure $\mu\in M^1(\mathbb R)$. Clearly, $G_\mu$ is analytic on $H$. We next derive PDEs 
for the  Stieltjes transforms
$$G(t,z):=G_{\mu_t}(z), \quad G^N(t,z):=G_{\mu_{N,t}}(z) \quad\quad (t\ge0, \> z\in H)$$
of the  measures $\mu_t$ and  $\mu_{N,t}$. 
In the setting of 
Proposition \ref{general-free-convolution-a-firststep} we now have:

\begin{proposition}\label{general-free-convolution-a-secondstep}
\begin{enumerate}
\item[\rm{(1)}] For all $N\in\mathbb N$, $t\ge0$, $z\in H$, the partial derivatives of $ G^N$ satisfy
$$G_t^N(t,z)=-G^N(t,z)G^N_z(t,z)  - \frac{1}{N} E^N(t,z)$$
with the error term  $E^N(t,z)$ defined below in (\ref{error-gn}).
\item[\rm{(2)}] Assume that in addition the moment condition  (\ref{strong-carleman}) holds for the start measure
$\mu$. Then for $t\ge0$, $z\in H$, the function $G$ satisfies Burgers equation 
$$G_t(t,z)=-G(t,z)G_z(t,z).$$
\end{enumerate}
\end{proposition}

The appearance of Burgers  equation here is not surprising, as this connection is well-known in the context of 
dynamic versions of Gaussian unitary (or symmetric or symplectic) ensembles; see  the next section and e.g. \cite{CG, Men}.

\begin{proof}
 For $t\ge0$ and $z\in H$ with $|z|$ sufficiently large (depending on $N$) we have
$$G^N(t,z)=\int_{\mathbb R} \frac{1}{z-x}\> d\mu_{N,t}(x)  = \sum_{l=0}^\infty \frac{S_{N,l}(t)}{z^{l+1}}$$
and thus
$$G^N_t(t,z)=  \sum_{l=0}^\infty \frac{1}{z^{l+1}}    \frac{d}{dt}S_{N,l}(t).$$
If we apply (\ref{dS1-hermite})-(\ref{dS3-hermite}) as well as the recurrence
 (\ref{derivative-recursion-n-a}), we obtain
\begin{align}
G_t^N(t,z)=& \frac{N-1}{N}\cdot \frac{1}{z^{3}}  +  \frac{3(N-1)}{N}\cdot S_{N,1}(0)\cdot \frac{1}{z^{4}} \notag\\
&+\sum_{l=4}^\infty   \frac{1}{z^{l+1}} \frac{l}{2}\Bigl(\frac{1-l}{N}S_{N,l-2}(t)+
\sum_{k=0}^{l-2} S_{N,l-2-k}(t)S_{N,k}(t)\Bigr)\notag\\
=&\sum_{l=2}^\infty  \frac{1}{z^{l+1}} \frac{l}{2}\sum_{k=0}^{l-2} S_{N,l-2-k}(t)S_{N,k}(t)
\quad+\frac{1}{N} E^N(t,z)\notag
\end{align}
with 
\begin{equation}\label{error-gn-1}
 E^N(t,z):=  - \frac{1}{z^{3}}-  \frac{3}{z^{4}}\cdot S_{N,1}(0)+ 
\frac{1}{2}\sum_{l=4}^\infty   \frac{l(1-l)}{z^{l+1}}\cdot S_{N,l-2}(t).
 \end{equation}
Using
$$z\cdot G_z^N(t,z)=-\sum_{l=0}^\infty    \frac{l+1}{z^{l+1}}\cdot S_{N,l}(t) , \quad
z^2\cdot G_{zz}^N(t,z)=\sum_{l=0}^\infty  \frac{(l+1)(l+2)}{z^{l+1}}\cdot S_{N,l}(t),$$
we obtain by some elementary calculation that
\begin{align}\label{error-gn}
 E^N(t,z)= &- \frac{1}{2}z^2\cdot G_{zz}^N(t,z)-2z\cdot G_z^N(t,z)-G^N(t,z)\\
&-  \frac{1}{z^3}\Bigl(1-S_{N,2}(0)-\frac{N-1}{N}t\Bigr) - 
\frac{3}{z^4}\Bigl(S_{N,1}(0)-  \frac{3(N-1)}{N}t-S_{N,3}(0)\Bigr).\notag
 \end{align}
As
\begin{align}
\frac{1}{2z}&\sum_{l=2}^\infty  \frac{1}{z^{l}}l\sum_{k=0}^{l-2} S_{N,l-2-k}(t)S_{N,k}(t)
= \frac{1}{2z}\sum_{l=0}^\infty\sum_{k=0}^{l} \frac{1}{z^{l+2}} (l+2) S_{N,l-k}(t)S_{N,k}(t)\notag\\
=&\frac{1}{2z^2}\sum_{l=0}^\infty\sum_{k=0}^{l} \left[(l+1-k)S_{N,l-k}(t) z^{-l-1+k} S_{N,k}(t) z^{-k}\right.\notag\\
&\quad\quad\quad\quad\quad\quad \left.+(k+1)S_{N,l-k}(t) z^{-l+k}S_{N,k}(t) z^{-k-1}\right]\notag\\
=&-G^N(t,z)G^N_z(t,z),\notag
\end{align}
part (1) follows for $z\in H$ sufficiently large. 
As both sides of the equation in (1) are analytic in $z\in H$, this equation
 holds for all  $z\in H$.

For (2) we recapitulate that the measures  $\mu_{N,t}$ tend weakly to  $\mu_t$ by Proposition 
\ref{general-free-convolution-a-firststep}. This implies that the Stieltjes transforms
$G^N(t,z)$ tend to $G(t,z)$ for  $t\ge0$ and locally uniformly for $z\in H$.
Hence, by the integral formulas of Cauchy, also $G^N_z(t,z)$ tends to $G_z(t,z)$ and  
 $G^N_{zz}(t,z)$ to $G_{zz}(t,z)$ for $z\in H$. Therefore, the error term $E^N(t,z)$ converges for $N\to\infty$
by Lemma \ref{recurrence-sl-hermite-det}. This and part (1) imply that the derivatives
$G_t^N(t,z)$ tend to $-G(t,z)G_z(t,z)$ for $N\to\infty$. Moreover, as
$$|G^N_t(t,z)|=\Bigl|\frac{d}{dt}\int_{\mathbb R} \frac{1}{z-x}\> d\mu_{N,t}(x)\Bigr|
\le \frac{1}{N}\sum_{k=1}^N \frac{1}{|z-\phi_{N,k}^\prime(t)/\sqrt N|}\le \frac{1}{\Im z},$$
we conclude by dominated convergence that for $t\ge0$
$$\lim_{N\to\infty}(G^N(t,z) - G^N(0,z))=-\int_0^t G(\tau,z)G_z(\tau,z)\> d\tau.$$
This implies (2).
\end{proof}

Proposition \ref{general-free-convolution-a-secondstep}(2)
now leads to (\ref{ident-free-a}) with the aid of the R-transform of measures
 $\mu\in M^1(\mathbb R)$ which is defined e.g.~in Section 5.3 of \cite{AGZ} 
as the formal power series
$R_\mu(z):= \sum_{n=0}^\infty k_{n+1}(\mu)z^n$ with the free cumulants $k_{n}(\mu)$ of the measure
$\mu$ for which all moments exist. As formal power and Laurent series and also as analytic functions on suitable domains in
 the upper halfplane (see Section 5.3.3 of \cite{AGZ}), the functions $R_\mu$ and $G_\mu$ are related  by 
\begin{equation}\label{relation-g-r-1}
z-\frac{1}{G_\mu(z)}=R_\mu(G_\mu(z)).
\end{equation}
If we apply this to the measures $\mu_t$ and the  R-transform $R(t,z):=R_{\mu_t}(z)$, we get
\begin{equation}\label{relation-g-r-2}
z=R(t,G(t,z)) +\frac{1}{G(t,z)}.
\end{equation}
Hence, on suitable domains,
\begin{eqnarray}
G_z(t,z)&=& -\frac{G^2(t,z)}{1-R_z(t, G(t,z))G^2(t,z)}\label{G-R-1}\\
G_t(t,z)&=& \frac{R_t(t,G(t,z))G^2(t,z)}{1-R_z(t, G(t,z))G^2(t,z)}.\label{G-R-2}
\end{eqnarray}
Proposition \ref{general-free-convolution-a-secondstep}(2) now implies that
$R_t(t,G(t,z))= G(t,z)$. As $z\mapsto G(t,z)$ is not constant, we arrive at
\begin{equation}\label{pde-r-a}
R_t(t,z)= z
\end{equation}
with $R(0,z)=R_\mu(z)$ the R-transform of the starting measure $\mu$.
Therefore, 
$$ R(t,z)= zt +R(0,z).$$
As by  5.3.23 and 5.3.26 of \cite{AGZ}, the R-transform satisfies
$$R_{\mu_{sc, 2\sqrt{t}}\boxplus \mu}(z)=R_{\mu_{sc, 2\sqrt{t}}}(z)+R_\mu(z)= zt +R(0,z),$$
and as the R-transform is injective, we finally obtain (\ref{ident-free-a}).

In summary, we have proved the following theorem  mentioned in the introduction

\begin{theorem}\label{general-free-convolution-a}
Let $\mu\in M^1(\mathbb R)$ be a probability measure satisfying (\ref{strong-carleman}), and let
 $(x_{N,n})_{N\ge1, 1\le n\le N}\subset\mathbb R$ 
with $x_{N, n-1}\ge x_{N, n}$ for  $2\le n\le N$ such that the empirical measures
\begin{equation}\label{starting-empirical-a}
\mu_{N,0}:= \frac{1}{N}(\delta_{x_{N, 1}/\sqrt N}+\ldots \delta_{x_{N, N}/\sqrt N})
\end{equation}
tend weakly to  $\mu$ for $N\to\infty$. If we form the associated solutions
$(\phi_{N,1}(t),\ldots,\phi_{N,N}(t))$  of (\ref{basic-ode-a}) and 
the associated  normalized
 empirical measures
$$\mu_{N,t}:= \frac{1}{N}(\delta_{\phi_{N,1}(t)/\sqrt N}+\ldots +\delta_{\phi_{N,N}(t)/\sqrt N}) \quad\quad(t\ge0),$$
then  for
$t\in[0,\infty[$, 
the  $\mu_{N,t}$ tend weakly to  $\mu_{sc, 2\sqrt{t}}\boxplus \mu$.
\end{theorem}

\begin{remark}
We show in the next section that the limit measures $\mu_t$ ($t>0$)
in Proposition \ref{general-free-convolution-a-firststep} also appear in a 
 dynamic version Wigner's semicircle law
for Gaussian unitary ensembles; c.f.~\cite{AGZ}. If one uses this together with
the results of Section 3, one obtains a further proof of
 (\ref{ident-free-a}).
\end{remark}

We finally study an ODE with an additional drift compared to (\ref{basic-ode-a}).
For this let $\phi_N(t,x_0)$ be a  solution of (\ref{basic-ode-a}) with start in $x_0\in C_N^A$. Then, by
 an easy computation (see \cite{VW}), 
\begin{equation}\label{renorming-stationary}
\tilde \phi_N(t,x_0):=       \phi_N(\frac{1-e^{-2\lambda t}}{2\lambda}, e^{-\lambda t}x_0)
\end{equation}
is a solution of the ODE
\begin{equation}\label{renorming-stationary-dgl}
\frac{dx}{dt}(t) =H(x(t))-\lambda x(t), \quad \phi_N(0,x_0)=x_0 
\end{equation}
 for $\lambda\in \mathbb{R}$ and vice versa. 
As the functions $\tilde \phi_N$ and  $\phi_N$ are related by the
 space-time transformation (\ref{renorming-stationary}),
  we obtain  the following semicircle limit law.

\begin{corollary}\label{ODE-stationaer}
Let  $x:=(x_{N,k})_{k\geq 1}$ be  starting sequences as in Lemma \ref{recurrence-sl-hermite-det} and $\lambda>0$.
Consider the solutions  $\tilde{\phi}_N(t):=\tilde{\phi}_N(t,(x_{N,1},\ldots,x_{N,1}))$ 
of (\ref{renorming-stationary-dgl}) and the associated normalized empirical measures $\tilde{\mu}_{N,t}$. Then
$$\lim_{t\to\infty}\lim_{N\to\infty}\tilde{\mu}_{N,t}= \lim_{N\to\infty}\lim_{t\to\infty}\tilde{\mu}_{N,t}=\mu_{sc,\sqrt{2/\lambda}}.$$
\end{corollary}

\begin{proof} Corollary \ref{special-solution} and (\ref{renorming-stationary}) yield
 $\lim_{t\to\infty}\tilde{\phi}_N(t)=\phi_N(1/(2\lambda),0)=\sqrt{1/\lambda}\cdot z$. Hence, by Example \ref{semicircle} and
 Corollary \ref{classical-limit-hermite-ns},
$\lim_{N\to\infty}\lim_{t\to\infty}\tilde{\mu}_{N,t}=\mu_{sc,\sqrt{2/\lambda}}$.

On the other hand, if we use the moments  $\tilde{S}_{N,l}$ ($l\ge0$) of the empirical measures $\tilde{\mu}_{N,t}$, 
we see from the space-time transformation (\ref{renorming-stationary}) and
 Lemma \ref{recurrence-sl-hermite-det}  that
\begin{eqnarray*}
\lim_{N\to\infty} \tilde{S}_{N,l}&=&\lim_{N\to\infty}\frac{1}{N^{l/2+1}}\sum_{i=1}^N\tilde{\phi}_{N,i}(t)^l\\
&=&\lim_{N\to\infty}\frac{1}{N^{l/2+1}}\sum_{i=1}^N\phi_{N,i}(\frac{1-e^{-2\lambda t}}{2\lambda}, e^{-\lambda t}(x_{N,1},\ldots,x_{N,1}))^l\\
&=& c_l(\frac{1-e^{-2\lambda t}}{2\lambda}, e^{-\lambda t}x),
\end{eqnarray*}
where $c_l(t,x)$ is a polynomial in $t$ and $x$ by the proof of  Lemma \ref{recurrence-sl-hermite-det}.
Hence
$$\lim_{t\to\infty}\lim_{N\to\infty} \tilde{S}_{N,l}=c_l(\frac{1}{2\lambda},0).$$
If we use the recurrence relations for the $ c_l$  in
 the proof of  Lemma \ref{recurrence-sl-hermite-det} together with Example \ref{semicircle}, we obtain
$\lim_{t\to\infty}\lim_{N\to\infty}\tilde{\mu}_{N,t}=\mu_{sc,\sqrt{2/\lambda}}.$
\end{proof}

\section{The Semicircle law for Bessel processes of type A}

Now we consider Bessel processes $( X_{t,k})_{t\ge0}$ on the Weyl chambers $C_N^A$ for the root systems $A_{N-1}$
which satisfy the SDE 
\begin{equation}\label{SDE-A}
 dX_{t,k}^i = dB_t^i+ k\sum_{j\ne i} \frac{1}{X_{t,k}^i-X_{t,k}^j}dt \quad\quad(i=1,\ldots,N).
\end{equation}
with  an $N$-dimensional Brownian motion $(B_t^1,\ldots,B_t^N)_{t\ge0}$.
By \cite{GM} (see also \cite{Sch} for a related situation)
we know that for $k\geq 1/2$ and all starting points $x\in C_N^A$, (\ref{SDE-A}) admits 
an a.s.~solution $( X_{t,k})_{t\ge0}$  which does not hit the boundary of $C_N^A$ 
for $t>0$ almost surely, even if $x$ is on  the boundary of $C_N^A$.
 In the following we only consider this regular case $k\geq 1/2$.

Under convergence conditions on the starting points as in Section 2 for $N\ge2$, we
now  derive limit theorems for the moments of the associated empirical measures
\begin{equation}\label{unscaled-empirical-A}
\mu_{N,t}:= \frac{1}{N}(\delta_{X_{t,k}^1/\sqrt N}+\ldots +\delta_{X_{t,k}^N/\sqrt N})
\end{equation}
for $t\ge0$ and $N\to\infty$. For this, it will be convenient 
 also to study the  renormalized processes $(\tilde X_{t,k}:=X_{t,k}/\sqrt k)_{t\ge0}$ which satisfy the SDE
\begin{equation}\label{SDE-A-normalized}
d\tilde X_{t,k}^i =\frac{1}{\sqrt k}dB_t^i + \sum_{j\ne i} 
 \frac{1}{\tilde X_{t,k}^i-\tilde X_{t,k}^j}dt\quad\quad(i=1,\ldots,N),
\end{equation}
which agrees, for $k=\infty$, with the ODE (\ref{basic-ode-a}). We also study the  empirical measures
\begin{equation}\label{SDE-empirical-a}
\tilde\mu_{N,t}:= \frac{1}{N}(\delta_{\tilde X_{t,k}^1/\sqrt N}+\ldots +\delta_{\tilde X_{t,k}^N/\sqrt N}).
\end{equation}
Denote the $l$-th moment $(l\in\mathbb N_0$) of  $\tilde\mu_{N,t}$ by
$$S_{N,l}(t):= \int_{\mathbb R} y^l \> d\tilde\mu_{N,t}(y) = \frac{1}{N^{l/2+1}}\sum_{i=1}^N (\tilde X_{t,k}^i)^l.$$
We will show that for all $l$ the moments $S_{N,l}(t)$  converge for $N\to\infty$ to the numbers
$c_l(t)$ of Lemma \ref{recurrence-sl-hermite-det} independent of $k\in[1/2,\infty]$.
The proof of this fact will be based on some induction  which even leads to a slightly more general convergence statement.

    To state this result, we need some notation about partitions. Let $\cal P$ the set of all partitions with $N$ components consisting of all
    $\lambda=(\lambda_1,\ldots,\lambda_N)\in\mathbb N_0^N$ with $\lambda_1\ge \lambda_2\ge\ldots\geq\lambda_N$. For $\lambda\in \cal P$,
    let $|\lambda|=\lambda_1+\ldots+\lambda_N$ its weight and $L(\lambda):=\max\{k:\> \lambda_k\ge1\}$ its length. We also consider the
    symmetric monomials
    $$m_\lambda(x):= \sum_{\pi\in S_N} x^{\pi(\lambda)}:= \sum_{\pi\in S_N} x_1^{\lambda_{\pi(1)}}x_2^{\lambda_{\pi(2)}}\cdots x_N^{\lambda_{\pi(N)}}$$
    for $x=(x_1,\ldots,x_N)\in\mathbb R^N$ where the sum runs over the symmetric group $S_N$ which acts on vectors in the obvious way.

\begin{lemma}\label{general-limit-expectations-a}
Let  $(x_N)_{N\ge1}:=(x_{N,n})_{N\ge1, 1\le n\le N}\subset\mathbb R$  be a family of starting numbers with $x_{N,n-1}\ge x_{N,n}$
for $2\le n\le N$, 
for which $$\lim_{N\to\infty} \frac{m_\lambda(x_N)}{N!\cdot N^{|\lambda|/2}}<\infty$$
exists for all  $\lambda\in \cal P$.
Let $k\in[1/2,\infty]$, and consider for $N\ge2$ the renormalized Bessel processes  $(\tilde X_{t,k})_{t\geq 0}$ 
with  start in $(x_{N,1}, \ldots,x_{N,n})\in C_N^A$. Then, for all $\lambda\in \cal P$, the limits
$$\lim_{N\to\infty} \frac{E(m_\lambda(\tilde X_{t,k}))}{N!\cdot N^{|\lambda|/2}}$$
exist locally uniformly in $t$ and are independent from $k$.
\end{lemma}

\begin{proof} We prove this statement by induction on  $|\lambda|$.

For $| \lambda|=0$ we have $\lambda=0$ and $m_\lambda(x)=N!$ which yields the claim here.
For $| \lambda|=1$ we have $\lambda=(1,0,\ldots,0)$ and $m_\lambda(x)=(N-1)!\cdot(x_1+\ldots+ x_N)$.
Thus, as
$$\sum_{i=1}^N\tilde X_{t,k}^i = \sum_{i=1}^N x_{N,i} + \frac{1}{\sqrt k}\sum_{i=1}^N B_t^i$$
by (\ref{SDE-A-normalized}), the claim also follows.

Now let  $\lambda\in \cal P$ with  $| \lambda|\ge 2$, and
we assume that the statement is already shown for partitions with weight at most $| \lambda|-1$.
 It\^{o}'s formula and (\ref{SDE-A-normalized})
yield
\begin{eqnarray}\label{SDE-Moments-A-allg}
  m_\lambda(\tilde X_{t,k})&=&
  m_\lambda(x_{N})+\frac{1}{\sqrt{k}}\sum_{i=1}^N\int_0^t \frac{ dm_\lambda}{dx_i}(\tilde X_{s,k})\>  dB_s^i
  \\
  && +\int_0^t\left(\sum_{i=1}^N \frac{1}{2k}\frac{ d^2m_\lambda}{dx_i^2}(\tilde X_{s,k})
+\sum_{i=1}^N \sum_{j\neq i} \frac{ dm_\lambda}{dx_i}(\tilde X_{s,k}) \cdot\frac{1}{\tilde X_{s,k}^i-\tilde X_{s,k}^j}
\right) ds.\nonumber
\end{eqnarray}

We now claim that the diffusion parts $\frac{1}{\sqrt{k}}\sum_{i=1}^N\int_0^t \frac{ dm_\lambda}{dx_i}(\tilde X_{s,k})\>  dB_s^i$
of (\ref{SDE-Moments-A-allg}) are martingales and thus
\begin{equation}\label{centered-parts}
E\left(\sum_{i=1}^N\int_0^t\frac{ dm_\lambda}{dx_i}(\tilde X_{s,k}) \> dB_s^i\right)=0 \quad\quad ( t\ge0).
\end{equation}
 To prove this we observe that
 \begin{equation}\label{brownian-rep}
   \sum_{i=1}^N\int_0^t \frac{ dm_\lambda}{dx_i}(\tilde X_{s,k}) dB_s^i\stackrel{d}{=}
   \int_0^t\sqrt{\sum_{i=1}^N(\frac{ dm_\lambda}{dx_i}(\tilde X_{s,k}))^{2}}\, d\tilde{B}_s
 \end{equation}
for some one-dimensional Brownian motion $\tilde{B}$ by the L\'{e}vy characterization of
the  one-dimensional Brownian motion. We now analyze the integrand of the RHS of (\ref{brownian-rep}).
As $\sum_{i=1}^N(\frac{ dm_\lambda}{dx_i})^{2}$ is a symmetric polynomial,
we may write it as a polynomial expression with rational coefficients of polynomials of the form $p_n(x)=\sum_{i=1}^Nx_i^n$.
Furthermore, for $l\geq 1$, we have
$$\sum_{i=1}^N x_i^{2l}\leq\Bigl(\sum_{i=1}^N x_i^2\Bigr)^l, \quad
\sum_{i=1}^N |x_i|^{2l+1}\leq \Bigl(\sum_{i=1}^N x_i^2\Bigr)^{(2l+1)/2}.$$
As
$\sum_{i=1}^N \tilde{X}_{s,k}^i= \sum_{i=1}^N B_s^i, $ and  as
$(\sum_{i=1}^N (\tilde{X}_{s,k}^i)^2)_{s\ge0}$ is a classical squared Bessel process by
$$\sum_{i=1}^N (\tilde X_{t,k}^i)^2 = \sum_{i=1}^N x_{N,i}^2 + \frac{2}{\sqrt k}\sum_{i=1}^N \int_0^t\tilde X_{s,k}^i dB_s^i
+ (N(N-1)+\frac{N}{k})t,$$
we see that we may bound $\sum_{i=1}^N(\frac{ dm_\lambda}{dx_i}(\tilde X_{s,k}))^{2}$ by
a polynomial in a one-dimensional Brownian motion and a one-dimensional squared Bessel process.
With standard results on  the It\^{o} integral, this
readily yields the claim and  (\ref{centered-parts}).

We now turn to the second term in the RHS of (\ref{SDE-Moments-A-allg}). We here use that
$$\sum_{i=1}^N\frac{ d^2m_\lambda}{dx_i^2}(x)=\sum_{i=1}^N\sum_{\pi\in S_N} \lambda_{\pi(i)}(\lambda_{\pi(i)}-1)x^{\pi(\lambda)-2e_i}$$
with the $i$-th unit vector $e_i=(0,\ldots,0,1,0,\ldots,0)\in\mathbb R^N$. Notice that this expression is a
symmetric polynomial which is homogeneous of order $|\lambda|-2$ and thus a linear combination of the symmetric monomials
$m_{\tilde \lambda}$ with $\tilde \lambda\in \cal P$ with $|\tilde \lambda|=|\lambda|-2$. Moreover, the coefficients in this 
linear combination depend only on $\lambda_1,\ldots,\lambda_{L(\lambda)}$ and not on $N\ge L(\lambda)$. Hence, by our induction assumption,
\begin{equation}\label{independence-reason}
\frac{1}{N!\cdot N^{|\lambda|/2}} E\left( \int_0^t\sum_{i=1}^N \frac{1}{2k}\frac{ d^2m_\lambda}{dx_i^2}(\tilde X_{s,k}) \> ds\right)   =O(1/N).
\end{equation}

We now turn to the 3rd term in the RHS of (\ref{SDE-Moments-A-allg}). We here first notice that for $u,v\in\mathbb R$, $a,b\in\mathbb N_0$,
$(u^av^b-u^bv^a)/(u-v)$ is some polynomial $p_{a,b}(u,v)$ of degree $a+b-1$.
With our notations and the transposition $(i,j)\in S_N$ we here have
\begin{align}\label{fracction-term}
\sum_{i=1}^N \sum_{j\neq i} &\frac{ dm_\lambda}{dx_i}(x) \cdot\frac{1}{x_i-x_j}
=\sum_{\pi\in S_N}\sum_{i=1}^N\sum_{j\neq i}  \lambda_{\pi(i)} \frac{ x^{\pi(\lambda)-e_i}}{x_{i}-x_{j}}\notag\\
&=\frac{1}{2}\sum_{\pi\in S_N} \sum_{i=1}^N \sum_{j\neq i}\frac{\lambda_{\pi(i)}x^{\pi(\lambda)-e_i}- \lambda_{\pi(j)}x^{\pi(\lambda)-e_j}}{x_{i}-x_{j}}\notag\\
&=\frac{1}{2}\sum_{\pi\in S_N}\sum_{i=1}^N \sum_{j\neq i} \frac{\lambda_{\pi(i)}x^{\pi(\lambda)-e_i}- \lambda_{\pi(i)}x^{\pi\circ(i,j)(\lambda)-e_j}}{x_{i}-x_{j}}\notag\\
&=\frac{1}{2}\sum_{\pi\in S_N}\sum_{i=1}^N \sum_{j\neq i} \lambda_{\pi(i)} p_{\lambda_{\pi(i)}-1,\lambda_{\pi(j)}-1}(x_i,x_j)\cdot  (x^{\pi(\lambda)})_{i,j}
\end{align}
where $(x^{\pi(\lambda)})_{i,j}$ means the product as in the definition of the $m_\lambda$ above where factors involving $x_i,x_j$ are cancelled,
as these parts are handled separately in the polynomial $p_{\lambda_{\pi(i)}-1,\lambda_{\pi(j)}-1}(x_i,x_j)$.
Please notice that the third $=$ in (\ref{fracction-term}) follows by an obvious
substitution in the summation over $\pi\in S_N$ in the second summand.

The expression (\ref{fracction-term}) obviously   is a symmetric polynomial which is homogeneous of
order $|\lambda|-2$ and thus a linear combination of the symmetric monomials
$m_{\tilde \lambda}$ with $\tilde \lambda\in \cal P$ with $|\tilde \lambda|=|\lambda|-2$. Moreover, if we analyze the coefficients $c_{\tilde \lambda}$
of the
$m_{\tilde \lambda}$
of this linear combination (w.l.o.g. take for simplicity the product $x^{\tilde \lambda}$),
we see that $c_{\tilde \lambda}/N$ converges for $N\to\infty$. Hence, by our induction assumption,
\begin{equation}\label{third-term-con}
  \frac{1}{N!\cdot N^{|\lambda|/2}} E\left( \int_0^t\sum_{i=1}^N \sum_{j\neq i}
  \frac{ dm_\lambda}{dx_i}(\tilde X_{s,k}) \cdot\frac{1}{\tilde X_{s,k}^i-\tilde X_{s,k}^j}
\right) ds.
\end{equation}
converges for $N\to\infty$. A combination of (\ref{SDE-Moments-A-allg}),
(\ref{independence-reason}),  (\ref{fracction-term}), and (\ref{third-term-con})
now completes our induction.
\end{proof}

\begin{remark} The proof of Lemma \ref{general-limit-expectations-a} shows that for fixed $k$ and $\lambda$, the limit in 
Lemma \ref{general-limit-expectations-a} has order $O(1/N)$.
  \end{remark}

Lemma \ref{general-limit-expectations-a} has the following application to the moments $S_{N,l}(t)$:
    
\begin{corollary}\label{convergence-in-expectation-a}
    Let  $(x_{N,n})_{N\ge1, 1\le n\le N}\subset\mathbb R$  be starting numbers with $x_{N,n-1}\ge x_{N,n}$
for $2\le n\le N$,   for which  the convergence condition in Lemma \ref{recurrence-sl-hermite-det} holds. 
Let $k\in[1/2,\infty]$, and let for $N\ge2$,  $(\tilde X_{t,k})_{t\geq 0}$ 
 the renormalized Bessel processes 
 starting in $(x_{N,1}, \ldots,x_{N,n})\in C_N^A$. Then, for $l\in\mathbb N_0$ and $c_l(t)$ from Lemma  \ref{recurrence-sl-hermite-det},
 $$E(S_{N,l}(t))\to c_l(t) \quad\quad\text{for}\quad N\to\infty$$
 \end{corollary}

\begin{proof} As the polynomials $m_\lambda$ in Lemma \ref{general-limit-expectations-a} are polynomials in the polynomials
  $\sum_{i=1}^N x_i^l$ for $l\le|\lambda|$ with coefficient independent of $N$ for $N\ge L(\lambda)$,
 the convergence condition in Lemma \ref{recurrence-sl-hermite-det} implies that the convergence 
condition in Lemma \ref{general-limit-expectations-a} holds.
As $S_{N,l}(t)=\frac{1}{N!N^{l/2}} m_{(l,0,\ldots,0)}(\tilde X_{t,k})$, Lemma \ref{general-limit-expectations-a} implies that the
  expectations converge, and that the limit is independent from $k\in[1/2,\infty]$. 
Lemma \ref{recurrence-sl-hermite-det} now completes the proof.
\end{proof}

 Corollary \ref{convergence-in-expectation-a} can be extended to an a.s. result:

\begin{theorem}\label{semicircle-A}
Let  $(x_{N,n})_{N\ge1, 1\le n\le N}\subset\mathbb R$  be starting numbers with $x_{N,n-1}\ge x_{N,n}$
for $2\le n\le N$,   for which the  condition in Lemma \ref{recurrence-sl-hermite-det}   holds.
Fix $k\geq 1/2$, and, for $N\ge2$, the renormalized Bessel processes  $(\tilde X_{t,k})_{t\geq 0}$ 
  starting in $(x_{N,1}, \ldots,x_{N,n})\in C_N^A$. 
Then, 
\begin{equation}\label{limit-A}
S_{N,l}(t)\to c_l(t) \quad\quad\text{for}\quad N\to\infty
\end{equation}
almost surely for all  $l\ge0$ and locally uniformly for  $t\in[0,\infty[$.
\end{theorem}

\begin{proof}
We first recapitulate from Corollary \ref{convergence-in-expectation-a} that for all $l$,
  \begin{equation}\label{limit-expectations}
 E\left(\frac{1}{N^{l/2+1}}\sum_{i=1}^N(\tilde X_{t,k}^i)^l\right)=c_l(t)+o(1) \quad(N\to\infty)
\end{equation}
  locally uniformly in $t\ge0$. Moreover, by It\^{o}'s formula and (\ref{SDE-A-normalized}),
\begin{eqnarray}\label{SDE-Moments-A}
\sum_{i=1}^N(\tilde X_{t,k}^i)^l&=&\sum_{i=1}^N x_{N,i}^l + \frac{l}{\sqrt k} \sum_{i=1}^N\int_0^t(\tilde X_{s,k}^i)^{l-1} dB_s^i\\
&& +\int_0^t\left(l \sum_{i=1}^N \sum_{j\neq i}
\frac{(\tilde X_{s,k}^i)^{l-1}}{\tilde X_{s,k}^i-\tilde X_{s,k}^j}+
\sum_{i=1}^N \frac{l(l-1)}{2k}(\tilde X_{s,k}^i)^{l-2}\right) ds.\nonumber
\end{eqnarray}
This together with the Chebychev and Burkholder-Davis-Gundy inequality shows that
 for all $l\in \mathbb N_0$, $\epsilon>0$ and $T>0$ and some universal constant $c>0$,
\begin{align}
\sup_{s\leq T}P\Bigl(&\Bigl|\frac{1}{N^{l/2+1}}\Bigl(\sum_{i=1}^N\Bigl((\tilde{X}_{s,k}^i)^l-\sum_{i=1}^N x_{N,i}^l \notag\\
&  -\int_0^t\left(l \sum_{i=1}^N \sum_{j\neq i}
\frac{(\tilde X_{s,k}^i)^{l-1}}{\tilde X_{s,k}^i-\tilde X_{s,k}^j}+
\sum_{i=1}^N \frac{l(l-1)}{2k}(\tilde X_{s,k}^i)^{l-2}\right) ds \Bigr) \Bigr|>\epsilon\Bigr)\notag\\
&\leq
 \frac{1}{\epsilon^2} E\Bigl(\sup_{s\leq T}\Bigl(\frac{1}{N^{l/2+1}} \frac{l}{\sqrt{k}}\sum_{i=1}^N\int_0^t(\tilde{X}_{s,k}^i)^{l-1} dB_s^i
\Bigr)^2\Bigr)\notag\\
&\leq \frac{c}{\epsilon^2}E\Bigl(\frac{l^2}{k}\frac{1}{N^{l+2}}\int_0^T \sum_{i=1}^N(\tilde{X}_{s,k}^i)^{2l-2} ds\Bigr)\to 0
\notag\end{align}
as $N\to\infty$. Furthermore, by (\ref{limit-expectations}),
$$\sum_{N=1}^\infty E\Bigl(\frac{1}{N^2} \frac{l^2}{k}\int_0^T \frac{1}{N^l} \sum_{i=1}^N(X_{s,k}^i)^{2l-2} ds\Bigr)<\infty.$$
This and the Borel-Cantelli Lemma imply that 
\begin{align}
  \frac{1}{N^{l/2+1}}&\Bigl(\sum_{i=1}^N(\tilde X_{s,k}^i)^l-\sum_{i=1}^N x_{N,i}^l \notag\\
  &-\int_0^t\Bigl(l \sum_{i=1}^N \sum_{j\neq i}
\frac{(\tilde X_{s,k}^i)^{l-1}}{\tilde X_{s,k}^i-\tilde X_{s,k}^j}+
\sum_{i=1}^N \frac{l(l-1)}{2k}(\tilde X_{s,k}^i)^{l-2}\Bigr) ds \Bigr)\to 0
\notag\end{align}
for $N\to\infty$  almost surely.
Hence, by (\ref{SDE-Moments-A}), also
\begin{equation}\label{as-convergence-brownian-part}
  \frac{1}{N^{l/2+1}} \left(\sum_{i=1}^N\frac{l}{\sqrt{k}}\int_0^t(\tilde{X}_{s,k}^i)^{l-1} dB_s^i\right)\to 0 \quad\text{a.s..}
  \end{equation}
With this result we now prove the claim by induction on $l$. In fact, the case $l=0$ is trivial,
and the case $l=1$ follows easily from (\ref{SDE-Moments-A}) and (\ref{as-convergence-brownian-part}) for $l=1$,
and the fact that the drift part on the RHS of (\ref{SDE-Moments-A}) disappears. Moreover, for $l\ge2$, we again use 
(\ref{SDE-Moments-A}) and (\ref{as-convergence-brownian-part}), i.e., it suffices to show that
for $N\to\infty$
\begin{align}\frac{1}{N^{l/2+1}}&\left(\sum_{i=1}^N x_{N,i}^l  +\int_0^t\left(l \sum_{i=1}^N \sum_{j\neq i}
\frac{(\tilde X_{s,k}^i)^{l-1}}{\tilde X_{s,k}^i-\tilde X_{s,k}^j}+
\sum_{i=1}^N \frac{l(l-1)}{2k}(\tilde X_{s,k}^i)^{l-2}\right) ds \right)\notag\\
&=c_l(t)+o_p(\frac{1}{N^{l/2+1}})\notag
\end{align}
almost surely. But this follows easily from our induction assumption and the recurrence of the numbers $c_l(t)$ in Lemma 
\ref{recurrence-sl-hermite-det}. This yields the claim.
\end{proof}

Theorem \ref{semicircle-A}
implies that for $N\to\infty$ the moments of the empirical measures of the renormalized
 Bessel processes $(\tilde{X}_{t,k})_{t\geq 0}$ behave almost surely like the solutions of the ODEs (\ref{ODE-Hermite}),
 i.e. the case $k=\infty$. 
Hence Theorem \ref{general-free-convolution-a}
 also holds for the empirical measure associated to the renormalized Bessel processes $(\tilde{X}_{t,k})_{t\geq 0}$.
In particular, we obtain the following result from the moment convergence theorem.

\begin{theorem}\label{limit-theorem-a-final}
Let $\mu\in M^1(\mathbb R)$ be a probability measure satisfying (\ref{strong-carleman}), and let
 $(x_{N,n})_{N\ge1, 1\le n\le N}\subset\mathbb R$ 
with $x_{N, n-1}\ge x_{N, n}$ for  $2\le n\le N$ such that the associated normalized empirical measures
as in (\ref{starting-empirical-a})
tend weakly to  $\mu$ for $N\to\infty$.

 For $k\geq 1/2$ and $N\in\mathbb N$, consider the renormalized 
Bessel processes $(\tilde X_{t,k})_{t\geq 0}$  with  start in
$(x_{N,1}, \ldots,x_{N,n})\in C_N^A$.
Then, for  $t\ge0$, the empirical measures $\tilde\mu_{N,t}$ from (\ref{SDE-empirical-a}) tend weakly to 
 $\mu_{sc, 2\sqrt{t}}\boxplus \mu$ almost surely for $N\to\infty$.
\end{theorem}

\begin{remark}\label{remark-freezing}
The previous results were obtained for fixed $k>1/2$. As in \cite{AV1, AV2, V, VW}
 we may also consider the freezing regime $k\to\infty$. Here we  choose starting points
$(x_{1,N},\ldots, x_{N,N})$ in the interior of  $ C_N^A$ for $N\in\mathbb N$.
Then by \cite{AV1}, 
$$ X_{t,k}/\sqrt{k}\to \phi_{N}(t)$$
locally uniform in $t$ a.s. for $k\to\infty$ and the Bessel processes   $( X_{t,k}^N)_{t\geq 0}$
starting in $(x_{1,N},\ldots, x_{N,N})$
 where the $\phi_N$ are the solutions of (\ref{basic-ode-a}) with start in $(x_{1,N},\ldots, x_{N,N})$.
 Hence, by (\ref{recurrence-sl-hermite-det}),
$$\lim_{N\to\infty}\lim_{k\to\infty}\frac{1}{N^{l/2+1}}\sum_{i=1}^N(\frac{X_{t,k}^i}{\sqrt{k}})^l=c_l(t)$$
a.s.. On the other hand, by Theorem \ref{semicircle-A},
$$\lim_{k\to\infty}\lim_{N\to\infty}\frac{1}{N^{l/2+1}}\sum_{i=1}^N(\frac{X_{t,k}^i}{\sqrt{k}})^l=c_l(t)$$
a.s., i.e., the limits of $k$ and $N$ may be interchanged here. 
\end{remark}

We next study a stochastic analogue of  Corollary \ref{ODE-stationaer} and
 modify the SDE of the Bessel processes  by an additional drift 
 $-\lambda x$ with some constant $\lambda\in\mathbb R$, i.e.
 a component as in a classical Ornstein-Uhlenbeck setting, cf. \cite{VW}. We thus consider 
Bessel-OU processes $Y:=(Y_{t,k})_{t\ge0}$ on $C_N^A$ of type $A_{N-1}$ 
as solutions of
\begin{equation}\label{SDE-A-mr}
 dY_{t,k}^i = dB_t^i+\left( k\sum_{j\ne i} \frac{1}{Y_{t,k}^i-Y_{t,k}^j}-\lambda  Y_{t,k}^i\right) dt \quad\quad(i=1,\ldots,N)
\end{equation}
for $k\ge1/2$
with an $N$-dimensional Brownian motion $(B_t^1,\ldots,B_t^N)_{t\ge0}$. 
For $\lambda>0$, these processes are mean  reverting ergodic process $Y=( Y_{t,k})_{t\ge0}$
 with speed of mean-reversion $\lambda$, and for  $\lambda\leq 0$,  non-ergodic.

 It\^{o}'s formula together with a time-change argument shows
 that $Y$ is a space-time transformation of the original Bessel process  $( X_{t,k})_{t\geq 0}$ (with $\lambda=0$) via 
$$Y_{t,k}=e^{-\lambda t}X_{\frac{e^{2\lambda t}-1}{2\lambda},k}.$$
 For a proof based on the generators of these diffusions see \cite{RV1}.
 Using this space-time transformation, we can immediately reformulate  Theorems \ref{semicircle-A} and
\ref{limit-theorem-a-final}
 for the  $Y_{t,k}$. Moreover, for $t\to\infty$ we obtain the following 
 analogue to Corollary \ref{ODE-stationaer}.

\begin{corollary}
Consider starting sequences
$(x_{1,N},\ldots, x_{N,N})\subset C_N^A$ for $N\in\mathbb N$ 
 as in Lemma \ref{recurrence-sl-hermite-det}, and let $\lambda>0$ and $k\geq 1/2$. For $N\in\mathbb N$ let
 $( Y_{t,k}^N)_{t\ge0}$ be the solution of (\ref{SDE-A-mr}) on $C_N^A$ with start in $(x_{1,N},\ldots, x_{N,N})$.
Then  the  empirical measures
 $$\tilde{\mu}_{N,t}:= \frac{1}{N}(\delta_{Y_{t,k}^1/\sqrt N}+\ldots +\delta_{Y_{t,k}^N/\sqrt N})$$ 
satisfy
$\lim_{N\to\infty}\lim_{t\to\infty}\tilde{\mu}_{N,t}=\lim_{t\to\infty}\lim_{N\to\infty}\tilde{\mu}_{N,t}= \mu_{sc,\sqrt{2/\lambda}}.$
\end{corollary}

\section{Zeroes of Laguerre polynomials and the Marchenko-Pastur law}

In this section we transfer the approach of Section 2 for the empirical measures of
 the zeroes of the Hermite polynomials to Laguerre polynomials.
We start with the appropriate ODE:

\begin{ODE}\label{ODE-Laguerre}
Let $\nu\in]0,\infty[$ and $N\ge 2$. On the interior of the closed Weyl chamber
$$C_N^B:=\{x\in \mathbb R^N: \quad x_1\ge x_2\ge\ldots\ge x_N\ge0\}\subset \mathbb R^N$$
of type B, we consider the  $\mathbb R^N$-valued function
$$H_\nu(x):=
\Bigl( \sum_{j\ne1} \frac{2x_1}{x_1^2-x_j^2}+  \frac{\nu}{x_1},
\ldots,\sum_{j\ne N} \frac{2x_N}{x_N^2-x_j^2}+  \frac{\nu}{x_N} \Bigr).$$
It is shown in \cite{VW2}
 that for each initial condition $x_0\in C_N^B$, the ODE
\begin{equation}\label{basic-ode-b}
\frac{dx}{dt}(t) =H_\nu(x(t)), \quad\quad x(0)=x_0
\end{equation}
has a unique solution for all $t\ge0$ in the sense that $[0,\infty[\to  C_N^B$, $t\mapsto x(t)$
is continuous where $x(t)$ is in the interior of  $C_N^B$ and solves the ODE in  (\ref{basic-ode-b}) for $t>0$.
For the ODE  (\ref{basic-ode-b})  see also  \cite{AV1, VW}. 
We  denote  solutions of the ODE (\ref{basic-ode-a})  by
 $\phi_N:=(\phi_{N,1},\ldots,\phi_{N,N})$  where we suppress the dependence on $x_0$ and $\nu$.
\end{ODE}

For  $x_0=0\in C_N^b$, the solution of (\ref{basic-ode-b})
  can be expressed in terms of the 
zeroes of the  Laguerre polynomial $L_N^{(\nu-1)}$   where the $(L_N^{(\nu-1)})_{N\ge 0}$ are orthogonal w.r.t.
 the density  $e^{-x}x^{\nu-1}$ on $]0,\infty[$ as in \cite{S}. In fact, by
\cite{AV1} we have:

\begin{lemma}\label{special-solution-B}
Let $\nu>0$ and  $z_1^{(\nu-1)}>\ldots>z_N^{(\nu-1)}>0$ be the ordered zeros of $L_N^{(\nu-1)}$.
Then, for $z:=(z_1^{(\nu-1)},\ldots,z_N^{(\nu-1)}) \in C_N^B$ and any
  $c\ge 0$,
$$\phi_N(t):= \sqrt{2t+c^2}\cdot (\sqrt{z_1^{(\nu-1)}},\ldots,  \sqrt{z_N^{(\nu-1)}}) \quad\quad (t\ge0)$$
is a solution of the ODE (\ref{basic-ode-b}).
\end{lemma}

We now study the empirical measures of  solutions of  (\ref{basic-ode-b}).
We proceed as in Section 2 and derive recurrence relations for the associated empirical moments.
However,  this works for the even moments only. For this reason we take squares in all 
components of  $\phi_N$ at some stages.
 As we are working here on the Weyl chambers $C_N^B$, no information is lost by taking these squares. 
Moreover, having the well-known
  relation $H_{2N}(x)=d_N L_N^{(-1/2)}(x^2)$ ($d_N>0$  some constants)  in mind
(see (5.6.1) of \cite{S}), we normalize our  empirical measures with $2N$ instead of $N$.

We now choose a family $(x_{N,k})_{1\le k\le N}\in [0,\infty[$ with $(x_{N,1},\ldots,x_{N,N})\in C_N^B$  for  $N\ge2$. 
We then form the associated solutions  $\phi_N(t)$ of  (\ref{basic-ode-b}) with start in $(x_{N,1},\ldots,x_{N,N})$ for  $N\ge2$, and
introduce  the  normalized
 empirical measures
\begin{equation}\label{empirical-gen-lagu}
\mu_{N,t}^2:= \frac{1}{N}(\delta_{\frac{\phi_{N,1}(t)^2}{2N}}+\ldots +\delta_{\frac{\phi_{N,N}(t)^2}{2 N}})
\end{equation}
 for $t\ge0$.
Denote the $l$-th moment $(l\in\mathbb N_0$) of $\mu_{N,t}^2$ by
$$S_{N,l}(t):= \int_{\mathbb R} y^l \> \mu_{N,t}(y) = \frac{1}{2^lN^{l+1}}(\phi_{N,1}(t)^{2l}+\ldots+\phi_{N,N}(t)^{2l}).$$
Then $S_{N,0}(t)=1$.  Moreover, by  (\ref{basic-ode-b}),
\begin{equation}\label{moment-1-lagu}
\frac{d}{dt}S_{N,1}(t)=\frac{2}{2N^{2}}\Bigl(\sum_{i,j=1; i\ne j}^N 
\frac{2\phi_{N,i}(t)^{2}}{\phi_{N,i}(t)^{2}-\phi_{N,j}(t)^{2}} + \nu N \Bigr)= \frac{N(N+\nu-1)}{N^{2}}.
\end{equation}
Furthermore, for $l\ge2$,
\begin{align}
\frac{d}{dt}&S_{N,l}(t)= \frac{2}{2^lN^{l+1}}\Bigl(\sum_{i,j=1; i\ne j}^N 
\frac{2\phi_{N,i}(t)^{2l}}{\phi_{N,i}(t)^{2}-\phi_{N,j}(t)^{2}} + \nu\sum_{i=1}^N\phi_{N,1}(t)^{2l-2} \Bigr)\notag \\
=& \frac{2}{2^lN^{l+1}}\Bigl(2\sum_{1\le i<j\le N}\frac{\phi_{N,i}(t)^{2l}-\phi_{N,j}(t)^{2l}}{\phi_{N,i}(t)^{2}-\phi_{N,j}(t)^{2}}
 + \nu\sum_{i=1}^N\phi_{N,1}(t)^{2l-2} \Bigr)\notag \\
=& \frac{2}{2^lN^{l+1}}\Bigl(\sum_{i,j=1; i\ne j}^N \Bigl( \phi_{N,i}(t)^{2(l-1)}+\phi_{N,i}(t)^{2(l-2)}\phi_{N,j}(t)^{2}+\ldots+
\phi_{N,j}(t)^{2(l-1)}\Bigr)\notag \\
&\quad + \nu\sum_{i=1}^N\phi_{N,1}(t)^{2l-2} \Bigr)\notag.
\notag\end{align}
As 
$$\frac{1}{2^lN^{l+1}}\sum_{i,j=1; i\ne j}^N \phi_{N,i}(t)^{2(l-k)}\phi_{N,j}(t)^{2k}= \frac{1}{2}(S_{N,l-k-1}(t)S_{N,k}(t)-\frac{S_{N,l-1}(t)}{N})$$
for $k=1,2,\ldots,l-2$, and as 
$$\frac{1}{2^lN^{l+1}}\sum_{i,j=1; i\ne j}^N \phi_{N,i}(t)^{2(l-1)}=\frac{N-1}{2N}S_{N,l-1}(t),$$
we obtain
\begin{equation}\label{moment-recurrence-lagu}
\frac{d}{dt}S_{N,l}(t)= l\Bigl(\frac{2N+\nu-l}{N}S_{N,l-1}(t)+
\sum_{k=1}^{l-2} S_{N,l-1-k}(t)S_{N,k}(t)\Bigr).
\end{equation}

Similar to the Hermite case in Section 2, 
these equations show that the limits $c_l(t):=\lim_{N\to\infty}S_{N,l}(t)$ ($l\in\mathbb N_0$)
exist and satisfy the following  recurrence relation:

\begin{lemma}\label{recurrence-sl-laguerre-det}
Let $(x_{N,k})_{1\le k\le N}\subset \mathbb R$ be  starting sequences  such that for all  $l\in \mathbb N_0$,
$$c_l(0):=\lim_{N\to\infty} S_{N,l}(0)=\lim_{n\to\infty}  \frac{1}{2^lN^{l+1}}(x_{N,1}^{2l}+\ldots+x_{N,N}^{2l})<\infty$$
exists. Assume that $\nu=\nu(N)$ depends on $N$ with
$$\lim_{N\to\infty} \frac{\nu(N)}{N}=\nu_0\geq 0.$$
Then
for  $l\in \mathbb N_0$,
$$c_l(t):=\lim_{N\to\infty} S_{N,l}(t)$$
exists locally uniformly in $t\in[0,\infty[$ and satisfies the recurrence relation
\begin{equation}\label{recurrence-laguerre}
c_0(t)=1 \quad\text{and}\quad
c_l(t)= c_l(0)+l \nu_0 \int_0^t c_{l-1}(s) ds + l\int_0^t 
\sum_{k=0}^{l-1} c_{l-1-k}(s)c_{k}(s)\> ds \quad(l\ge 1).
\end{equation}
\end{lemma}

As in Section 2 we now show that under mild conditions on the  starting sequences,
 the $c_l(t)$ are the moments of some
unique probability measures  which can be described via free probability.

We first consider the case $\nu_0=0$ which includes the case that all $\nu$ are independent of $N$.
This special case can be easily reduced to Section 2. We here denote the image of some probability measure 
$\mu$ under some continuous mapping $f$ by  $f(\mu)$. We use this notation in particular for the maps
 $x\mapsto|x|$ and  $x\mapsto x^2$ and write $|\mu|$ and $\mu^2$. Moreover, for a probability measure $\mu$ on $[0,\infty[$, let 
 $\mu_{even}$ the unique even  probability measure on $\mathbb R$ with $|\mu_{even}|=\mu$.

\begin{theorem}\label{general-free-convolution-b-null}
Let $\mu\in M^1([0,\infty[)$ be a probability measure satisfying the moment condition 
(\ref{strong-carleman}). Let
 $(x_{N,n})_{N\ge1, 1\le n\le N}\subset\mathbb [0,\infty[$ 
with $(x_{N,1},\ldots, x_{N, N})\in C_N^B$ for  $ N\ge2$ such that the normalized empirical measures
\begin{equation}\label{starting-empirical-a1}
\mu_{N,0}:= \frac{1}{N}(\delta_{x_{N, 1}/\sqrt {2N}}+\ldots \delta_{x_{N, N}/\sqrt {2N}})
\end{equation}
tend weakly to  $\mu$ for $N\to\infty$. If we form the associated solutions
$\phi_{N}(t)$  of (\ref{basic-ode-b}) and 
the   normalized
 empirical measures
$$\mu_{N,t}:= \frac{1}{N}(\delta_{\phi_{N,1}(t)/\sqrt {2N}}+\ldots +\delta_{\phi_{N,N}(t)/\sqrt {2N}}) \quad\quad(t\ge0),$$
then  for 
$t\in[0,\infty[$, 
the  $\mu_{N,t}$ tend weakly to  $|\mu_{sc, 2\sqrt{t}}\boxplus \mu_{even}|$.
\end{theorem}

\begin{proof} 
For $N\ge2$ and   $(x_{N,1},\cdots, x_{N,N})\in C_N^B$ we define
\begin{equation}\label{start-laguerre}
(y_{2N,1},\cdots, y_{2N,2N})=(x_{N,1},\cdots, x_{N,N}, -x_{N,N},\cdots, -x_{N,1})\in C_{2N}^A.
\end{equation}
Clearly, the associated empirical measures 
$$\mu_{N,even}:=\frac{1}{2N}(\delta_{y_{2N,1}/\sqrt{2N} }+\ldots +\delta_{y_{2N,1}/\sqrt{2N} })$$
tend weakly to $\mu_{even}$. Moreover, all odd moments of $\mu_{even}$ disappear,
 and the even moments of  $\mu_{even}$ and $\mu$ are equal. In particular,  $\mu_{even}$ also satisfies
 the moment condition (\ref{strong-carleman}).

We now consider the even  measures $\mu_{sc, 2\sqrt{t}}\boxplus \mu_{even}$ 
whose odd moments disappear, and whose even moments satisfy (\ref{recurrence-herm}) by  Section 2.
We obtain from  (\ref{recurrence-herm}) for these even moments and from (\ref{recurrence-laguerre}),
that for all $l\ge0$, the $c_l(t)$ from Lemma \ref{recurrence-sl-laguerre-det} are just the
 $2l$-th moments of $\mu_{sc, 2\sqrt{t}}\boxplus \mu_{even}$. This, Carleman's condition (\ref{carleman})
for $(\mu_{sc, 2\sqrt{t}}\boxplus \mu_{even})^2$, Lemma \ref{recurrence-sl-laguerre-det},
 and moment convergence  now  imply that 
the  probability measures $\mu_{N,t}^2$ tend weakly to
$(\mu_{sc, 2\sqrt{t}}\boxplus \mu_{even})^2$. This implies the claim.
\end{proof}

If we start the ODE (\ref{basic-ode-b}) in $x_0=0$,  Theorem 
\ref{general-free-convolution-b-null} and Lemma \ref{special-solution-B} lead to the following 
well  known result for the empirical measures of the  zeroes $z_{1,N}^{(\nu-1)},\ldots,z_{N,N}^{(\nu-1)}$
of $L_N^{(\nu-1)}$; see e.g. \cite{G}:

\begin{corollary} Let $\nu>0$ be fixed. The empirical measures
$$\mu_{N}:= \frac{1}{N}(\delta_{z_{1,N}^{(\nu-1)}/ (4N)}+\ldots +\delta_{z_{N,N}^{(\nu-1)}/ (4N)})$$
tend weakly to the beta distribution $\beta(1/2,3/2)$ on $[0,1]$ with the density
$$\frac{2}{\pi}t^{-1/2}(1-t)^{1/2}{\bf 1}_{]0,1[}(t).$$
\end{corollary}

We now turn to the general case $\nu_0\geq 0$.
 We proceed similarly as in the Hermite case and 
  deduce first a PDE for the Stieltjes transform $G$. This  PDE will then be transformed 
into a PDE for the R-transform. Then we combine our results for $\nu_0=0$ and the R-transform
 for an appropriately parametrized Marchenko-Pastur distribution to identify the limit measure
 in general.

 Recall that for the parameters $c\ge0$, $t>0$, the
Marchenko-Pastur distribution $\mu_{MP,c,t}$ is the probability measure on $[x_-,x_+]\subset[0,\infty[$ with $\mu_{MP,c,t}=\tilde{\mu}$ for $c\geq 1$ and $\mu_{MP,c,t}=(1-c)\delta_0+c \tilde{\mu}$ for $0\leq c < 1$, where
$x_\pm:=t(\sqrt{c}\pm1)^2$  and $\tilde{\mu}$ has the Lebesgue density
$$\frac{1}{2\pi x t}\sqrt{(x_+-x)(x-x_-)}\cdot{\bf 1}_{[x_-,x_+]}(x).$$
We also recall (see Exercise 5.3.27 of \cite{AGZ}), that the Marchenko-Pastur distributions
have the  R-transforms 
\begin{equation}\label{r-transform-mp}
R_{MP,c,t}(z)=\frac{ct}{1-tz}.
\end{equation}
As these R-transforms are linear in  $c$, we in particular conclude that
\begin{equation}\label{mp-product-formula} \mu_{MP,a,t}\boxplus \mu_{MP,b,t}=\mu_{MP,a+b,t}.\end{equation}
Now we proceed as in Section 2. The proof of the first step is completely analog to Proposition
\ref{general-free-convolution-a-firststep}.

\begin{proposition}\label{general-free-convolution-b-firststep}
Let $\lim_{N\to\infty} \nu(N)/N=\nu_0\ge0$ in the setting of Lemma \ref{recurrence-sl-laguerre-det}.
Then the limits
 $$c_l(t):=\lim_{N\to\infty} S_{N,l}(t)  \quad\quad (t\in[0,\infty[, l\ge0)$$
exist. Moreover, if the moment condition (\ref{strong-carleman}) holds for $\mu^2$, then
 for each 
$t\in[0,\infty[$, the sequence $(c_l(t))_{l\ge0}$ is the sequence of moments of some unique probability
measure $\mu_t^2\in M^1(\mathbb R)$ for which  (\ref{strong-carleman}) also holds. Moreover, 
the  $\mu_{N,t}^2$ tend weakly to  $\mu_t^2$ for $N\to\infty$.
\end{proposition}

We next derive PDEs 
for the  Stieltjes transforms
$$G(t,z):=G_{\mu_t^2}(z), \quad G^N(t,z):=G_{\mu_{N,t}^2}(z) \quad\quad (t\ge0, \> z\in H=\{z\in\mathbb C:\> \Im z>0\})$$
of the  measures $\mu_t^2$ and  $\mu_{N,t}^2$. 
In the setting of 
Proposition \ref{general-free-convolution-b-firststep} we now have:

\begin{proposition}\label{general-free-convolution-b-secondstep}
\begin{enumerate}
\item[\rm{(1)}] For all $N\in\mathbb N$, $t\ge0$, $z\in H$, 
$$G_t^N(t,z)= -\frac{\nu(N)}{N} G^N_z(t,z)-2 z G^N_z(t,z)G^N(t,z) -(G^N)^2(t,z) - \frac{1}{N} E^N(t,z)$$
with the error term  $E^N(t,z)$ defined below in (\ref{error-gn-b}).
\item[\rm{(2)}] Assume that in addition the moment condition  (\ref{strong-carleman}) holds for 
$\mu^2$. Then for $t\ge0$, $z\in H$, 
\begin{eqnarray}\label{pde-g-b}
G_t(t,z)&=&-\nu_0 G_z(t,z)-2 z G_z(t,z)G(t,z) -G^2(t,z)\\
G(0,z)&=&G_{\mu}(z).\nonumber
\end{eqnarray}
\end{enumerate}
\end{proposition}

\begin{proof}
 For $t\ge0$ and $z\in H$ with $|z|$ sufficiently large (depending on $N$)  we have
$$G^N(t,z)=\int_{\mathbb R} \frac{1}{z-x}\> d\mu_{N,t}(x)  = \sum_{l=0}^\infty \frac{S_{N,l}(t)}{z^{l+1}}$$
and thus
$$G^N_t(t,z)=  \sum_{l=0}^\infty \frac{1}{z^{l+1}}    \frac{d}{dt}S_{N,l}(t).$$
If we apply the recurrence relation
 (\ref{moment-recurrence-lagu}) and the start (\ref{moment-1-lagu}), we obtain
\begin{eqnarray*}
G_t^N(t,z)&=& \frac{N+\nu-1}{N}\cdot \frac{1}{z^{2}} \\ 
&&+\sum_{l=2}^\infty   \frac{1}{z^{l+1}} l\Bigl(\frac{2N+\nu-l}{N}S_{N,l-1}(t)+
\sum_{k=1}^{l-2} S_{N,l-1-k}(t)S_{N,k}(t)\Bigr)\\
&=&\sum_{l=0}^\infty  \frac{1}{z^{l+2}} (l+1)\sum_{k=0}^{l} S_{N,l-k}(t)S_{N,k}(t)+\frac{\nu}{N}\sum_{l=0}^\infty(l+1) \frac{1}{z^{l+2}}S_{N,l}(t)\\
&&-\frac{1}{N} E^N(t,z)
\end{eqnarray*}
with 
 $$E^N(t,z):=  \sum_{l=0}^\infty   \frac{(l+1)^2}{z^{l+2}}\cdot S_{N,l-2}(t).$$
Using as in the proof of (\ref{general-free-convolution-a-secondstep})
$$z\cdot G_z^N(t,z)=-\sum_{l=0}^\infty    \frac{l+1}{z^{l+1}}\cdot S_{N,l}(t) , \quad
z^2\cdot G_{zz}^N(t,z)=\sum_{l=0}^\infty  \frac{(l+1)(l+2)}{z^{l+1}}\cdot S_{N,l}(t),$$
we obtain by some simple calculation that
\begin{equation}\label{error-gn-b}
 E^N(t,z)= z G_{zz}^N(t,z)-G_z^N(t,z)
\end{equation}
As
\begin{eqnarray*}
\lefteqn{\sum_{l=0}^\infty  \frac{1}{z^{l+2}} (l+1)\sum_{k=0}^{l} S_{N,l-k}(t)S_{N,k}(t)+\frac{\nu}{N}\sum_{l=0}^\infty(l+1) \frac{1}{z^{l+2}}S_{N,l}(t)}&&\\
&=&\sum_{l=0}^\infty  \frac{1}{z^{l+2}} (l-k+1)\sum_{k=0}^{l} S_{N,l-k}(t)S_{N,k}(t)+\sum_{l=0}^\infty  \frac{1}{z^{l+2}} (k+1)\sum_{k=0}^{l} S_{N,l-k}(t)S_{N,k}(t)\\
&&-\sum_{l=0}^\infty  \frac{1}{z^{l+2}}\sum_{k=0}^{l} S_{N,l-k}(t)S_{N,k}(t)+\frac{\nu}{N}\sum_{l=0}^\infty(l+1) \frac{1}{z^{l+2}}S_{N,l}(t)\\
&=&-\frac{\nu}{N} G^N_z(t,z)-2 z G^N_z(t,z)G^N(t,z) -(G^N)^2(t,z),
\end{eqnarray*}
part (1) follows for $z\in H$ sufficiently large. 
As both sides of the equation in part (1) are analytic in $z\in H$, this equation
 holds for all  $z\in H$.

For (2) the arguments are the same as in the proof of (\ref{general-free-convolution-a-secondstep}).
\end{proof}
Now we  turn to the main result of this section:

\begin{theorem}\label{free-convolution-b}
Let $\mu\in M^1([0,\infty[)$ be a probability measure such that $\mu^2$ satisfies
 the moment condition (\ref{strong-carleman}).
Let
 $(x_{N,n})_{N\ge1, 1\le n\le N}\subset[0,\infty[$ 
with $(x_{N, 1},\ldots, x_{N, N})\in C_N^B$ such that the normalized empirical measures
\begin{equation}
\mu_{N,0}:= \frac{1}{N}(\delta_{x_{N, 1}/\sqrt{2N}}+\ldots \delta_{x_{N, N}/\sqrt{2N}})
\end{equation}
tend weakly to  $\mu$ for $N\to\infty$. For $N\ge2$ we form the  solutions
$\phi_N$  of (\ref{basic-ode-b}) with start in
$(x_{N,1},\cdots, x_{N,N})$. If
$$\lim_{N\to\infty} \frac{\nu(N)}{N}=\nu_0\geq 0,$$
then for each 
$t\in[0,\infty[$, the associated  normalized
 empirical measures
$$\mu_{N,t}:= \frac{1}{N}(\delta_{\frac{\phi_{N,1}(t)}{\sqrt{2 N}}}+\ldots +\delta_{\frac{\phi_{N,N}(t)}{\sqrt{2N}}}) \quad\quad(t\ge0),$$
tend weakly to  $\sqrt{\mu_{MP, \nu_0, t}\boxplus (\mu_{sc, 2\sqrt{t}}\boxplus \mu_{even})^2}$.
\end{theorem}

\begin{proof}
We use (\ref{relation-g-r-2}),  (\ref{G-R-1}), and (\ref{G-R-2}) in the PDE (\ref{pde-g-b})
  and obtain
$$R_t(t,G(t,z)) =\nu_0 +2R(t,G(t,z))\cdot   G(t,z)+2-1+R_z(t,G(t,z)) 
\cdot G(t,z)^2.$$
For $z$ instead of  $G(t,z)$ we arrive  at the following PDE for the R-transform:
\begin{eqnarray}\label{pde-r-b}
R_t(t,z)&=& \nu_0+1-2 z R(t,z) + z^2 R_z(t,z)\\
R(0,z)&=& R_\mu(z). \nonumber
\end{eqnarray}
By Theorem \ref{general-free-convolution-b-null} we know that 
$R_{(\mu_{sc, 2\sqrt{t}}\boxplus \mu)^2}(z)$ solves (\ref{pde-r-b}) with $\nu_0=0$. 
This, (\ref{r-transform-mp}), and
 a straight forward calculation now show that $R_{\mu_{MP,\nu_0,t}}(z)+R_{(\mu_{sc, 2\sqrt{t}}\boxplus \mu_{even})^2}(z)$ 
solves the general PDE (\ref{pde-r-b}). This implies the desired result.
\end{proof}

We finally consider Theorem \ref{free-convolution-b}
 for the case with start in zero. As by a straight forward calculation $(\mu_{sc, 2\sqrt{t}})^2=\mu_{MP, 1,t}$, 
 we here obtain the measures $\sqrt{\mu_{MP, 1+\nu_0,t}}$ as weak limits  for the empirical measures $\mu_{N,t}$. 
This result and Lemma \ref{special-solution-B} on the zeroes
 $z_1^{(\nu(N)-1)}>\ldots>z_N^{(\nu(N)-1)}>0$  of $L_N^{(\nu(N)-1)}$ for $t=1/2$ lead to the following well known 
result (see e.g.~\cite{KM2} and references there):

\begin{corollary}
If $\lim_{N\to\infty} \nu(N)/N=\nu_0\geq 0,$
then the 
 empirical measures
$$ \frac{1}{N}(\delta_{z_1^{(\nu(N)-1)}/2 N}+\ldots +\delta_{z_N^{(\nu(N)-1)}/2N}) \quad\quad(t\ge0),$$
tend weakly to  $\mu_{MP, 1+\nu_0,1/2}$.
\end{corollary}

 Theorem \ref{free-convolution-b} has
 also the following consequence:

\begin{corollary}
For all $s,t\ge0$ and $\nu_0\geq 0$,
$$\mu_{MP, \nu_0, s}\boxplus\Bigl(\mu_{sc, 2\sqrt{s}}\boxplus (\sqrt{ \mu_{MP, \nu_0+1, t}})_{even}\Bigr)^2=
\mu_{MP, \nu_0+1, s+t}.$$
\end{corollary}

\begin{proof}
Theorem \ref{free-convolution-b}  with $\mu=\delta_0$
together with the semigroup property
of our solutions of the ODEs (\ref{basic-ode-b}) imply
\begin{align}
\mu_{MP, \nu_0+1, s+t}&= \mu_{MP, \nu_0, s+t}\boxplus\Bigl(\mu_{sc, 2\sqrt{s+t}}\boxplus \delta_0\Bigr)^2\notag\\
&=\mu_{MP, \nu_0, s}\boxplus\Bigl(\mu_{sc, 2\sqrt{s}}\boxplus (\sqrt{ \mu_{MP, \nu_0+1, t}})_{even}\Bigr)^2.
\notag
\end{align}
\end{proof}

This corollary, the additivity of the  R-transform, and  the known R-transforms of Marchenko-Pastur distributions
 lead immediately to the  R-transform 
\begin{align}
R_{\Bigl(\mu_{sc, 2\sqrt{s}}\boxplus (\sqrt{ \mu_{MP, \nu_0+1, t}})_{even}\Bigr)^2}(z)
&= \frac{(\nu_0+1)(s+t)}{1-(s+t)z}-\frac{\nu_0s }{1-sz}\\
&=R_{ \mu_{MP, 1, s+t}}(z) +\frac{\nu_0 t}{(1-(s+t)z)(1-sz)}.
\notag
\end{align}

\section{The Marchenko-Pastur law for Bessel processes of type B}

In this section we transfer the results  for the ODEs (\ref{basic-ode-b})
to a stochastic setting for the Bessel processes of type B.
For this we consider Bessel processes $( X_{t,k})_{t\ge0}$ on the Weyl chambers $C_N^B$ for the root systems $B_{N}$ with $N\geq 2$ and multiplicities
 $k=(k_1,k_2)=(\nu\cdot \beta, \beta)$, $\nu,\beta>0$. These processes
 satisfy the SDE 
\begin{equation}\label{SDE-B}
 dX_{t,k}^i = dB_t^i+ \beta\sum_{j\ne i} \frac{X_{t,k}^i}{(X_{t,k}^i)^2-(X_{t,k}^j)^2}dt+ \beta \frac{\nu}{X_{t,k}^i}dt \quad\quad(i=1,\ldots,N).
\end{equation}
with  an $N$-dimensional Brownian motion $(B_t^1,\ldots,B_t^N)_{t\ge0}$.
By \cite{GM} we know, similar to Bessel processes of type A, that for $\beta\geq 1/2$ and $\nu\beta\geq 1/2$,
the  process $( X_{t,k})_{t\ge0}$ does not hit the boundary of $C_N^B$ 
a.s.~for $t>0$ for arbitrary starting points in $C_N^B$.
 In the following we only consider this regular case $\beta\geq 1/2, \nu\beta\geq 1/2$.
 We now derive limit theorems for the moments of the associated empirical measures, where, as in the deterministic setting,
 we first consider the  squares of all coordinates of our processes, namely
$$\mu^2_{N,t}:= \frac{1}{N}(\delta_{(X_{t,k}^1)^2/ {2N}}+\ldots +\delta_{(X_{t,k}^N)^2/ {2N}})$$
for $t\ge0$. For this it will be convenient 
 also to study the  renormalized processes $(\tilde X_{t,k}:=X_{t,k}/\sqrt \beta)_{t\ge0}$ which satisfy the SDE
\begin{equation}\label{SDE-B-normalized}
d\tilde X_{t,k}^i = \frac{1}{\sqrt{\beta}}dB_t^i+ \sum_{j\ne i} \frac{\tilde X_{t,k}^i}{(\tilde X_{t,k}^i)^2-(\tilde X_{t,k}^j)^2}dt+ \frac{\nu}{\tilde X_{t,k}^i}dt \quad\quad(i=1,\ldots,N).
\end{equation}
which agrees, for $\beta=\infty$, with the ODE (\ref{basic-ode-b}). Again we consider the  renormalized  empirical measures of the squares
$$\tilde\mu^2_{N,t}:= \frac{1}{N}(\delta_{(\tilde X_{t,k}^1)^2/ {2N}}+\ldots +\delta_{(\tilde X_{t,k}^N)^2/ {2N}}).$$
Denote the $l$-th moment $(l\in\mathbb N_0$) of  $\tilde\mu_{N,t}$ by
$$S_{N,l}(t):= \int_{\mathbb R} y^l \> d\tilde\mu_{N,t}(y) = \frac{1}{2^lN^{l+1}}\sum_{i=1}^N (\tilde X_{t,k}^i)^{2l}.$$

Now we derive limit theorems for these moments as $N\to\infty$. As in the deterministic setting the limits depend on the asymptotic behaviour of $\nu$.

\begin{theorem}\label{semicircle-mp-B}
Consider the processes $(\tilde X_{t,k})_{t\geq 0}$  with $\beta\geq 1/2$, $\nu>0$ and 
with  starting sequences
$(x_N)_{N\ge1}:=(x_{N,n})_{1\le k\le N}\subset [0,\infty[$
   as before
 such that the limits
 $$c_l(0):=\lim_{N\to\infty} S_{N,l}(0)=\lim_{N\to\infty}  \frac{1}{2^lN^{l+1}}(x_{N,1}^{2l}+\ldots+x_{N,N}^{2l})<\infty$$
 exist for  $l\ge0$. 
Assume that  $\nu=\nu(N)$ depends on $N$  with $\nu_0:=\lim_{N\to\infty}\nu(N)/N\ge0$.
Then,
for $l\in \mathbb N_0$,
$$c_l(t):=\lim_{N\to\infty} S_{N,l}(t)$$
exists almost surely locally uniformly in $t\in[0,\infty[$ with $c_0(t)=1$ and 
\begin{equation}
c_l(t)= c_l(0)+l \nu_0 \int_0^t c_{l-1}(s) ds + l\int_0^t 
\sum_{k=0}^{l-1} c_{l-1-k}(s)c_{k}(s)\> ds \quad(l\ge 1).
\end{equation}
\end{theorem}

\begin{proof}
Using It\^{o}'s formula we obtain for $l\geq 1$
\begin{align}\label{SDE-Moments-B}
\sum_{i=1}^N(\tilde X_{t,k}^i)^{2l}&=\sum_{i=1}^N x_i^{2l} + \frac{2l}{\sqrt \beta} \sum_{i=1}^N\int_0^t(\tilde X_{s,k}^i)^{2l-1} dB_s^i\\
+\int_0^t&\left(2l \sum_{i=1}^N \sum_{j\neq i}
\frac{2(\tilde X_{s,k}^i)^{2l}}{(\tilde X_{s,k}^i)^2-(\tilde X_{s,k}^j)^2}+
\sum_{i=1}^N (2l\nu+\frac{l(2l-1)}{\beta})(\tilde X_{s,k}^i)^{2l-2}\right) ds\notag
\end{align}
Now we see that for the normalized drift term we obtain the recurrence relation (\ref{moment-recurrence-lagu}),
where $\nu$ is replaced by $\nu +\frac{2l-1}{2\beta}$. The desired results now follow along the same lines as
in the proofs of Lemma \ref{general-limit-expectations-a}, Corollary \ref{convergence-in-expectation-a}, and
Theorem \ref{semicircle-A} using the deterministic results of  (\ref{recurrence-sl-laguerre-det}),
as well as Theorem \ref{free-convolution-b}. We here skip the details, as the details are part of the more complicated setting in Section 8.
\end{proof}

The methods of  Section 3 and Theorem \ref{semicircle-mp-B} lead to the following  limit theorem.

\begin{theorem}\label{free-convolution-b-random}
Let $\mu\in M^1([0,\infty[)$  such that $\mu^2$ satisfies
 the moment condition (\ref{strong-carleman}).
Let
 $(x_{N,n})_{N\ge1, 1\le n\le N}\subset[0,\infty[$ 
with $(x_{N, 1},\ldots, x_{N, N})\in C_N^B$ such that the  empirical measures
\begin{equation}
\mu_{N,0}:= \frac{1}{N}(\delta_{x_{N, 1}/\sqrt{2N}}+\ldots \delta_{x_{N, N}/\sqrt{2N}})
\end{equation}
tend weakly to  $\mu$ for $N\to\infty$. Consider the normalized Bessel processes
$(\tilde{X}_{t,k})_{t\ge0}$ of type B with start in $(x_{N,1},\ldots,x_{N,N})\in C_N^B $ for $N\ge2$.
Then, for $t\ge0$  and
$$\lim_{N\to\infty} \frac{\nu(N)}{N}=\nu_0\geq 0,$$
 the  normalized
 empirical measures
 $$\mu_{N,t}:= \frac{1}{N}(\delta_{\frac{\tilde{X}^1_{t,k}}{\sqrt{2 N}}}+\ldots +\delta_{\frac{\tilde{X}^N_{t,k}}{\sqrt{2N}}}) $$
tend a.s.~weakly to  $\sqrt{\mu_{MP, \nu_0, t}\boxplus (\mu_{sc, 2\sqrt{t}}\boxplus \mu_{even})^2}$ for $N\to\infty$.
\end{theorem}

 \section{Dunkl processes and their frozen versions}

In the remainder of the paper we extend the results of Sections 4 and 5 for Bessel processes of type B and their frozen versions to
 Dunkl processes of type B.
 As already discussed in the introduction, this extension from Bessel to Dunkl processes does not lead to new limit
 results for the root systems of type A. Nevertheless, we recapitulate briefly some facts on general Dunkl processes
 for the convenience of the reader.
 Moreover, the frozen versions  do not appear in the literature.
 To keep the discussion concise we assume that the reader is familiar with  root systems, Weyl groups, and Weyl chambers
 for the general theory. On the other hand, we present all details for the root systems of type B,
 as we are interested mainly in this case.

 All results on Dunkl theory and Dunkl processes and the algebraic and analytic background can be found in \cite{An, CGY, R1, R2, RV1, RV2}
 and references there.

   Let $R\subset \mathbb R^N$ be a root system and $R_+\subset R$ a subsystem consisting of positive roots.
   We assume that all roots $\alpha\in R$ satisfy $\|\alpha\|_2^2=2$.
   Then, for all  $\alpha\in R$, the reflection $\sigma_\alpha$ on the hyperplane perpendicular to  $\alpha$ is given by
   $\sigma_\alpha(x)=x-(\alpha\cdot x)\alpha$ for $x\in\mathbb R^N$ with   the standard scalar product $\cdot$  on $\mathbb R^N$.
   Let $W\subset O(N)$ be the finite reflection group, or Weyl group, generated by the  reflections $\sigma_\alpha$, $\alpha\in R_+$.
The Weyl group $W$ acts on  $\mathbb R^N$ as usual, and we have $W(R)=R$.
We next fix some nonnegative multiplicity function $k:R\to[0,\infty[$ which is by definition invariant under the canonical action of $W$ on $R$.

    For given $R$ and $k$ we now define the Dunkl operators $T_i$ ($i=1,\ldots,N$) as the differential-difference operators
    $$T_if(x):= \frac{\partial f(x)}{\partial x_i} + \sum_{\alpha\in R_+} k(\alpha) \alpha_i \frac{f(x)-f(\sigma_\alpha(x))}{\alpha\cdot x}
      \quad(f\in C^2(\mathbb R^N)).$$
      It is well known by Dunkl (see e.g.~\cite{DX}) that the operators $T_i$ commute with $T_if\in C^1(\mathbb R^N)$ for $f\in C^2(\mathbb R^N)$.
      Moreover,  the $T_i$ lower the degree of homogeneous polynomials in $N$
      variables by one like the usual partial derivative operators, which appear as special cases for $k= 0$.

      We next introduce the Dunkl Laplacian $L_k:=\sum_{i=1}^N T_i^2$ as well as its renormalization
      $\cal L_k:=\frac{1}{2} L_k$, which fits better to the usual normalization of Brownian motions in probability.
      $\cal L_k$ is given explicitly by
      \begin{equation}\label{Dunkl-Laplacian-general}
        \cal L_k f(x)= \frac{1}{2}\Delta f(x)+\sum_{\alpha\in R_+} k(\alpha)\Bigl(
        \frac{\nabla f(x)\cdot\alpha}{\alpha\cdot x}+ 
          \frac{f(\sigma_\alpha(x))-f(x)}{(\alpha\cdot x)^2}\Bigr)
      \end{equation}
      for $f\in C_c^2(\mathbb R^N)$ with the usual Laplacian $\Delta$ on $\mathbb R^N$; see e.g.~\cite{CGY, RV2}.
      The operator  $\cal L_k$ is the generator of some Feller
      semigroup on $\mathbb R^N$ whose transition densities can be written down
      explicitly in terms of so-called Dunkl kernels; see \cite{R1, RV1} for details.  Associated Feller processes $(X_{t,k})_{t\ge0}$
      on $\mathbb R^N$ with c\`adl\`ag
      paths are called Dunkl processes associated with the root system $R$ with multiplicity $k$.

 Let us briefly discuss the corresponding Bessel processes. For this we fix some closed Weyl chamber associated with $R$.
 This chamber may be regarded as the space $\mathbb R^N/W$ of all $W$-orbits in $\mathbb R^N$, i.e.,
 for each $x\in\mathbb R^N$ there is some unique $\pi(x)\in C$ for some $\pi\in W$.
 We thus have some canonical projection $\pi:\mathbb R^N\to \mathbb R^N/W \simeq C$, and we can consider the projected processes
 $$(X_{t,k}^W:=\pi(X_{t,k}))_{t\ge0}$$
 on $C$. By symmetry arguments, these processes $(X_{t,k}^W)_{t\ge0}$ are
 time-homogeneous diffusion processes on $C$ with reflecting boundaries and with the generators
 \begin{equation}\label{generator-bessel-general}
   \cal L_k^W f(x)= \frac{1}{2}\Delta f(x)+\sum_{\alpha\in R_+} k(\alpha)
        \frac{\nabla f(x)\cdot\alpha}{\alpha\cdot x}
 \end{equation}
 for functions  $f\in C_c^2(\mathbb R^N)$ which are $W$-invariant. Notice that this restriction of the domain of $\cal L_k^W $
fits to the  reflecting boundaries.
 These diffusions on $C$
 are called Bessel processes associated with the root system $R$  with multiplicity $k$.

 We now discuss  examples which are connected with the preceding sections:

 \begin{examples}
\begin{enumerate}
\item[\rm{(1)}] For $N\ge2$, the root system of type $A_{N-1}$ is given by
  $$R= \{e_i-e_j: \quad i,j=1,\ldots, N, \> i\ne j\}$$
  where $e_i\in \mathbb R^N$ is the $i$-th unit vector. Here, $W$ is the symmetric group acting on $\mathbb R^N$ as usual,
  the multiplicity $k$ is just a constant, the Weyl chamber $C$ may be chosen as $C_N^A$ from the introduction,
  and the generator in (\ref{generator-bessel-general}) is just the generator of the Bessel process of
  type  $A_{N-1}$ as in (\ref{def-L-A-intro}).
\item[\rm{(2)}] For $N\ge1$, the root system of type $B_{N}$ is given by
  $$R= \{\pm(e_i \pm e_j): \quad1\le i<j\le N\}\cup\{ \pm\sqrt 2 \cdot e_i: \quad i=1,\ldots,N\}.$$
 Here, $W$ is the hyperoctahedral group $\mathbb Z_2^N \ltimes S_N$, 
 the multiplicity $k$ consists of 2 constants, the Weyl chamber $C$ may be chosen as $C_N^B$ from the introduction,
 and the generator in (\ref{generator-bessel-general})
 is just the generator of the Bessel process of type  $B_N$ introduced
 in (\ref{def-L-B}).
  \end{enumerate}
 \end{examples}

   We now consider the freezing of Dunkl processes similar to that of Bessel processes in the preceding sections.
   For this we fix some root system $R$ and write the multiplicity function $k:R\to[0,\infty[$ as
       $k=\beta\cdot k_0$ with some fixed multipicity function $k_0$ and a varying constant $\beta>0$.
       For an Dunkl process $(X_{t,k})_{t\ge0}$ on $\mathbb R^N$ associated with $R$ and $k$, we now also study
       its renormalized version $(\tilde X_{t,k}:= \frac{1}{\sqrt\beta}X_{t,k})_{t\ge0}$ which then has the generator
  \begin{equation}\label{Dunkl-Laplacian-general-renormalized}
        \tilde{\cal L}_{k_0,\beta} f(x)= \frac{1}{2\beta}\Delta f(x)+\sum_{\alpha\in R_+} k_0(\alpha)\Bigl(
        \frac{\nabla f(x)\cdot\alpha}{\alpha\cdot x}+ 
          \frac{f(\sigma_\alpha(x))-f(x)}{(\alpha\cdot x)^2}\Bigr)
  \end{equation}
  for $f\in C_c^2(\mathbb R^N)$.  Clearly, this is also the generator of some Feller semigroup for $\beta=\infty$.
  Associated Feller processes then will be called frozen Dunkl processes associated with $R$ and $k_0$.
  These processes are pure jump processes.

 \begin{example}
   In the next section we study  frozen Dunkl processes for  the root system  $B_{N}$.
   Here we choose the fixed multiplicity $k_0$, which consists of 2 parameters, as $k_0=(1,\nu)$ with $\nu\ge0$.
   In this case, we denote the frozen  Dunkl processes on $\mathbb R^N$ by $(X_{t,\nu})_{t\ge0}$.
   The associated generator $\tilde{\cal L}_{k_0,\infty}$   according to (\ref{Dunkl-Laplacian-general-renormalized})
   will now be denoted as $L_\nu$ for simplicity. We then have
\begin{align}\label{generator-frozen-b}
L_\nu u(x)&= \sum_{i=1}^N \Bigl(\sum_{j:\> j\ne i} \frac{2x_i}{x_i^2-x_j^2}+\frac{\nu}{x_i}\Bigr) u_{x_i}(x)
+\frac{\nu}{2}\sum_{i=1}^N \frac{u(\sigma_ix)-u(x)}{x_i^2}\notag\\
&+\frac{1}{2} \sum_{i,j: \> j\ne i} \Bigl(\frac{u(\sigma_{i,j}x)-u(x)}{(x_i-x_j)^2}+\frac{u(\sigma_{i,j}^-x)-u(x)}{(x_i+x_j)^2}\Bigr)
\end{align}
  for $u\in \cal C^2(\mathbb R^N)$ with the reflections $\sigma_i, \sigma_{i,j}, \sigma_{i,j}^-$ $(i\ne j$) on  $\mathbb R^N$
  where $\sigma_i$ changes the sign of the $i$-th coordinate, $\sigma_{i,j}$ exchanges the coordinates $i,j$, and
  $\sigma_{i,j}^-$ exchanges the coordinates $i,j$ and changes the signs of these coordinates in addition.

  The Dunkl Laplacians $\tilde{\cal L}_{k_0,\beta}$ from (\ref{Dunkl-Laplacian-general-renormalized})
  for  the root system  $B_{N}$ with  $k_0=(1,\nu)$ will be denoted  by $\tilde{\cal L}_{\nu,\beta}$ in the next sections.
\end{example}

 \section{Empirical limit distributions for frozen Dunkl processes of type B}

 In this section we study the empirical distributions of
 frozen Dunkl processes $(X_{t,\nu})_{t\ge0}$ of type $B$ on $\mathbb R^N$ 
 with the generators $L_\nu$ defined in (\ref{generator-frozen-b}).
We assume that these processes start in deterministic points
in  $\mathbb R^N$.  
Moreover, we denote the components of $X_{t,\nu}$ by 
$X_{j,t,\nu}$  for $j=1,\ldots,N$.

Similar to Sections 2 and 4 we study
the  (random) normalized
 empirical measures
\begin{equation}\label{emp-measure-basic-b-dunkl}
\mu_{N,t,\nu}:= \frac{1}{N}(\delta_{X_{1,t,\nu}/\sqrt N}+\ldots +\delta_{X_{N,t,\nu}/\sqrt N})\in M^1(\mathbb R).
\end{equation}
of the processes  $(X_{t,\nu})_{t\ge0}$ of dimension $N$ for $N\to\infty$.
For this we study their moments
\begin{equation}\label{def-moments-case-b-dunkl}
  S_{N,l,\nu}(t):= \frac{1}{N^{l/2+1}}(X_{1,t,\nu}^l+\ldots+X_{N,t,\nu}^l)   \quad\quad (l\ge0).\end{equation}
By the very construction of the processes $(X_{t,\nu})_{t\ge0}$, the even moments
$ S_{N,2l,\nu}(t)$  are deterministic (and closely related to corresponding moments of the corresponding frozen Bessel processes
of type B in Section 4), while this is not the case for the odd moments $ S_{N,2l+1,\nu}(t)$ ($l\ge0$).
With the  functions 
$$u_l(x):= x_1^l+\ldots+ x_N^l \quad\quad(l\ge0)$$
we  have
\begin{equation}\label{rel-s-u-dunkl}
    S_{N,l,\nu}(t)= \frac{1}{N^{1+l/2}}\cdot u_l(X_{t,\nu})\quad\quad(l\ge0).
   \end{equation} 

We now compute  $L_\nu u_l$ for $l\ge0$. For this we observe that all $u_l$  are invariant under permutations of coordinates.
Therefore, for all $u:=u_l$
  \begin{align}\label{generator-frozen-b-perminvariant}
L_\nu u(x)&= \sum_{i=1}^N \Bigl(\sum_{j:\> j\ne i} \frac{2x_i}{x_i^2-x_j^2}+\frac{\nu}{x_i}\Bigr) u_{x_i}(x)
+\frac{\nu}{2}\sum_{i=1}^N \frac{u(\sigma_ix)-u(x)}{x_i^2}\notag\\
&+\frac{1}{2} \sum_{i,j: \> j\ne i} \frac{u(\sigma_{i,j}^-x)-u(x)}{(x_i+x_j)^2}.
\end{align}
  Moreover,   Dynkin's formula for Markov processes (see e.g. Section III.10 of \cite{RW}) shows that the processes
  \begin{equation}\label{dynkin-formula}
    \Bigl(u_l(X_{t,\nu})-u_l(X_{0,\nu})- \int_0^t (L_\nu u_l)(X_{s,\nu})\> ds\Bigr)_{t\ge 0}\end{equation}
  are martingales. Hence, for all $l\ge0$,
 \begin{equation}\label{dynkin-formula2}
   \frac{d}{dt} E(u_l(X_{t,\nu}))=E((L_\nu u_l)(X_{t,\nu})).\end{equation}
 
 We now compute $L_\nu u_l$ for $l\ge0$. We first notice that $u_0=N$ and thus $L_\nu u_0=0$.
 Moreover, a simple computation  yields that $L_\nu u_1=0$. By Dynkin's formula
 (\ref{dynkin-formula}) this corresponds to the well-known fact that the process  $(u_1(X_{t,\nu}))_{t\ge0}$ is a martingale; see e.g.~\cite{RV1}.
 We next turn to the general odd case $u_{2l+1}$ for $l\ge1$.  By (\ref{generator-frozen-b-perminvariant}) we have
\begin{align}
L_\nu u_{2l+1}(x)&= \sum_{i=1}^N \Bigl(\sum_{j:\> j\ne i} \frac{2x_i^{2l+1}}{x_i^2-x_j^2}+\nu x_i^{2l-1}\Bigr) (2l+1)\notag\\
&\quad -\sum_{i,j: \> j\ne i}\frac{x_i^{2l+1}+x_j^{2l+1}}{(x_i+x_j)^2} - \nu\sum_{i=1}^N x_i^{2l-1}\notag\\
&= (2l+1) \sum_{i,j: \> j\ne i}  \frac{x_i^{2l+1}- x_j^{2l+1}}{x_i^2-x_j^2}-
\sum_{i,j: \> j\ne i}\sum_{h=0}^{2l}(-1)^h \frac{x_i^h x_j^{2l-h}}{x_i+x_j}
  \notag\\
  &\quad + (2l+1)\nu u_{2l-1}(x) - \nu u_{2l-1}(x)\notag
  \end{align}
and thus 
    \begin{align}\label{ungerade-L}
L_\nu u_{2l+1}(x)&= 
   (2l+1)\sum_{i,j: \> j\ne i}\sum_{h=0}^{2l} \frac{x_i^h x_j^{2l-h}}{x_i+x_j}\notag\\
&\quad - \sum_{i,j: \> j\ne i}\sum_{h=0}^{2l}(-1)^h \frac{x_i^h x_j^{2l-h}}{x_i+x_j} +2l\nu u_{2l-1}(x)\notag\\
 &= \sum_{i,j: \> j\ne i}\frac{1}{x_i+x_j}\Bigl( 2l \sum_{h=0}^{l}x_i^{2h} x_j^{2l-2h}+2(l+1)\sum_{h=0}^{l-1}x_i^{2h+1} x_j^{2l-2h-1}\Bigr)
\notag\\ &\quad +2l\nu u_{2l-1}(x)\notag\\
&=\sum_{i,j: \> j\ne i}\Bigl( 2l x_j^{2l-1}+ 2 x_j^{2l-2}x_i + (2l -2) x_j^{2l-3}x_i^2 + 4  x_j^{2l-4}x_i^3 +\ldots\notag\\
&\quad\quad\ldots+ 2 x_jx_i^{2l-2} +   2l x_i^{2l-1}\Bigr) +2l\nu u_{2l-1}(x)\notag\\
&= 2\Bigl(\sum_{i,j: \> j\ne i}\sum_{h=0}^{l-1} x_i^{2h}x_j^{2l-1-2h}(l-h)+ \sum_{i,j: \> j\ne i}\sum_{h=0}^{l-1} x_j^{2h}x_i^{2l-1-2h}(l-h)\Bigr)\notag\\
&\quad +2l\nu u_{2l-1}(x)\notag\\
&= 4 \sum_{i,j: \> j\ne i}\sum_{h=0}^{l-1}(l-h)x_i^{2h}x_j^{2l-1-2h} +2l\nu u_{2l-1}(x).
\end{align}
    As
    $\sum_{i,j: \> j\ne i}x_i^{2h}x_j^{2l-1-2h}= u_{2h}(x)u_{2l-1-2h}(x)-u_{2l-1}(x)$,
 $\sum_{h=0}^{l-1}(l-h)=l(l+1)/2$ and $u_0=N$,  we get
\begin{align}\label{ungerade-L2}
L_\nu u_{2l+1}&=2l\nu u_{2l-1} + 4 \sum_{h=0}^{l-1}(l-h)u_{2h}u_{2l-1-2h} - 4 \sum_{h=0}^{l-1}(l-h)u_{2l-1} \notag\\
&=(2l\nu +4lN -2l(l+1))u_{2l-1} + 4 \sum_{h=1}^{l-1}(l-h)u_{2h} u_{2l-1-2h}\notag\\
&=2l(2N+\nu-(l+1))u_{2l-1} + 4 \sum_{h=1}^{l-1}(l-h)u_{2h} u_{2l-1-2h} .
\end{align}
Moreover, by a similar computation we obtain in the even case for $l\ge1$ that
\begin{equation}\label{gerade-L2}
L_\nu u_{2l}=2l(\nu+l+1)u_{2l-2}+2l \sum_{h=0}^{l-1}u_{2h} u_{2l-2-2h}.
\end{equation}
If we combine (\ref{dynkin-formula2}) with (\ref{ungerade-L2}) and use  that $u_{2h}(X_{t,\nu})$ is deterministic,
we see that
\begin{align} \frac{d}{dt}E( u_{2l+1}(X_{t,\nu}))&= E((L_\nu u_{2l+1})(X_{t,\nu}))\notag\\
  &= 2l(2N+\nu-(l+1))E( u_{2l-1}(X_{t,\nu}))\notag\\
  &\quad\quad +4 \sum_{h=1}^{l-1}(l-h)u_{2h}(X_{t,\nu}) \cdot E( u_{2l-1-2h}(X_{t,\nu})).
\notag\end{align}
Therefore, by (\ref{rel-s-u-dunkl}),
\begin{align}\label{ungerade-L3}
  \frac{d}{dt}E( S_{N,2l+1,\nu}(t))&= \frac{2l(2N+\nu-(l+1))}{N} E( S_{N,2l-1,\nu}(t))\\
 &\quad\quad+ 4\sum_{h=1}^{l-1}(l-h) S_{N,2h,\nu}(t) E( S_{N,2l-1-2h,\nu}(t)).
\notag\end{align}
Moreover, in the even case we obtain in a similar way from  (\ref{gerade-L2}) that
\begin{equation}\label{gerade-L3}
  \frac{d}{dt} S_{N,2l,\nu}(t)= \frac{2l(\nu+l+1)}{N}S_{N,2l-2,\nu}(t) + 2l \sum_{h=0}^{l-1}S_{N,2l-2-2h,\nu}(t)S_{N,2h,\nu}(t).
\end{equation}
These two recurrence relation and the methods of the proofs of Lemmas \ref{recurrence-sl-hermite-det}
and \ref{recurrence-sl-laguerre-det} now readily lead to the following limit result.

\begin{lemma}\label{recurrence-sl-dunkl-det}
  Let $(x_{N,k})_{1\le k\le N}\subset \mathbb R$ be the starting sequences of the frozen Dunkl processes $(X_{t,\nu})_{t\ge0}$ for
  $N\in\mathbb N$ with
$X_{0,\nu}=(x_{N,1},\ldots,x_{N,N})$
   such that 
   $$c_l(0):=\lim_{N\to\infty} S_{N,l,\nu}(0)=\lim_{n\to\infty}  \frac{1}{N^{l/2+1}}(x_{N,1}^l+\ldots+x_{N,N}^l)<\infty$$
exists for all $l\in \mathbb N_0$. Assume that $\nu=\nu(N)$ depends on $N$ with
$\lim_{N\to\infty} \frac{\nu(N)}{N}=\nu_0$ with some constant $\nu_0 \geq 0$.
Then
for $l\in \mathbb N_0$,
$$c_l(t):=\lim_{N\to\infty} E(S_{N,l,\nu(N)}(t))$$
exists locally uniformly in $t\in[0,\infty[$ and satisfies the recurrence relations
     $c_0(t)=1$, $ c_1(t)=c_1(0)$, and for $l\ge1$,
\begin{equation}\label{recurrence-dunkl-gerade}
c_{2l}(t)=c_{2l}(0)+ 2l \int_0^t\Bigl(\nu_0 c_{2l-2}(s) +\sum_{h=0}^{l-1} c_{2h}(s)c_{2l-2h-2}(s)\Bigr)\> ds,
\end{equation}
\begin{equation}\label{recurrence-dunkl-ungerade}
  c_{2l+1}(t)=c_{2l+1}(0)+ \int_0^t\Bigl(2l\nu_0 c_{2l-1}(s)+4 \sum_{h=0}^{l-1}(l-h) c_{2h}(s)c_{2l-2h-1}(s)\Bigr)\> ds.
\end{equation}
\end{lemma}

Please notice that the even limit moments $c_{2l}(t)$ in (\ref{recurrence-dunkl-gerade})
have the same recurrence relation as the moments $c_l(t)$ in (\ref{recurrence-laguerre}) in the context of Bessel processes of type B up to an factor $2$,
This factor 2 is caused by by the slightly different scalings of the emirical measures in (\ref{emp-measure-basic-b-dunkl}) and
(\ref{empirical-gen-lagu}).

We now proceed as in Sections 2 and 4 and obtain that under some mild conditions on the initial data, the $c_{l}(t)$ ($l\ge0$)
in the lemma are the moments of  unique probability measures $\mu_t$ for $t\ge0$.
 For this we fix some  probability measure
 $\mu\in M^1(\mathbb R)$ which is 
  determined uniquely by its moments $c_l$ ($l\ge0$).
We now choose 
 $(x_{N,n})_{N\ge1, 1\le n\le N}\subset\mathbb R$  
with $x_{N, n-1}\ge x_{N, n}$ for  $2\le n\le N$ such that the  empirical measures
$$\mu_{N,0}:= \frac{1}{N}(\delta_{x_{N, 1}/\sqrt N}+\ldots \delta_{x_{N, N}/\sqrt N})$$
tend weakly to  $\mu$ for $N\to\infty$, i.e., by the moment convergence theorem, that
$$\lim_{N\to\infty} S_{N,l,\nu}(0):=\lim_{N\to\infty}  \frac{1}{N^{l/2+1}}(x_{N,1}^l+\ldots+x_{N,N}^l)=c_l \quad(l\ge0).$$
For  $N\ge 2$ we now consider the measures $\mu_{N,t,\nu}$ from (\ref{emp-measure-basic-b-dunkl}).
The proof of the following result is now completely analogous to that of 
Propositions \ref{general-free-convolution-a-firststep} or \ref{general-free-convolution-b-firststep}:

\begin{proposition}\label{general-free-convolution-dunkl-firststep}
  Let $\mu\in M^1(\mathbb R)$ be a probability measure for which the moment condition (\ref{strong-carleman}) holds. Assume
 that with the preceding notations $\lim_{N\to\infty} \frac{\nu(N)}{N}=\nu_0\geq 0$ exists.
 Then for $t\ge0$ there are unique probability
 measures $\mu_t\in M^1(\mathbb R)$  with the moments $(c_l(t))_{l\ge0}$. Moreover, for all $t\ge0$,
 the moment condition (\ref{strong-carleman}) holds for $\mu_t$, and  $\mu_t$ is the weak limit of the  empirical measures
 $\mu_{N,t,\nu(N)}$ for $N\to\infty$.
\end{proposition}

In the situation of Proposition \ref{general-free-convolution-dunkl-firststep} we now derive PDEs 
for the  Stieltjes transforms
$$G(t,z):=G_{\mu_t}(z), \quad G^N(t,z):=G_{\mu_{N,t, \nu(N)}}(z) \quad\quad (t\ge0, \> z\in\mathbb C, \Im z\ne0)$$
of the  measures $\mu_t$ and $\mu_{N,t, \nu(N)}$. As the recurrence relations in Lemma \ref{recurrence-sl-dunkl-det}
are different for even and odd moments, we  decompose  $\mu_t$ and  $\mu_{N,t,\nu(N)}$ 
into their even and odd parts and study the  Stieltjes transforms of these measures.
We thus define the reflected probability measures $\mu_{t}^*\in M^1(\mathbb R)$ with $\mu_{t}^*(A)=\mu_{t}(-A)$ for Borel sets $A\subset\mathbb R$ and
$$\mu_{t,even}:= \frac{1}{2}(\mu_{t}+\mu_{t}^*), \quad \mu_{t,odd}:= \frac{1}{2}(\mu_{t}-\mu_{t}^*).$$
Please notice that $\mu_{t,odd}$ usually is a signed measure, and that $\mu_{t}=\mu_{t,even}+\mu_{t,odd}$.
We now introduce the Stieltjes transforms
$$G^{even}(t,z):=G_{\mu_{t,even}}(z), \quad G^{odd}(t,z):=G_{\mu_{t,odd}}(z)$$
with $G=G^{even}+G^{odd}$.  As by the definition (\ref{def-Stieltjes}) of the  Stieltjes transform
$G_{\mu_t^*}(z)= -G_{\mu_t}(-z)$, and as
$G_{\mu_t}(z)=\sum_{l=0}^\infty c_l(t) z^{-(l+1)}$ for $|z|$ sufficiently large, we obtain that
\begin{align}\label{deg-g-even-odd}
  G^{even}(t,z)=& \frac{1}{2}( G(t,z)-G(t,-z))= \sum_{l=0}^\infty \frac{c_{2l}(t)}{z^{2l+1}}= z \tilde G^{even}(t,z^2)\notag\\
  G^{odd}(t,z)=& \frac{1}{2}( G(t,z)+G(t,-z))= \sum_{l=0}^\infty \frac{c_{2l+1}(t)}{z^{2l+2}}=  \tilde G^{odd}(t,z^2)
\end{align}
with the functions
\begin{equation}\label{g-tilde}
  \tilde G^{even}(t,z):=\sum_{l=0}^\infty \frac{c_{2l}(t)}{z^{l+1}}, \quad
  \tilde G^{odd}(t,z):=\sum_{l=0}^\infty \frac{c_{2l+1}(t)}{z^{l+1}}.
  \end{equation}
We use the corresponding notations and relations also for the measures $\mu_{N,t, \nu(N)}$.

In the setting of 
Proposition \ref{general-free-convolution-dunkl-firststep} we now have:

\begin{proposition}\label{general-free-convolution-dunkl-secondstep}
  The Stieltjes transforms $G^{even}, G^{odd}$ and the functions $\tilde G^{even},\tilde G^{odd}$
satisfy the PDEs
\begin{align}\label{pde-g-dunkl-tilde}
  \tilde G_t^{even}(t,z)&=-2\nu_0 \tilde G_z^{even}(t,z)-4 z\tilde G_z^{even}(t,z)\tilde G^{even}(t,z) -2\tilde G^{even}(t,z)^2,
  \notag\\
  \tilde G_t^{odd}(t,z)&=\Bigl(-2\nu_0- 4z\tilde G^{even}(t,z)\Bigr)\tilde G_z^{odd}(t,z)
\end{align}
and
\begin{align}\label{pde-g-dunkl}
   G_t^{even}(t,z)&= \nu_0 \Bigl(\frac{G^{even}(t,z)}{z^{2}} -\frac{G^{even}_z(t,z)}{z} \Bigr)  -2  G^{even}(t,z) G_z^{even}(t,z)
  \notag\\
   G_t^{odd}(t,z)&=\Bigl(-\frac{\nu_0}{z}- 2 G^{even}(t,z)\Bigr) G_z^{odd}(t,z)
\end{align}
for  $t\ge0$, $z\in H$. 
\end{proposition}

\begin{proof}
  We first consider $\tilde G^{even}$. Here, (\ref{g-tilde}), $\frac{d}{dt}c_0(t)=0$, and (\ref{recurrence-dunkl-gerade}) imply that
  \begin{align}
    \frac{d}{dt}  \tilde G^{even}(t,z)&=\sum_{l=1}^\infty \frac{1}{z^{l+1}}\frac{d}{dt}c_{2l}(t)
    = \sum_{l=0}^\infty \frac{1}{z^{l+2}}\frac{d}{dt}c_{2l+2}(t)\notag\\
  &=  \sum_{l=0}^\infty \frac{1}{z^{l+2}}\Bigl( 2(l+1) \nu_0c_{2l}(t) +2\sum_{h=0}^l(l+1)c_{2h}(t) c_{2l-2h}(t)\Bigr)
\notag\\
  &=-2\nu_0 \tilde G_z^{even}(t,z)-2 \frac{d}{dz}\Bigl(  z\cdot \tilde G^{even}(t,z)^2\Bigr)   
  \end{align}
  which implies the first equation in (\ref{pde-g-dunkl-tilde}). Moreover,
  (\ref{g-tilde}), $\frac{d}{dt}c_1(t)=0$, and (\ref{recurrence-dunkl-ungerade})  lead to
\begin{align}
    \frac{d}{dt}  \tilde G^{odd}(t,z)&=\sum_{l=1}^\infty \frac{1}{z^{l+1}}\frac{d}{dt}c_{2l+1}(t)
    = \sum_{l=0}^\infty \frac{1}{z^{l+2}}\frac{d}{dt}c_{2l+3}(t)\notag\\
  &=  \sum_{l=0}^\infty \frac{1}{z^{l+2}}\Bigl( 2(l+1) \nu_0c_{2l+1}(t) +4\sum_{h=0}^l(l+1-h)c_{2h}(t) c_{2l+1-2h}(t)\Bigr)
\notag\\
  &=-2\nu_0 \tilde G_z^{odd}(t,z)-4   \tilde G^{even}(t,z)  \tilde G_z^{odd}(t,z)  . 
\end{align}
Please notice that we here interchanged derivatives and summations several times.
This can be made precise by the same methods as in the proof of Proposition \ref{general-free-convolution-a-secondstep}
by studying the corresponding approximating PDEs for  the Stieltjes transforms of the measures $\mu_{N,t, \nu(N)}$.

Finally, the equations in (\ref{pde-g-dunkl}) follow from that in (\ref{pde-g-dunkl-tilde}) by an easy computation.

\end{proof}

We now analyze the PDEs  (\ref{pde-g-dunkl}) with the corresponding initial data for $t=0$. Clearly, the quasilinear
PDE for  $G^{even}$ has to be solved first which then leads to the even parts $\mu_{t,even}$ of our probability measures $\mu_t$.
On the other hand, the remark after Lemma \ref{recurrence-sl-dunkl-det} implies
that the  $\mu_{t,even}$ can be determined via the results in Section 4. More precisely, this remark and
Theorem \ref{free-convolution-b}
imply:

\begin{corollary}\label{dunkl-even-solutions}
  For $t\ge0$, the even parts  $\mu_{t,even}$ of the  probability measures $\mu_t$ of Proposition
  \ref{general-free-convolution-dunkl-firststep} are the unique even probability measures on $\mathbb R$
  whose pushforwards under the mapping $x\mapsto x^2/2$ are given by $\mu_{MP, \nu_0, t}\boxplus (\mu_{sc, 2\sqrt{t}}\boxplus \mu_{even})^2$.
  Hence,
  \begin{equation}\label{mu-t-even-cor}
    \mu_{t,even}= \Bigl(\sqrt{\mu_{MP, \nu_0, 2t}\boxplus (\mu_{sc, 2\sqrt{2t}}\boxplus \mu_{even})^2}\bigr)_{even} \quad\quad(t\ge0).
    \end{equation}
\end{corollary}

Therefore, if the initial measure $\mu_0=\mu$ is even, then the associated linear PDE for $ G^{odd}$ in  (\ref{pde-g-dunkl})
has the solution $ G^{odd}=0$, i.e., the $\mu_t$ of Proposition
\ref{general-free-convolution-dunkl-firststep} are given by the measures in Eq.~(\ref{mu-t-even-cor}).

We now study the case where  the initial measure $\mu_0=\mu$ is not even.
We here have to solve the linear PDE of first order in (\ref{pde-g-dunkl}) for  $ G^{odd}$ by using
$ G^{even}$ from Corollary \ref{dunkl-even-solutions}.
This can be carried out  by classical methods on PDEs; see e.g.~Section 1.2 of \cite{St}. 
However,  we can solve the second PDE (\ref{pde-g-dunkl}) directly:

\begin{lemma}\label{dunkl-solution-nu0}
  Let $\mu\in M^1(\mathbb R)$ be a probability measure for which the moment condition (\ref{strong-carleman}) holds.
Let $D\subset H$ be some non-empty open domain such that there is some (analytical) function $K$ with
$$K[G_{\mu_{even}}(z)\cdot(G_{\mu_{even}}(z)+\nu_0 G_{\delta_0}(z))]=z \quad\quad (z\in D).$$
Then, with the measures $\mu_{t,even}$ from  Eq.~(\ref{mu-t-even-cor}),
\begin{equation}\label{def-odd-stieltjes-nu0}
  G(t,z):=G_{\mu_{odd}}[K\big(G_{\mu_{t,even}}(z)\cdot(G_{\mu_{t,even}}(z)+\nu_0 G_{\delta_0}(z))\big)]+G_{\mu_{t,even}}(z)
\end{equation}
 is the solution of (\ref{pde-g-dunkl}) with the initial condition $G(0,z)=G_\mu(z)$ for $z\in D$.

In particular, for  $\nu_0=0$ this solution simplifies to
 \begin{equation}\label{def-odd-stieltjes-0}
     G(t,z)= G_\mu(G_{\mu_{even}}^{-1}( G_{\mu_{sc, 2\sqrt{2t}}\boxplus \mu_{even}}    (z))).
  \end{equation}  
\end{lemma}

\begin{proof} We first notice that the Stieltjes transforms $G_\mu, G_{\mu_{even}}$ are
  analytic  on $\mathbb C\setminus\mathbb R$ where, by (\ref{g-tilde}),
  $ G_{\mu_{even}}(w)=  \frac{1}{w}+  \sum_{l=1}^\infty \frac{c_{2l}(0)}{w^{l+1}}$ for $|w|$ sufficiently large.
  This implies that $G_{\mu_{even}}\cdot(G_{\mu_{even}}+\nu_0 G_{\delta_0})$ is not constant on $H$ which ensures that there are plenty of 
  non-empty open domains $D\subset H$ where this function admits some analytical inverse function $K$.

  We now analyze the right hand side of (\ref{def-odd-stieltjes-nu0}).
As $(z,t)\mapsto G_{\mu_{t,even}}(z)$ solves the  first  PDE  in (\ref{pde-g-dunkl}) by Corollary \ref{dunkl-even-solutions},
we only have to establish that the first summand of (\ref{def-odd-stieltjes-nu0}) solves the second PDE  in (\ref{pde-g-dunkl}),
and that the initial value is correct. The first summand in (\ref{def-odd-stieltjes-nu0})
has the form $$G^{odd}(z,t):=f(G^{even}(z,t)\cdot(G^{even}(z,t)+\frac{\nu_0}{z})).$$ Hence,
$$G^{odd}_t(z,t)=f'(G^{even}(z,t)\cdot(G^{even}(z,t)+\frac{\nu_0}{z}))\cdot G_t^{even}(z,t)(2G^{even}(z,t)+\frac{\nu_0}{z})$$
and
\begin{eqnarray*}
G^{odd}_z(z,t)&=&f'(G^{even}(z,t)\cdot(G^{even}(z,t)+\frac{\nu_0}{z}))\cdot\\
&&\cdot[2G^{even}(z,t) G_z^{even}(z,t)+ G_z^{even}(z,t)\cdot\frac{\nu_0}{z}-\frac{\nu_0}{z^2} G^{even}(z,t)].
\end{eqnarray*}
These derivatives yield immediately that  $G^{odd}$ satisfies the second PDE  in (\ref{pde-g-dunkl}).
Moreover, a simple calculation  shows that the corresponding initial value is correct.

Let us consider the special case $\nu_0=0$. Here, $ G_{\mu_{t,even}}=G_{\mu_{sc, 2\sqrt{2t}}\boxplus \mu_{even}}$ and
 $K=(G_{\mu_{even}}^2)^{-1}$. This reduces  the function in (\ref{def-odd-stieltjes-nu0})  to
that in  (\ref{def-odd-stieltjes-0}).
\end{proof}

We now discuss an example with $\nu_0=0$ where the initial distribution $\mu$ has a semicircle law as initial even part.

\begin{example}\label{quarter-example}
  Let the starting distribution $\mu$ be the quartercircle distribution on $[0,2]$ with Lebesgue density
  $\frac{1}{\pi}\sqrt{4-x^2}{\bf 1}_{[0,2]}$.
  Then
  $$\mu_{even}=\mu_{sc,2}
  \quad\text{and}\quad
  \mu_{sc, 2\sqrt{2t}}\boxplus \mu_{even}=\mu_{sc, 2\sqrt{2t+1}}.$$
  Moreover, 
 \begin{equation}\label{stieltjes-trafo-symm-a1}
   G_{\mu_{sc, 2}}(z)=  \frac{1}{2}(z-\sqrt{z^2-4})\quad\text{and}\quad
   G_{\mu_{sc, 2}}^{-1}(z)=z+1/z
\end{equation}
  (see e.g.~page 305 of \cite{AGZ}) and thus
  \begin{align}\label{stieltjes-trafo-symm-a2}
    w(z):=G_{\mu_{sc, 2\sqrt{2t+1}}}(z)&= \frac{1}{\sqrt{2t+1}}G_{\mu_{sc, 2}}\Bigl(\frac{z}{\sqrt{2t+1}}\Bigr)\\
    &= \frac{1}{2(2t+1)}\Bigl( z- \sqrt{z^2-4(2t+1)}\Bigr).
  \notag\end{align}

Now consider the probability measures $\mu_t$ for $t\ge0$ whose even parts are the  semicircle laws $\mu_{sc, 2\sqrt{2t+1}}$.
  The $\mu_t$ are supported on $[-2\sqrt{2t+1}, +2\sqrt{2t+1}]$ and admit densities $f_t$.
We recapitulate from Theorem 2.4.3 of \cite{AGZ} that for $x\in]-2\sqrt{2t+1}, +2\sqrt{2t+1}[$
\begin{equation}\label{density-from-Stieltjes}
  f_t(x)= \frac{-1}{\pi}\lim_{\epsilon\downarrow 0} \Im G(t, x+i\epsilon)
\end{equation}
where we can determine $G(t, x+i\epsilon)$ by Lemma \ref{dunkl-solution-nu0}.
Moreover, as the even parts of the $\mu_t$ are already known, we can restrict our attention to  $x\in]0,2\sqrt{2t+1}[$.
    
 For $z:=x+i\epsilon$ with $\epsilon>0$ we then see that
 \begin{equation}\label{def-w}
   \lim_{\epsilon\downarrow 0} w(x+i\epsilon)=\frac{1}{2(2t+1)}\Bigl( x- i\sqrt{4(2t+1)-x^2}\Bigr)=w(x)=:w
   \end{equation}
    is contained in the closed fourth quadrant with $|w|= 1/\sqrt{2t+1}\le1$. In particular, 
    $1/w$ is in the closed  first quadrant with $|1/w|\ge|w|$, and 
    $1/w +w$ and $1/w -w$ are also in the  closed  first quadrant.

We next determine
$$G_\mu(u)=\frac{1}{\pi}\int_{0}^{2} \frac{\sqrt{4-x^2}}{u-x}\> dx .$$
As e.g.~by WOLFRAM Alpha
\begin{align}
  \int\frac{\sqrt{4-x^2}}{u-x}\> dx
  =&-\sqrt{4 - x^2}+ u \arcsin(x/2) \notag\\
&+ \sqrt{4 - u^2}\Bigl(\ln(\sqrt{(4 - x^2)(4 - u^2)}  - x u + 4)-\ln(u - x)\Bigr)+c,
\notag\end{align}
we obtain 
\begin{align}\label{gmu-special}
  G_\mu(u)&= \frac{2}{\pi}+ \frac{u}{2}+ \frac{\sqrt{4-u^2}}{\pi}
  \Bigl(\ln(2-u)-\ln(u-2) -\ln(\sqrt{4-u^2}+2)+\ln u\Bigr) \notag\\
  &= \frac{2}{\pi}+\frac{u}{2} \pm \sqrt{u^2-4}- \frac{\sqrt{4-u^2}}{\pi} \Bigl(\ln(\sqrt{4-u^2}+2)-\ln u\Bigr).
\end{align}
For $w$ as above we now define  $u:=1/w +w$. We then have $\sqrt{u^2-4}=\pm(1/w -w)$. An analysis of the correct
branch of the square root for $u$ in the closed first quadrant and
   (\ref{def-odd-stieltjes-nu0}) now lead to 
\begin{align}\lim_{\epsilon\downarrow 0} G(t,x+i\epsilon)&= G_\mu(G_{\mu_{even}}^{-1}( G_{\mu_{sc, 2\sqrt{2t}}\boxplus \mu_{even}}    (x)))\\
  &=\frac{2}{\pi}+ \frac{1/w +w}{2}+(w- 1/w)- \frac{i}{\pi}(1/w-w)\ln\Bigl(\frac{\sqrt{4-u^2}+2}{u}\Bigr)\notag\\
  &=\frac{2}{\pi}+ \frac{1/w +w}{2}+(w- 1/w)- \frac{i}{\pi}(1/w-w)\ln\Bigl(\frac{ i(1/w-w)+2}{1/w +w}\Bigr).\notag
\end{align}
Moreover, as $1/w=(2t+1)\bar w$, we obtain from (\ref{def-w}) that
$$\Im\bigl(\frac{1/w +w}{2}+(w- 1/w)\bigr)= -\frac{1+t/2}{2t+1}\sqrt{4(2t+1)-x^2}$$
and
$$\Im(w- 1/w)=- \frac{t+1}{2t+1}\sqrt{4(2t+1)-x^2}.$$
Furthermore, as
\begin{align}
  \frac{ i(1/w-w)+2}{1/w +w}&=\frac{ i-iw^2+2w}{1 +w^2}=\frac{ (-i)(w^2+2iw-1)}{1 +w^2}
  =\frac{  (-i)(w+i)^2}{(w+i)(w-i)}
\notag \\
  &  =\frac{  (-i)(w+i)}{w-i}
 =\frac{  (-i)(w+i)(\bar w +i)}{(w-i)(\bar w +i)}=\frac{  (-i)(w+i)(\bar w +i)}{|w-i|^2}\notag \\
  & =\frac{w+ \bar w+i(1-w\bar w)}{|w-i|^2},
  \notag\end{align}
  we obtain from (\ref{def-w}) that
$$\Im\ln \Bigl(\frac{ i(1/w-w)+2}{1/w +w}\Bigr)=\arctan \Bigl(\frac{1-w\bar w}{w+ \bar w}\Bigr)
  =\arctan (2t/x).$$
  Finally,
  $\Re(w- 1/w)=- \frac{tx}{2t+1}$ and
	$$\Re\ln \Bigl(\frac{ i(1/w-w)+2}{1/w +w}\Bigr)=\Re\ln\Bigl(\frac{w+i}{w-i} \Bigr)
  = \frac{1}{2}\ln\Bigl(\frac{2(t+1)-\sqrt{4(2t+1)-x^2}}{ 2(t+1)+\sqrt{4(2t+1)-x^2}}\Bigr)$$
  In summary, we see from $\arctan y + \arctan (1/y)=\pi/2$ for $y>0$ that for $x\in]0,2\sqrt{2t+1}[$,
    \begin{align}\label{density-dunkl-final}
      f_t(x)&=\frac{-1}{\pi}\lim_{\epsilon\downarrow 0}\Im G(t,x+i\epsilon)\\
      &=\frac{1}{(2t+1)\pi}\Bigl(\frac{1}{2}+\frac{t+1}{\pi} \arctan \frac{x}{2t} \Bigr)\sqrt{4(2t+1)-x^2}
      \notag\\
      &\quad -\frac{1}{\pi^2} \frac{tx}{2(2t+1)}\ln\Bigl(\frac{2(t+1)+\sqrt{4(2t+1)-x^2}}{ 2(t+1)-\sqrt{4(2t+1)-x^2}}\Bigr).
     \notag\end{align}
     Moreover, as the even  part of $\mu_t$ is the semicircle law $\mu_{sc, 2\sqrt{2t+1}}$, Eq.~(\ref{density-dunkl-final})
     remains also valid for  $x\in]-2\sqrt{2t+1},0[$.

     In order to understand densities $f_t$ better,
     we define the rescaled densities
     $$\tilde f_t(x)=\sqrt{2t+1} f(x\sqrt{2t+1})$$  of probability measures on $[-2,2]$.
     We then have for $x\in[-2,2]$,
     \begin{align}\tilde f_t(x)
&=  \frac{1}{\pi}\Bigl(\frac{1}{2}+\frac{t+1}{\pi}\arctan(\sqrt{2t+1}x/(2t))\Bigr)\sqrt{4-x^2}
\notag\\
&\quad -\frac{1}{\pi^2} \frac{tx}{2}
\ln\Bigl(\frac{1+\frac{\sqrt{2t+1}}{2(t+1)}\sqrt{4-x^2}}{ 1-\frac{\sqrt{2t+1}}{2(t+1)}\sqrt{4-x^2}}\Bigr). \end{align}
A Taylor expansion of the $\arctan$- and $\ln$-term now yields that
\begin{equation}
  \tilde f_t(x)= \frac{\sqrt{4-x^2}}{2\pi}+\frac{2\sqrt 2 \cdot x\sqrt{4-x^2}}{3\pi^2}\cdot O(\frac{1}{\sqrt{t}}) +O(1/t)
\end{equation}
for $t\to\infty$ where the term $ O(\frac{1}{\sqrt{t}})$ is independent of $x\in[-2,2]$
while the term $O(1/t)$ may depend on $x$.  In particular, for $t\to\infty$, $\tilde f_t$ tends to the density of the
semicircle law $\mu_{sc,2}$, i.e.,  the influence of the asymmetric starting measure vanishes
of order $O(\frac{1}{\sqrt{t}})$. Figure \ref{fig} illustrates the time-behaviour of $\tilde f_t(x)$.
It is plotted for $t=0.1$ (bold black), $t=1$ (dashed), $t=10$ (dashed small), $t=100$ (dotted)
together with the start $t=0$ and the limit $t\to\infty$ (both grey).
\begin{figure}[h]
			\centering 
			\includegraphics[scale=0.35]{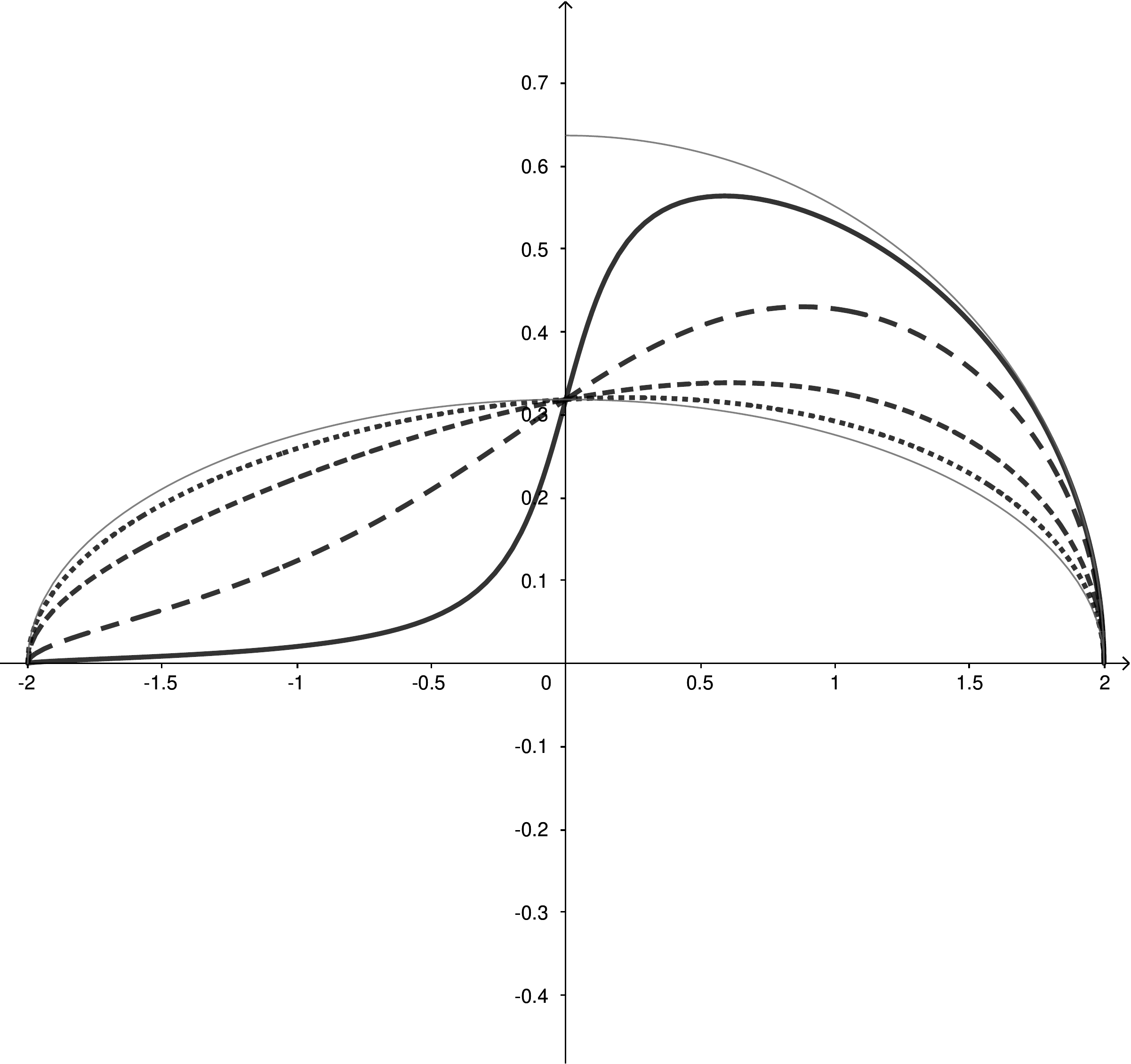}
			\caption{ $\tilde f_t(x)$ for $t= 0, \;\, 0.1,\;\, 1,\;\, 10,\;\, 100$ and $t= \infty$.}
			\label{fig}
	\end{figure} 
  \end{example}

 \section{Limit theorems for Dunkl processes of type B}

In this section we proceed to the next step and study the empirical distributions of
 normalized Dunkl processes $(\tilde{X}_{t,\nu,\beta})_{t\ge0}$ of type $B$ on $\mathbb R^N$ 
 with the generators $\tilde{\cal L}_{\nu, \beta}= \frac{1}{2\beta} \Delta+L_{\nu}$ with $L_{\nu}$
 as in (\ref{generator-frozen-b}); see Example 6.2.
 On some informal level,
  the processes $(\tilde{X}_{t,\nu,\beta})_{t\ge0}$ converge  for $\beta\to\infty$
 to the frozen Dunkl processes.
We assume that these processes start in deterministic points
in  $\mathbb R^N$ independent of $\beta$.

We denote the components of $X_{t,\nu,\beta}$ by 
$X^j_{t,\nu,\beta}$  for $j=1,\ldots,N$, and similar to
 Sections 3 and 5 we study
the  random normalized
 empirical measures
\begin{equation}\label{emp-measure-basic-b-dunkl-renorm}
\mu_{N,t,\nu,\beta}:= \frac{1}{N}(\delta_{\tilde X^1_{t,\nu,\beta}/\sqrt N}+\ldots +\delta_{\tilde X^N_{t,\nu,\beta}/\sqrt N})\in M^1(\mathbb R)
\end{equation}
of the processes  $(\tilde X_{t,\nu, \beta})_{t\ge0}$  for $N\to\infty$.
We claim that the measures $\mu_{N,t,\nu,\beta}$  converge
to the same limit as the normalized empirical measures of the expectations of the frozen Dunkl process of the previous section.
For this we study the moments
\begin{equation}\label{def-moments-case-b-dunkl-renorm}
  S_{N,l,\nu,\beta}(t):= \frac{1}{N^{l/2+1}}((\tilde X^1_{t,\nu,\beta})^l+\ldots+(\tilde X^N_{t,\nu,\beta})^l)   \quad\quad (l\ge0).\end{equation}
By the construction of the processes $(\tilde X_{t,\nu, \beta})_{t\ge0}$, the even moments
$ S_{N,2l,\nu,\beta}(t)$  are closely related to corresponding moments of the Bessel processes of type B,
and for them we can proceed as in Section 5.
The odd moments however are different due to the additional jump components.
We proceed as in Section 3 with a Lemma concerning the symmetric monomials $m_\lambda$ and refer to the notation there.

\begin{lemma}\label{general-limit-expectations-dunkl}
Let  $(x_N)_{N\ge1}:=(x_{N,n})_{N\ge1, 1\le n\le N}\subset\mathbb R$  be a family of starting numbers with $x_{N,n-1}\ge x_{N,n}$
for $2\le n\le N$, 
for which $$\lim_{N\to\infty} \frac{m_\lambda(x_N)}{N!\cdot N^{|\lambda|/2}}<\infty$$
exists for all  $\lambda\in \cal P$.

Let $\beta\in[1/2,\infty]$, $\nu>0$, and $(\tilde X_{t,\nu,\beta})_{t\geq 0}$ the renormalized Dunkl processes
with  start in $(x_{N,1}, \ldots,x_{N,n})$.
 Then, for all $\lambda\in \cal P$, the limits
$$\lim_{N\to\infty} \frac{E(m_\lambda(\tilde X_{t,\nu,\beta}))}{N!\cdot N^{|\lambda|/2}}$$
exist locally uniformly in $t$ and are independent from $\beta$.
\end{lemma}

\begin{proof} We prove this statement by induction on  $|\lambda|$.
For  $\lambda=0$ we have $m_\lambda(x)=N!$ and thus  the claim.
For $\lambda=(1,0,\ldots,0)$  we have  $m_\lambda(x)=(N-1)!\cdot(x_1+\ldots+ x_N)$.
Thus, by It\^{o}'s formula for Dunkl processes  in Corollary 3.6 of \cite{CGY}
\begin{equation}\label{dunkl-sum1}
    \sum_{i=1}^N\tilde X_{t,\nu,\beta}^i = \sum_{i=1}^N x_{N,i} + \frac{1}{\sqrt \beta}\sum_{i=1}^N B_t^i+\sqrt{2\nu}\sum_{i=1}^N M_t^i+2 \sum_{i=1}^N\sum_{j\neq i} M_t^{ij-},\end{equation}
where, by Eq.~(50) of \cite{CGY},
the jump components of the normalized Dunkl process $M^i$ and $M^{ij^-}$ associated to the different roots are given by
\begin{eqnarray*}
  M_t^i&=&\sum_{s\leq t}\frac{-\sqrt{2}\tilde X^i_{s-,\nu,\beta}}{\sqrt{\nu}}{\bf{1}}_{(-\tilde X^{i}_{s,\nu,\beta}\neq \tilde X^i_{s-,\nu,\beta})}+
  \int_0^t \frac{\sqrt{\nu}}{\sqrt{2}\tilde X^i_{s-,\nu,\beta}} ds\\
  M_t^{ij^-}&=&\sum_{s\leq t}-(\tilde X^i_{s-,\nu,\beta}+\tilde X^j_{s-,\nu,\beta}){\bf{1}}_{(-\tilde X^{i}_{s,\nu,\beta}\neq \tilde X^j_{s-,\nu,\beta})}+
  \int_0^t \frac{1}{\tilde X^i_{s-,\nu,\beta}+\tilde X^j_{s-,\nu,\beta}} ds.
\end{eqnarray*}
Notice that in  the RHS of (\ref{dunkl-sum1})  the additional sum $S:=2 \sum_{i=1}^N\sum_{j\neq i} M_t^{ij}$
with
$$M_t^{ij}=\sum_{s\leq t}-(\tilde X^i_{s-,\nu,\beta}-\tilde X^j_{s-,\nu,\beta}){\bf{1}}_{(\tilde X^{i}_{s,\nu,\beta}\neq \tilde X^j_{s-,\nu,\beta})}+
\int_0^t \frac{1}{\tilde X^i_{s-,\nu,\beta}-\tilde X^j_{s-,\nu,\beta}} ds$$
appears for which $S=0$ holds.
As the  $M_t^i,  M_t^{ij^-}$ are martingales by \cite{CGY}, the claim  follows for $| \lambda|=1$.

Now let  $\lambda\in \cal P$ with  $| \lambda|\ge 2$;
 assume that the statement is already shown for partitions with weight at most $| \lambda|-1$.
 It\^{o}'s formula  in Corollary 3.6 of \cite{CGY}
yields
\begin{equation}\label{SDE-Moments-dunkl-allg}
  m_\lambda(\tilde X_{t,\nu,\beta})=
  m_\lambda(x_{N})+\frac{1}{\sqrt{\beta}}\sum_{i=1}^N\int_0^t \frac{ dm_\lambda}{dx_i}(\tilde X_{s,\nu,\beta})\>  dB_s^i +{\cal M}_t
+ \int_0^t (\tilde{\cal L}_{\nu,\beta}m_\lambda)(\tilde X_{s,\nu,\beta})\>ds
\end{equation}
with
\begin{align}\label{complete-jump-martingale-dunkl}
 & {\cal M}_t:=\sqrt{\nu}\sum_{\pi\in S_N}\sum_{i=1}^N\int_0^t
  \frac{((\tilde X_{s-,\nu,\beta}^i)^{\lambda_{\pi(i)}}-(-\tilde X_{s-,\nu,\beta}^i)^{\lambda_{\pi(i)}})(\tilde X_{s-,\nu,\beta}^{\pi(\lambda)})_i}{\sqrt{2}
    \tilde X_{s-,\nu,\beta}^i}\>dM_s^i\\
  &+ \sum_{\pi\in S_N} \sum_{i,j:\> j\neq i}\int_0^t\frac{(\tilde X_{s-,\nu,\beta}^i)^{\lambda_{\pi(i)}}(\tilde X_{s-,\nu,\beta}^j)^{\lambda_{\pi(j)}}
    -(-\tilde X_{s-,\nu,\beta}^i)^{\lambda_{\pi(j)}}(-\tilde X_{s-,\nu,\beta}^j)^{\lambda_{\pi(i)}}}{\tilde X_{s-,\nu,\beta}^i+\tilde X_{s-,\nu,\beta}^j}\notag\\
&\quad\quad\quad\quad\quad\quad\quad\quad\quad\quad\quad\quad\cdot (\tilde X_{s-,\nu,\beta}^{\pi(\lambda)})_{i,j}\>  dM_s^{ij^-}\notag\\
&+ \sum_{\pi\in S_N} \sum_{i,j:\> j\neq i}\int_0^t\frac{(\tilde X_{s-,\nu,\beta}^i)^{\lambda_{\pi(i)}}(\tilde X_{s-,\nu,\beta}^j)^{\lambda_{\pi(j)}}
    -(\tilde X_{s-,\nu,\beta}^i)^{\lambda_{\pi(j)}}(\tilde X_{s-,\nu,\beta}^j)^{\lambda_{\pi(i)}}}{\tilde X_{s-,\nu,\beta}^i-\tilde X_{s-,\nu,\beta}^j}\notag\\
 &\quad\quad\quad\quad\quad\quad\quad\quad\quad\quad\quad\quad\cdot (\tilde X_{s-,\nu,\beta}^{\pi(\lambda)})_{i,j}\> dM_s^{ij}\notag
  \end{align}
where $(x^{\pi(\lambda)})_{i}$ and $(x^{\pi(\lambda)})_{i,j}$ denote the multivariate products as  in Section 3
where the factors involving $x_i$ or in addition  $x_j$ are omitted respectively.

The diffusion parts $\frac{1}{\sqrt{\beta}}\sum_{i=1}^N\int_0^t \frac{ dm_\lambda}{dx_i}(\tilde X_{s,\nu,\beta})\>  dB_s^i$
in (\ref{SDE-Moments-dunkl-allg}) are martingales by the same arguments as in the proof of Lemma \ref{general-limit-expectations-a},
taking into account that the sum of the squared components is again a one-dimensional Bessel process as all contributions
from the jump components vanish. This yields
\begin{equation*}
E\left(\sum_{i=1}^N\int_0^t\frac{ dm_\lambda}{dx_i}(\tilde X_{s,\nu,\beta}) \> dB_s^i\right)=0 \quad\quad ( t\ge0).
\end{equation*}
Moreover, the integrals w.r.t.~$(M_t^i)_t, (M_t^{ij^-})_t,  (M_t^{ij})_t $
 are also martingales and hence their expectations equal to zero.
This follows easily from  the representation of these martingales $(M_t^i)_t, (M_t^{ij^-})_t,  (M_t^{ij})_t $
as  compensated sums of jumps as on p.~125 of \cite{CGY}; for instance, for  $(M_t^i)_t$ we have
\begin{eqnarray*}
\lefteqn{\sum_{\pi\in S_N}\sum_{i=1}^N\int_0^t\frac{((\tilde X_{s-,\nu,\beta}^i)^{\lambda_{\pi(i)}}-(-\tilde 
X_{s-,\nu,\beta}^i)^{\lambda_{\pi(i)}})(\tilde X_{s-,\nu,\beta}^{\pi(\lambda)})_i}{\sqrt{2}\tilde 
X_{s-,\nu,\beta}^i}\sqrt{\nu}dM_s^i}&&\\
&=&-\sum_{\pi\in S_N}\sum_{i=1}^N\sum_{s\leq t}((\tilde X_{s-,\nu,\beta}^i)^{\lambda_{\pi(i)}}-(-\tilde 
X_{s-,\nu,\beta}^i)^{\lambda_{\pi(i)}})(\tilde X_{s-,\nu,\beta}^{\pi(\lambda)})_i{\bf{1}}_{(-\tilde X^{i}_{s,\nu,\beta}\neq \tilde 
X^i_{s-,\nu,\beta})}\\
&&+\nu \int_0^t\sum_{\pi\in S_N}\sum_{i=1}^N ((\tilde X_{s-,\nu,\beta}^i)^{\lambda_{\pi(i)}}-(-\tilde 
X_{s-,\nu,\beta}^i)^{\lambda_{\pi(i)}})(\tilde X_{s-,\nu,\beta}^{\pi(\lambda)})_i ds.
\end{eqnarray*}

We finally turn to the drift term of the RHS in in (\ref{SDE-Moments-dunkl-allg}).
We there observe that by the theory of Dunkl operators (see e.g.~\cite{DX}) $\tilde{\cal L}_{\nu,\beta}m_\lambda$ is a homogeneous
polynomial of the order $|\lambda|-2$. Moreover, by the definition of $\tilde{\cal L}_{\nu,\beta}$ in Section 6,
it can be easily checked that it also symmetric and that it has the form
$\frac{1}{2\beta} Q_\lambda + R_\lambda$
with 
$$Q_\lambda(x)=\sum_{\pi\in S_N} \sum_{i=1}^N
\lambda_{\pi(i)}(\lambda_{\pi(i)}-1)(x_i)^{\lambda_{\pi(i)}-2}( x^{\pi(\lambda)})_i$$
and with some symmetric  homogeneous polynomial  $R_\lambda$ of order $|\lambda|-2$ which only depends on $\nu$, but not on $\beta$.
The methods of the proof of (\ref{independence-reason})   show that
$Q_\lambda$ is a linear combination of the $m_{\tilde\lambda}$ with $|\tilde\lambda|=|\lambda|-2$ with coefficients independent of $N\ge L(\lambda)$.
Moreover, as in the  proof of Lemma \ref{general-limit-expectations-a},  $R_\lambda$
is a linear combination of the $m_{\tilde\lambda}$ with $|\tilde\lambda|=|\lambda|-2$ with coefficients
$c_{\tilde\lambda}$ such that the terms $c_{\tilde\lambda}/N$ converge to some limits for $N\to\infty$. 
As in the proof of Lemma \ref{general-limit-expectations-a}, these assertions together with the
induction assumption now lead to claim for $\lambda$.
\end{proof}

\begin{remark} The proof of Lemma \ref{general-limit-expectations-dunkl} shows that for fixed $\beta$ and $\lambda$, the limit in 
Lemma \ref{general-limit-expectations-dunkl} has order $O(1/N)$.
  \end{remark}

Lemma \ref{general-limit-expectations-dunkl} has the following application to the moments $S_{N,l}(t)$:
    
\begin{corollary}\label{convergence-in-expectation-dunkl}
    Let  $(x_{N,n})_{N\ge1, 1\le n\le N}\subset\mathbb R$  be starting numbers with $x_{N,n-1}\ge x_{N,n}$
for $2\le n\le N$,   for which  the convergence condition in Lemma \ref{recurrence-sl-dunkl-det} holds. 
Let $\beta\in[1/2,\infty]$, $\nu>0$ , and let for $N\ge2$,  $(\tilde X_{t,\nu, \beta})_{t\geq 0}$ 
 the renormalized Dunkl processes 
 starting in $(x_{N,1}, \ldots,x_{N,n})$. Then, for $l\in\mathbb N_0$ and $c_l(t)$ from Lemma  \ref{recurrence-sl-dunkl-det},
 $$E(S_{N,l}(t))\to c_l(t) \quad\quad\text{for}\quad N\to\infty$$
 \end{corollary}

\begin{proof} This follows from Lemma \ref{general-limit-expectations-dunkl}
  analogous to the proof of Corollary \ref{convergence-in-expectation-a}.
\end{proof}

 Corollary \ref{convergence-in-expectation-dunkl} can be extended to an a.s.~result:

\begin{theorem}\label{Dunkl-semicircle-mp-B}
Consider the Dunkl processes $(\tilde X_{t,\nu, \beta})_{t\geq 0}$  with $\beta\geq 1/2$, $\nu>0$ and 
with  starting sequences
$(x_{N,i})_{i\ge1}\subset \mathbb R$  as before
 such that for  $l\ge0$,
$$c_l(0):=\lim_{N\to\infty} S_{N,l,\nu,\beta}(0)=\lim_{n\to\infty}  \frac{1}{N^{l/2+1}}(x_{N,1}^{l}+\ldots+x_{N,N}^{l})<\infty$$
 exists.
Let  $\nu=\nu(N)$ with $\nu_0:=\lim_{N\to\infty}\nu(N)/N\ge0$.
Then,
for $l\in \mathbb N_0$,
$$c_l(t):=\lim_{N\to\infty} S_{N,l,\nu,\beta}(t)$$
exists a.s. locally uniformly in $t\in[0,\infty[$.
    Furthermore, the $c_l(t)$ satisfy the recurrence relation from Lemma \ref{recurrence-sl-dunkl-det}, i.e.,
     $c_0(t)=1$, $ c_1(t)=c_1(0)$, and for $l\ge1$,
\begin{align}
c_{2l}(t)&=c_{2l}(0)+ 2l \int_0^t\Bigl(\nu_0 c_{2l-2}(s) +\sum_{h=0}^{l-1} c_{2h}(s)c_{2l-2h-2}(s)\Bigr)\> ds,\notag\\
  c_{2l+1}(t)&=c_{2l+1}(0)+ \int_0^t\Bigl(2l\nu_0 c_{2l-1}(s)+4 \sum_{h=0}^{l-1}(l-h) c_{2h}(s)c_{2l-2h-1}(s)\Bigr)\> ds.\notag
\end{align}
\end{theorem}

\begin{proof}
  Again, by It\^{o}'s formula for Dunkl processes  in Corollary 3.6 of \cite{CGY},
  we obtain for $l\geq 1$
\begin{align}\label{Dunkl-SDE-Moments-B}
  \sum_{i=1}^N&(\tilde X_{t,\nu,\beta}^i)^{l}=\sum_{i=1}^N x_i^{l}+D_l(t) + \frac{l}{\sqrt{\beta}} \sum_{i=1}^N\int_0^t(\tilde X_{s,\nu,\beta}^i)^{l-1} dB_s^i
 \\
&+\sum_{i=1}^N\int_0^t\frac{(\tilde X_{s-,\nu,\beta}^i)^{l}-(-\tilde X_{s-,\nu,\beta}^i)^{l}}{\sqrt{2}\tilde X_{s-,\nu,\beta}^i}\sqrt{\nu}dM_s^i\notag\\
 &+\sum_{i=1}^N \sum_{j\neq i}\int_0^t\frac{(\tilde X_{s-,\nu,\beta}^i)^{l}-(-\tilde X_{s-,\nu,\beta}^j)^{l}+(\tilde X_{s-,\nu,\beta}^j)^{l}-(-\tilde X_{s-,\nu,\beta}^i)^{l}}{\tilde X_{s-,\nu,\beta}^i+\tilde X_{s-,\nu,\beta}^j}dM_s^{ij^-}\notag
\end{align}
with the drift term
\begin{align}
D_l(t)&:=
\int_0^t l \sum_{i=1}^N \sum_{j\neq i}
\frac{2(\tilde X_{s,\nu,\beta}^i)^{l}}{(\tilde X_{s,\nu,\beta}^i)^2-(\tilde X_{s,\nu,\beta}^j)^2} ds\notag\\
&+
\sum_{i=1}^N \int_0^t(l\nu+\frac{l(l-1)}{2\beta})(\tilde X_{s,\nu,\beta}^i)^{l-2} ds\notag\\
&+\frac{\nu}{2}\int_0^t\sum_{i=1}^N\frac{(-\tilde X_{s-,\nu,\beta}^i)^{l}-(\tilde X_{s-,\nu,\beta}^i)^{l}}{(\tilde X_{s-,\nu,\beta}^i)^2}ds\notag\\
&+\frac{1}{2}\int_0^t \sum_{i=1}^N \sum_{j\neq i}\frac{(-\tilde X_{s-,\nu,\beta}^i)^{l}+(-\tilde X_{s-,\nu,\beta}^j)^{l}-(\tilde X_{s-,\nu,\beta}^j)^{l}-(\tilde X_{s-,\nu,\beta}^i)^{l}}{(\tilde X_{s-,\nu,\beta}^i+\tilde X_{s-,\nu,\beta}^j)^2}ds.
\notag\end{align}
Please notice  that here the sums of the integrals w.r.t.~the $ M^{ij}$ are zero and thus omitted.
As  in the proof of Lemma \ref{general-limit-expectations-dunkl}, the integrals with respect to $B^i, M^i, M^{ij-}$ in the RHS of
(\ref{Dunkl-SDE-Moments-B}) are martingales. For simplicity we denote them by $A^{1,l},A^{2,l},A^{3,l}$ respectively. 
We also notice  that the covariations between the jump processes associated to different roots are zero by Eq.~(49) in \cite{CGY}.

As for the even moments of order $2l$ all terms associated with the jump component of the Dunkl process vanish,
 we are left with the terms of a Bessel process of type B, and the claim follows by the results of Section 5.

 Hence it remains to prove the claim for the odd moments. Here we proceed similar to the proof of Theorem \ref{semicircle-A}
 where we now
 apply the Burkholder-Davis-Gundy inequality with exponent four in order to get a sufficiently fast convergence of
 the bound leading to a.s. convergence in the end.
 In fact, (\ref{Dunkl-SDE-Moments-B}) together with the Markov inequality, Burkholder-Davis-Gundy inequality with exponent 4  and
the inequality $(a+b+c)^2\le3(a^2+b^2+c^2)$ for $a,b,c\in\mathbb R$
 show that
 for all $l\in \mathbb N_0$, $\epsilon>0$ and $T>0$ and some universal constant $c>0$,
\begin{align}\label{BDG-bound-dunkl}
\sup_{s\leq T}P\Bigl(&\Bigl|\frac{1}{N^{\frac{2l+1}{2}+1}}\Bigl(\sum_{i=1}^N(\tilde{X}_{s,\nu,\beta}^i)^{2l+1}-\sum_{i=1}^N x_{N,i}^{2l+1}  - D^{2l+1}_t \Bigr) \Bigr|>\epsilon\Bigr)\\
&\leq
 \frac{1}{\epsilon^4} E\Bigl(\sup_{s\leq T}\Bigl(\frac{1}{N^{\frac{2l+1}{2}+1}}\Bigl( A^{1,2l+1}+A^{2,2l+1}+A^{3,2l+1}\Bigr)
\Bigr)^4\Bigr)\notag\\
&\leq \frac{c}{\epsilon^4}\frac{1}{N^{4l+6}}E\Bigl(\Bigl(\sum_{i=1}^3[A^{i,2l+1},A^{i,2l+1}]_T\Bigr)^2\Bigr)\notag\\
&\leq 
\frac{3c}{\epsilon^4}\frac{1}{N^{2}}\sum_{i=1}^3E\Bigl(\frac{1}{N^{4l+4}}[A^{i,2l+1},A^{i,2l+1}]^2_T\Bigr)\notag\\
&=:\frac{3c}{\epsilon^4}\frac{1}{N^{2}}B_{N,T, 2l+1}\notag.
\end{align}
We next prove
\begin{align}\label{sum-bound-dunkl}
\sum_{N=1}^\infty \frac{1}{N^{2}}B_{N,T, 2l+1}<\infty.
\end{align}
For this, we  show that $\lim_{N\to\infty} B_{N,T,2l+1}<\infty$, which holds if
$$ \lim_{N\to\infty} E\Bigl(\frac{1}{N^{4l+4}}[A^{i,2l+1},A^{i,2l+1}]^2_T\Bigr)<\infty    \quad\quad \text{for}\quad i=1,2,3.$$
 We consider these expectations separately with the aid of Corollary \ref{convergence-in-expectation-dunkl}. 

 The Brownian martingale $A^{1,2l+1}$ can be handled as in the proof of Theorem \ref{semicircle-A}; in fact, H\"older's inequality and
 Corollary \ref{convergence-in-expectation-dunkl}  yield that
\begin{align}
E(\frac{1}{N^{4l+4}}[A^{1,2l+1},A^{1,2l+1}]^2_T)
&= E\Bigl(\frac{(2l+1)^2}{\beta N^{4l+4}}\Bigl(\sum_{i=1}^N\int_0^T(\tilde X_{s,\nu,\beta}^i)^{4l}ds\Bigr)^2\Bigr)\\
&\leq \frac{(2l+1)^2}{\beta N^{4l+4}}E\Bigl(N\sum_{i=1}^N\Bigl(\int_0^T(\tilde X_{s,\nu,\beta}^i)^{4l}ds\Bigr)^2\Bigl)\notag\\
&\leq \frac{T (2l+1)^2}{\beta N^{2}}E\Bigl(\sum_{i=1}^N\int_0^T\frac{(\tilde X_{s,\nu,\beta}^i)^{8l}}{N^{4l+1}}ds\Bigl)\notag\\
&=\frac{T (2l+1)^2}{\beta N^{2}} \int_0^T E( S_{N,4l,\nu,\beta}(s))\> ds \to 0
\notag\end{align}
 for $N\to\infty$.
 For  $A^{2,2l+1}$  we use Eq.~(48) from \cite{CGY},  H\"older's inequality, and  Corollary \ref{convergence-in-expectation-dunkl} again
 and conclude that
\begin{align}
E(\frac{1}{N^{4l+4}}[A^{2,2l+1},A^{2,2l+1}]^2_T)
&= E\Bigl(\frac{4\nu^2}{N^{4l+4}}\Bigl(\sum_{i=1}^N\int_0^T(\tilde X_{s-,\nu,\beta}^i)^{4l}d[M^i,M^i]_s\Bigr)^2\Bigl)\notag\\
&\leq \frac{4\nu^2}{N^{4l+4}} E\Bigl(N\sum_{i=1}^N(\int_0^T(\tilde X_{s-,\nu,\beta}^i)^{4l}ds)^2\Bigr)\notag\\
&\leq 4T\frac{\nu^2}{N^2}\sum_{i=1}^NE\Bigl(\int_0^T\frac{(\tilde X_{s-,\nu,\beta}^i)^{8l}}{N^{4l+1}}\> ds\Bigr)
\notag\end{align}
remains bounded for $N\to\infty$; notice here that $\nu/N$ tends to $\nu_0$.
Moreover, using the polynom division as in (\ref{ungerade-L}) and H\"older's inequality three times, we see
from Lemma \ref{general-limit-expectations-dunkl} that
\begin{align}
E(&\frac{1}{N^{4l+4}}[A^{3,2l+1},A^{3,2l+1}]^2_T)\notag\\
&= E\Bigl(\frac{4}{N^{4l+4}}\Bigl(\sum_{i,j:j\neq i}\int_0^T
\Bigl(\sum_{h=0}^{2l}(-1)^h(\tilde X_{s-,\nu,\beta}^i)^{h}(\tilde X_{s-,\nu,\beta}^j)^{2l-h}\Bigr)^2d[M^{ij^-},M^{ij^-}]_s\Bigr)^2\Bigl)\notag\\
&\leq \frac{4}{N^{4l+4}} E\Bigl(N(N-1)\sum_{i,j:j\neq i}
\Bigl(\int_0^T\Bigl(\sum_{h=0}^{2l}(-1)^h(\tilde X_{s-,\nu,\beta}^i)^{h}(\tilde X_{s-,\nu,\beta}^j)^{2l-h}\Bigr)^2ds\Bigr)^2\Bigr)\notag\\
&\leq \frac{4T}{N^{4l+4}} E\Bigl(N(N-1)\sum_{i,j:j\neq i}\int_0^T(\sum_{h=0}^{2l}(-1)^h(\tilde X_{s-,\nu,\beta}^i)^{h}(\tilde X_{s-,\nu,\beta}^j)^{2l-h})^4 ds\Bigr)\notag\\
&\leq 4T (2l+1)^3 E\Bigl(\frac{N(N-1)}{N^2}\sum_{i,j:j\neq i}\int_0^T\sum_{h=0}^{2l}\frac{(\tilde X_{s-,\nu,\beta}^i)^{4h}}{N^{2h+1}}\frac{(\tilde X_{s-,\nu,\beta}^j)^{8l-4h}}{N^{4l-2h+1}} ds\Bigr)\notag
\end{align}
also remains bounded for $N\to\infty$. This completes the proof of (\ref{sum-bound-dunkl}).

   Looking at the drift term $D^{2l+1}$ of the RHS of (\ref{Dunkl-SDE-Moments-B}),
we obtain the recurrence relation (\ref{recurrence-dunkl-ungerade}),
where $\nu$ is replaced by $\nu +\frac{l-1}{2\beta}$. The desired results now follow by the same arguments as
in the proof of Theorem \ref{semicircle-A} using the  results of  Lemma \ref{recurrence-sl-dunkl-det},
as well as Theorem \ref{free-convolution-b}.
\end{proof}

As in  Section 3, Theorem \ref{Dunkl-semicircle-mp-B} leads to the following final limit theorem.

\begin{theorem}\label{Dunkl-free-convolution-b-random}
Let $\mu\in M^1([0,\infty[)$ be a probability measure which satisfies
 the moment condition (\ref{strong-carleman}).
Let
 $(x_{N,n})_{N\ge1, 1\le n\le N}\subset[0,\infty[$ 
such that the empirical measures
\begin{equation}
\mu_{N,0}:= \frac{1}{N}(\delta_{x_{N, 1}/\sqrt{N}}+\ldots \delta_{x_{N, N}/\sqrt{N}})
\end{equation}
tend weakly to  $\mu$ for $N\to\infty$. Consider the normalized Dunkl processes
$(\tilde{X}_{t,\nu, \beta})_{t\ge0}$ of type B with start in $(x_{N,1},\ldots,x_{N,N}) $ for $N\ge2$.
Then, for $t\ge0$  and
$$\lim_{N\to\infty} \frac{\nu(N)}{N}=\nu_0\geq 0,$$
 the  
 empirical measures
 $$\mu_{N,t}:= \frac{1}{N}(\delta_{\frac{\tilde{X}^1_{t,\nu,\beta}}{\sqrt{ N}}}+\ldots +\delta_{\frac{\tilde{X}^N_{t,\nu, \beta}}{\sqrt{N}}}) $$
 tend weakly a.s.~to  the limiting measure whose Stieltjes transform satisfies the PDEs
 (\ref{pde-g-dunkl}) with the corresponding initial condition.
\end{theorem}


\begin{thebibliography}{999}

\bibitem[A]{A} N.I. Akhiezer, The Classical Moment Problems and Some Related Questions in Analysis.
 Engl. Translation, Hafner Publishing Co., New York, 1965. 

\bibitem[AGZ]{AGZ} G.W. Anderson, A. Guionnet, O. Zeitouni, 
An Introduction to Random Matrices. Cambridge University Press, 2010.

\bibitem[AHV]{AHV} S. Andraus, K. Hermann, M. Voit,
  Limit theorems and soft edge of freezing random matrix models via dual orthogonal polynomials. Preprint, arXiv:2009.01418.


\bibitem[AKM1]{AKM1} S. Andraus, M. Katori, S. Miyashita, Interacting particles on the line 
and Dunkl intertwining operator of type $A$: Application to the freezing regime. 
\textit{J. Phys. A: Math. Theor. } 45  (2012) 395201.

\bibitem[AKM2]{AKM2} S. Andraus, M. Katori, S. Miyashita, 
Two limiting regimes of interacting Bessel processes. 
\textit{J. Phys. A: Math. Theor. } 47  (2014) 235201.

\bibitem[AV1]{AV1} S. Andraus, M. Voit, Limit theorems
 for multivariate Bessel processes in the freezing regime. 
\textit{Stoch. Proc. Appl. } 129 (2019), 4771-4790.

\bibitem[AV2]{AV2} S. Andraus, M. Voit, Central limit theorems
 for multivariate Bessel processes in the freezing regime II: the covariance matrices of the limit.
 \textit{J. Approx. Theory}  246 (2019), 65-84.


 \bibitem[An]{An} J.-P. Anker,  An introduction to Dunkl theory and its analytic aspects. In: G. Filipuk,
Y. Haraoka, S. Michalik. Analytic, Algebraic and Geometric Aspects of Differential Equations,
Birkh{\"a}user, Cham (Switzerland), pp.3-58, 2017.


\bibitem[CG]{CG} T. Cabanal Duvillard, A. Guionnet, 
Large deviations upper bounds for the laws of matrix-valued processes and non-commutative entropies. 
\textit{Ann. Probab.} 29 (2001), 1205-1261.

\bibitem[CGY]{CGY} O. Chybiryakov, L. Gallardo, M. Yor, Dunkl processes and their
 radial parts relative to a root system. In:
 P. Graczyk et al. (eds.), Harmonic and stochastic analysis of Dunkl processes, pp. 113-198. Hermann, Paris 2008.

\bibitem[D]{D} P. Deift, Orthogonal Polynomials and Random Matrices: A Riemann-Hilbert Approach. Amer. Math. Soc. 2000.


\bibitem[DV]{DV} J.F. van Diejen, L. Vinet, Calogero-Sutherland-Moser Models.
 CRM Series in Mathematical Physics, Springer, Berlin, 2000.



\bibitem[DE1]{DE1} I. Dumitriu, A. Edelman, Matrix models for beta-ensembles. \textit{J. Math. Phys.} 43 (2002),  5830-5847.

\bibitem[DE2]{DE2} I. Dumitriu, A. Edelman, Eigenvalues of Hermite and Laguerre ensembles: large beta asymptotics,
  \textit{Ann. Inst. Henri Poincare (B)} 41 (2005), 1083-1099.

  \bibitem[DX]{DX} C.F. Dunkl, Y. Xu, Orthogonal Polynomials of Several Variables. Cambridge University Press, Cambridge, 2001.
  
\bibitem[GY]{GY} L. Gallardo, M. Yor, Some remarkable properties of the Dunkl martingale. In:
 Seminaire de Probabilites XXXIX, pp. 337-356, dedicated to P.A. Meyer, vol. 1874,
 Lecture Notes in Mathematics, Springer, Berlin, 2006.

\bibitem[G]{G} W. Gawronski, On the asymptotic distribution of the zeros of
 Hermite, Laguerre, and Jonquiere polynomials. 
 \textit{J. Approx. Theory} 50 (1987), 214--231.

\bibitem[GK]{GK} V. Gorin, V. Kleptsyn, Universal objects of the infinite beta random matrix theory.
  Preprint, arXiv:2009.02006.
  
\bibitem[GM]{GoM} V. Gorin, A.W. Marcus, Crystallization of random matrix orbits.
  \textit{Int. Math. Res. Notices}  2020(3),  883–913.

\bibitem[GrM]{GM} P.  Graczyk, J. Malecki, Strong solutions of non-colliding particle systems.
\textit{ Electron. J. Probab.} 19 (2014), 21 pp.

\bibitem[HT]{HT} U. Haagerup,  S. Thorbjornsen, 
Random matrices with complex gaussian entries,\textit{ Expo. Math.} 21 (2003), 293-337.

\bibitem[KM1]{KM} M. Kornyik, Gy. Michaletzky, 
Wigner matrices, the moments of Hermite polynomials and the semicircle law. 
\textit{J. Approx. Theory} 211 (2016), 29-41.
 
\bibitem[KM2]{KM2} M. Kornyik, Gy. Michaletzky, On the moments of roots of Laguerre-polynomials and the Marchenko-Pastur law. 
  \textit{Ann. Univ.  Sci.  Budapest., Sect.  Comp.}  46 (2017), 137--151.

\bibitem[KVW]{KVW}  M. Kornyik, M. Voit, J. Woerner,  Some martingales associated with
  multivariate Bessel  processes. \textit{Acta Math. Hung.}, to appear, arXiv:1908.11189. 
  

\bibitem[Me]{Me} M. Mehta, Random matrices (3rd ed.), Elsevier/Academic Press, Amsterdam, 2004. 

\bibitem[Men]{Men} G. Menon, Lesser known miracles of Burgers equation, \textit{ Acta  Math. Sci.} 32B (2012), 281-294.

\bibitem[NS]{NS} A. Nica, R. Speicher, Lectures on the Combinatorics of Free Probability Theory, Cambridge University Press, Cambridge, 2006.

\bibitem[OP]{OP} F. Oravecz, D. Petz, On the eigenvalue distribution of some symmetric random matrices,
\textit{ Acta Sci. Math.} 63 (1997), 383-395.

\bibitem[P]{P} P.E. Protter, Stochastic Integration and Differential Equations. A New Approach.
 Springer, Berlin, 2003.

\bibitem[RS]{RS} L.C.G. Rogers, Z. Shi, Interacting Brownian particles and the Wigner law. 
\textit{Probab. Theory Rel. Fields} 95 (1993), 555-570.

\bibitem[RW]{RW} L.C.G.~Rogers, D.~Williams, Diffusions, Markov Processes and Martingales, Vol. 1 Foundations.
Cambridge University Press, Cambridge, 2000.

\bibitem[R1]{R1} M. R\"osler,
Generalized Hermite polynomials and the heat equation for Dunkl operators.
\textit{Comm. Math. Phys.} 192 (1998),  519-542.

\bibitem[R2]{R2} M. R\"osler, Dunkl operators: Theory and applications.
In: Orthogonal polynomials and special functions, Leuven 2002, \textit{Lecture Notes in Math.} 1817 (2003), pp. 93--135. Springer Verlag, Berlin.



\bibitem[RV1]{RV1} M. R\"osler, M. Voit, Markov processes related with Dunkl operators.
\textit{Adv. Appl. Math.}  21 (1998), 575-643.

\bibitem[RV2]{RV2} M. R\"osler, M. Voit, Dunkl theory, convolution algebras, and related Markov processes.
 In: P. Graczyk et al. (eds.), Harmonic and stochastic analysis of Dunkl processes. pp. 1-112. Hermann, Paris 2008.

\bibitem[Sch]{Sch} B. Schapira,
The Heckman-Opdam Markov processes.
 \textit{Probab. Theory Rel. Fields} 138 (2007), 
 495-519.

 \bibitem[St]{St} W.A. Strauss, Partial Differential Equations: An Introduction. Wiley, 1992. 

\bibitem[Sz]{S} G. Szeg{\"o}, Orthogonal Polynomials. 
Colloquium Publications (American Mathematical Society), Providence, 1939.

\bibitem[V]{V} M. Voit,  Central limit theorems for multivariate Bessel processes in the freezing regime.
  \textit{ J. Approx. Theory } 239 (2019), 210--231.

 \bibitem[VW1]{VW} M. Voit, J.H.C. Woerner, Functional 
 central limit theorems for multivariate Bessel processes in the
 freezing regime. \textit{ Stoch. Anal. Appl.}, https://doi.org/10.1080/07362994.2020.1786402, arXiv:1901.08390. 

\bibitem[VW2]{VW2} M. Voit, J.H.C. Woerner, The differential equations associated with Calogero-Moser-Sutherland
  particle models in the freezing regime. Preprint 2019, arXiv:1910.07888.

\end{thebibliography}
\end{document}